\documentclass[11pt,reqno]{amsart}
\usepackage{fullpage,a4wide}

\usepackage[hidelinks]{hyperref}
\usepackage{amsmath,amsopn,amssymb,amsfonts}
\usepackage{amsthm}
\usepackage{mathtools}
\usepackage{enumitem}
\usepackage{multicol}
\usepackage{bm}

\newcommand{\N}{\mathbb N}
\newcommand{\Z}{\mathbb Z}
\newcommand{\Q}{\mathbb Q}
\newcommand{\R}{\mathbb R}
\newcommand{\C}{\mathbb C}
\renewcommand{\H}{\mathbb{H}}
\renewcommand{\Im}{{\rm Im}}
\newcommand{\sgn}{{\rm sgn}}
\newcommand{\SL}{\mathrm{SL}}
\newcommand{\sltwoz}{\mathrm{SL}_2(\Z)}
\newcommand{\abcd}{\left(\begin{smallmatrix} a & b \\ c & d \end{smallmatrix}\right)}
\renewcommand{\binom}[2]{\left(\begin{smallmatrix}#1\\\\#2\end{smallmatrix}\right)}
\renewcommand{\pmod}[1]{\ \left( \mathrm{mod} \, #1 \right)}
\newcommand{\Log}{\operatorname{Log}}
\def\bar{\overline}
\newcommand{\df}{\mathchoice{\: :=\: }{:=}{:=}{:=}}
\def\={\: =\: }
\renewcommand{\vec}{\bm}
\renewcommand{\epsilon}{\varepsilon}

\theoremstyle{plain}
\newtheorem{thm}{Theorem}[section]
\newtheorem{cor}[thm]{Corollary}
\newtheorem{lem}[thm]{Lemma}
\newtheorem{prop}[thm]{Proposition}
\newtheorem{conj}[thm]{Conjecture}

\theoremstyle{definition}
\newtheorem{example}[thm]{Example}
\newtheorem{defn}[thm]{Definition}
\newtheorem*{rem}{Remark}
\newtheorem*{rems}{Remarks}
\numberwithin{equation}{section}

\setlist{itemsep=3pt, topsep=2pt}
\allowdisplaybreaks

\title{Quasimodularity of $q$-series associated to root lattices}
\author{Antun Milas}
\address{Department of Mathematics, SUNY-Albany, Albany NY 12222, USA}
\email{amilas@albany.edu}

\author{Jan-Willem M. van Ittersum}
\address{Korteweg--de Vries Institute for Mathematics, University of Amsterdam, Postbus 94248, 1090 GE  Amsterdam, The Netherlands}
\email{j.w.m.vanittersum@uva.nl}

\subjclass[2020]{Primary 11F11, 11F50; Secondary 11M32, 17B22, 17B69}
\keywords{Quasimodular forms, quasi-Jacobi forms, $q$-zeta values, Lie algebras, root systems, lattice sums, hyperplane arrangements, vertex operator algebras}

\begin{document}
\begin{abstract}
Motivated by conjectures arising in part from the study of class-$\mathcal{S}$ vertex algebras associated with four-dimensional supersymmetric field theories, we study multiple $q$-zeta functions attached to finite-dimensional simple Lie algebras. We prove their conjectured quasimodularity for all classical Lie algebras of types~$A$, $B$, $C$, and~$D$, as well as a more general Weyl-symmetrized version in which the exponents attached to the positive roots are allowed to vary, under a suitable parity condition. 

Our proof is based on a substantially more general result. We introduce new combinatorial objects, called sets of level one and pairs of level two, which generalize the configurations arising from classical root systems. To these we associate multivariable lattice sums and prove that, under a natural admissibility condition, they are Fourier expansions of polar quasi-Jacobi forms of level one and level two, respectively.
Quasimodularity of the Lie-theoretic~$q$-series follows by specialization of the elliptic variables.

The positive roots of type~$A$ form a set of level one, and those of types~$B$, $C$, and~$D$ are described by pairs of level two. The proof combines a structural analysis of these new combinatorial configurations with reduction procedures and induction on the rank. The resulting framework applies to a much wider class of lattice sums than those coming from root systems and supports more general quasimodularity conjectures for~$q$-series associated with integral hyperplane arrangements.
\end{abstract}
\maketitle

\section{Introduction}
Multiple zeta values associated with root systems and Lie algebras form a natural meeting point of number theory, representation theory, and mathematical physics~\cite{Witten1991,Zagier1994}. Their~$q$-analogues have recently appeared in a different but closely related context, namely the theory of vertex operator algebras and the study of supersymmetric indices arising from four-dimensional quantum field theories. In this setting modularity and quasimodularity are expected to reflect underlying representation-theoretic structures~\cite{Arakawa,Arakawa-Kawasetsu}.

The starting point of this paper is a family of multiple~$q$-zeta functions associated with a finite-dimensional simple Lie algebra~$\mathfrak g$. Motivated by~\cite{Arakawa,Arakawa-Kawasetsu,Milas}, and especially the important work of Beem and Rastelli~\cite{BeemRastelli} (see also~\cite{BRS22}), we consider 
\begin{equation}\label{eq:Fzeta}
F_{\mathfrak g,\ell,k}(\boldsymbol{\zeta};q)\df \sum_{\lambda\in P_+}
{\rm ch}(V_\lambda)(\boldsymbol{\zeta})^\ell
\frac{
q^{\frac{k}{2}\sum_{\alpha\in\Delta_+}(\lambda+\rho,\alpha)}
}{
\prod_{\alpha\in\Delta_+}
\left(1-q^{(\lambda+\rho,\alpha)}\right)^{\!k}
}.
\end{equation}
Here~$\Delta_+$ is the set of positive roots, $P_+$ is the set of dominant integral weights, $\rho$ is the Weyl vector, and~$V_\lambda$ is the irreducible finite-dimensional~$\mathfrak g$-module of highest weight~$\lambda$, with character~${\rm ch}(V_\lambda)$ (see Section~\ref{subsec:function}).
By the Weyl dimension formula, after letting the elliptic variables tend to~$1$, one obtains the convergent~$q$-series
\begin{equation}\label{eq:Fq}
F_{\mathfrak g,\ell,k}(q)
\=
\sum_{\lambda\in P_+}
\dim(V_\lambda)^\ell
\frac{
q^{\frac{k}{2}\sum_{\alpha\in\Delta_+}(\lambda+\rho,\alpha)}
}{
\prod_{\alpha\in\Delta_+}
\left(1-q^{(\lambda+\rho,\alpha)}\right)^{\!k}
}.
\end{equation}
The conjecture motivating the present work is that under a suitable parity condition on~$k$ and~$\ell$, the function $F_{\mathfrak g,\ell,k}$ is quasimodular  (more precisely, it belongs to the ring of quasimodular forms) on a congruence subgroup of~$\SL_2(\mathbb Z)$. In particular, this conjecture implies quasimodularity of characters of genus zero class-$\mathcal{S}$ vertex algebras of Arakawa~\cite{Arakawa, BRS22, Milas}.

One of the principal applications of our results is a proof of this conjecture for all classical simple Lie algebras. More importantly, the Lie-theoretic statement turns out to be a consequence of a much more general phenomenon. The main purpose of this paper is to identify the combinatorial structure behind quasimodularity and to prove a general quasi-Jacobi theorem for a large class of lattice sums.

A first important observation is that, for~$k\equiv\ell\pmod 2$, Weyl-group symmetry allows us to replace the sum over the dominant weights in $F_{\mathfrak g,\ell,k}(q)$ by~$\frac1{|W|}$ times a sum over all weights off the root hyperplanes (see Theorem~\ref{thm.var}). This is the~$q$-analogue of Zagier's observation~\cite[Section~7, p.~507]{Zagier1994} that, for even arguments, Weyl symmetry allows one to replace the sum over one chamber in the Witten zeta function~$\zeta_{\mathfrak g}$ by the sum over the whole lattice minus the hyperplanes (see below).
Thus the basic objects are naturally lattice sums over
$$
\mathbb Z^n\setminus
\bigcup_{\bm v\in S}
\{\bm m:\bm v\cdot\bm m=0\},
$$
where~$S$ is a finite collection of integral vectors.
This observation leads us to introduce two new classes of combinatorial configurations, which we call \emph{sets of level one} (see Definition~\ref{def:level1}) and \emph{pairs of level two} (see Definition~\ref{def:level2}).  They contain some of the features of the classical root systems that are relevant for the analytic and modular properties of the associated lattice sums. 
However, our definitions are less restrictive than those of root systems  and, to our knowledge, have not yet appeared in the literature.

For a non-empty set~$S\subset\mathbb Z^n$ of level one and positive integers~$k_{\bm v}$ for~$\bm v\in S$, we consider
\begin{equation} \label{Psi-A}
\Psi_{S,\vec{k}}(\bm{\zeta};q)\df\!\!\!\!\sum_{\substack{
\bm m\in\mathbb Z^n\\
\bm v\cdot\bm m\neq0\ {\rm for}\ \bm v\in S
}}\!\!\!\!
\bm{\zeta}^{\bm m}
\prod_{\bm v\in S}
\left(
\frac{q^{\bm v\cdot\bm m}}
{\left(1-q^{\bm v\cdot\bm m}\right)^{\!2}}
\right)^{\!k_{\bm v}} \qquad \bigl(\bm{\zeta}^{\bm m} \df \prod_{i} \zeta_i^{m_i}\bigr).
\end{equation}
In order to prove our modularity results, it is useful to work in a more flexible algebraic setting. Writing $ X_{\bm v}(\bm m)\df X(\bm v\cdot\bm m)$ with $X(a)\df\frac{1+q^a}{1-q^a}$, we associate to a polynomial $p\in\mathbb C[X_{\bm v}:\bm v\in S]$ the (regular) lattice sum
$$
\Psi_{S,p}(\bm\zeta;q)\df \!\!\!\!\sum_{\substack{
\bm m\in\mathbb Z^n\\
\bm v\cdot\bm m\neq0\ {\rm for}\ \bm v\in S
}}\!\!\!\!
\bm\zeta^{\bm m}\,
p\bigl((X_{\bm v}(\bm m))_{\bm v\in S}\bigr).
$$
If~$p$ is admissible (see Definition~\ref{def:admissible}), then every coefficient of a power of~$q$ in this series is a Laurent polynomial in~$\bm\zeta$, so that~$\Psi_{S,p}$ is a well-defined~$q$-series (see Lemma~\ref{lem:admissible}). The original series~$\Psi_{S,\vec{k}}$ is obtained from an admissible polynomial~$p_{\vec{k}}$ (see Section~\ref{sec:admissible}).

Our first main theorem states that the resulting lattice sums are not merely quasimodular after specialization, but are themselves Fourier expansions of {\em polar quasi-Jacobi forms} (see Definition~\ref{def:quasijacobi}). 
\begin{thm} \label{thm:main1}
Let~$S$ be a set of level one and let $p\in\mathbb C[X_{\bm v}:\bm v\in S]$ be admissible. Then there is a constant $c_S>0$ such that $\Psi_{S,p}$ converges absolutely and locally uniformly for $|q|^{c_S}<|\zeta_j|<|q|^{-c_S}$ $(1\leq j\leq n)$ and is the Fourier expansion of a polar quasi-Jacobi form of level one, index~$0$ and of mixed weight at most the total degree of~$p$ that is holomorphic on this region.
\end{thm}
It follows that~$\Psi_{S,\vec{k}}$ is a polar quasi-Jacobi form of level one, index~$0$, and of mixed weight at most $2\sum_{\bm v\in S}k_{\bm v}$; specializing the elliptic variables gives quasimodular forms (Corollary~\ref{cor:polyinsertion}).

We need to work at level two in order to treat the root systems of types $B$, $C$, and $D$. For a pair~$(S,U)$ of level two, a coset $ \bm s\in\frac12\mathbb Z^n, $ and an admissible~$p\in R_{S,U}$, i.e., a polynomial in the variables~$X_{\bm v}$ for~$\bm v\in S$ in which negative powers are allowed exactly for~$\bm v\in U$ (see Section~\ref{subsec:laurent}), we define
$$
\Psi^{(\bm s)}_{S,p}(\bm\zeta;q)\df \!\!\!\!\sum_{\substack{
\bm m\in\mathbb Z^n+\bm s\\
\bm v\cdot\bm m\neq0\ {\rm for}\ \bm v\in S
}}\!\!\!\!
\bm\zeta^{\bm m}\,
p\bigl((X_{\bm v}(\bm m))_{\bm v\in S}\bigr).
$$
For~$p\in R_{S,U}$, let~$\|p\|$ be the maximum of $\sum_{\bm v}|e_{\bm v}|$ over the monomials $\prod_{\bm v}X_{\bm v}^{e_{\bm v}}$ occurring in~$p$ with non-zero coefficient. If no negative exponents occur, this is the total degree of~$p$. Our second main theorem is the level-two analogue of Theorem~\ref{thm:main1}, with polar quasi-Jacobi forms of level two as in Definition~\ref{def:quasijacobi2}.
\begin{thm} \label{thm:main2}
Let~$(S,U)$ be of level two, $\bm s\in\frac12\mathbb Z^n$, and let~$p\in R_{S,U}$ be admissible. Then there is a constant $c_S>0$ such that $\Psi^{(\bm s)}_{S,p}$ converges absolutely and locally uniformly for $|q|^{c_S}<|\zeta_j|<|q|^{-c_S}$ $(1\leq j\leq n)$ and is the Fourier expansion of a polar quasi-Jacobi form of level two, index~$0$ and of mixed weight at most~$\|p\|$ that is holomorphic on this region. If~$\bm s=\bm0$, then it is a form on~$\Gamma_0(2)$.
\end{thm}

The proofs of Theorems~\ref{thm:main1} and~\ref{thm:main2} are based on induction on the rank. A central issue is therefore to identify combinatorial classes that are stable under the reduction procedures appearing in the induction. This leads to a theory of level-one and level-two configurations. 

The classical root systems fit naturally into this framework. Type~$A$ is governed by the level-one theory, while types~$B$, $C$, and~$D$ are governed by the level-two theory. 
Extracting appropriate coefficients in the elliptic variables of the series of Theorems~\ref{thm:main1} and~\ref{thm:main2} therefore proves the quasimodularity conjecture for all classical simple Lie algebras. 

We also obtain a more general Weyl-symmetrized statement conjectured in~\cite{Milas}. Let $ \mathbf{k} = \{k_\alpha:\alpha\in\Delta_+\}, $ with~$k_{-\alpha}=k_\alpha$, and define
\begin{equation}\label{eq:FWq}
F^W_{\mathfrak g,\ell,\mathbf{k}}(q)\df \sum_{w\in W}
\sum_{\lambda\in P_+}
\dim(V_\lambda)^\ell
\frac{
q^{\frac12
\sum_{\alpha\in\Delta_+}
k_{w\alpha}(\lambda+\rho,\alpha)}
}{
\prod_{\alpha\in\Delta_+}
\left(1-q^{(\lambda+\rho,\alpha)}\right)^{\!k_{w\alpha}}
}.
\end{equation}
As another consequence of our general theory, we prove that if $ k_\alpha\equiv\ell\equiv 0 \pmod2 \ {\rm for\ all}\ \alpha\in\Delta_+, $ then $ F^W_{\mathfrak g,\ell,\mathbf{k}}$ is quasimodular, for a suitable congruence subgroup.
This generalization makes clear that the modular phenomenon is not tied to the case of a single uniform exponent~$k$. 
Thus the original vertex algebraic conjecture (with equal exponents) is only one manifestation of the structure uncovered here. 

Obviously, we would like to place our quasimodularity results in a wider framework which would also cover exceptional Lie types ($G_2$, $F_4$ and~$E_{6,7,8}$) and beyond. It seems that all that matters is the structure of the integral vector configuration indexing the linear forms $ \bm m\mapsto \bm v\cdot\bm m.$ From this perspective, our results suggest a much broader problem. Given a finite set $ S\subset\mathbb Z^n, $ one may consider lattice sums over the complement of the hyperplanes $ \bm v\cdot\bm m=0$ for $\bm v\in S$, with products of factors of the form $q^{\bm v\cdot\bm m}/(1-q^{\bm v\cdot\bm m})^2$ and their higher weight analogues. A general conjectural principle, communicated to us by D. Zagier, predicts that under appropriate convergence and parity assumptions such sums should be quasimodular. Theorems~\ref{thm:main1} and~\ref{thm:main2} establish this principle for two broad and natural classes of configurations. Moreover, the proofs indicate which parts of the argument are specific to levels one and two and which are likely to survive in a more general setting.

There are several points of contact with previous work. As we explained earlier, the original motivation comes from vertex algebra theory and from the appearance of the functions $F_{\mathfrak g,\ell,k}$ in problems related to four-dimensional~$\mathcal N=2$ superconformal field theories; see~\cite{Arakawa,Arakawa-Kawasetsu,BeemRastelli,BRS22,Milas,PP22}. Beyond the case $\mathfrak g=\mathfrak{sl}_2$ \cite{PP22}, those works do not yield transformations as Jacobi forms.
On the other hand, there is also an extensive literature on multivariable Jacobi and quasi-Jacobi forms in which quasimodular forms or mock modular forms are obtained by taking constant terms, contour integrals, or suitable Laurent coefficients in elliptic variables. This perspective appears, for example, in the work of Oberdieck and Pixton~\cite{OP} and of Goujard and M\"oller~\cite{GM}, and is related more broadly to the theory of meromorphic Jacobi forms in one variable developed in~\cite{DMZ}. Such methods have proved effective in a number of geometric and combinatorial problems and have also been successfully applied in low-rank Lie-theoretic examples such as for $\mathfrak{g}=\mathfrak{sl}(3)$ \cite{Milas}.
Our approach is different in an important respect. We start directly with lattice sums over complements of integral hyperplane arrangements and prove that these sums themselves represent Fourier expansions of polar quasi-Jacobi forms. The proof does not rely on first embedding the desired series into a previously known global Jacobi form. Instead, quasi-Jacobi transformation properties are produced recursively from the structure of the lattice sum. The new notions of level one and level two isolate the combinatorial conditions needed for this argument.

There is a clear connection between our work and multiple zeta values (MZVs) and their~$q$-analogs, the~$q$-MZVs.
As argued in~\cite{Milas}, for~$q\to1$ the functions $F_{\mathfrak g,\ell,k}$ reduce to values of the Witten zeta function \[\zeta_{\mathfrak g}(s)\df \sum_{\lambda\in P_+}\dim(V_\lambda)^{-s}.\]
More generally, $F^W_{\mathfrak g,\ell,\mathbf k}$ reduces in the same way to Weyl-symmetrized values of the zeta function of the root system, with one variable $s_\alpha=k_\alpha-\ell$ for every root. For even~$s$ one has $\zeta_{\mathfrak g}(s)\in\Q\pi^{s|\Delta_+|}$ \cite[Section~7]{Zagier1994}. Zagier's proof combines the passage to the whole lattice mentioned above with the fact that the sum of $1/(\ell_1(\bm x)\cdots\ell_N(\bm x))$ over all~$\bm x\in\Z^n$ off the hyperplanes~$\ell_i(\bm x)=0$ is a rational multiple of~$\pi^N$ if it converges absolutely (see also~\cite{KMT14}).  For odd~$s$ the sum over the whole lattice vanishes, and for $\mathfrak g=\mathfrak{sl}_3$ the sums over a cone lead to \emph{double zeta values} \cite[Section~8]{Zagier1994}
\[
\zeta(k_1,k_2)\=\sum_{n_1>n_2\geq1}n_1^{-k_1}n_2^{-k_2}\qquad(k_1\geq2, k_2\geq 1).
\]
These are the first instances of \emph{multiple zeta values} (MZVs)
\[\zeta(k_1,\dots,k_r)\=\sum_{n_1>\cdots>n_r\geq1}n_1^{-k_1}\cdots n_r^{-k_r} \qquad (k_1\geq2, k_2,\ldots,k_r\geq 1),\] whose~$\Q$-linear span is an algebra. In contrast to the single even values $\zeta(2m)\in\Q\pi^{2m}$, double and higher MZVs are in general not expected to be rational multiples of powers of~$\pi$.

Quasimodular forms enter this picture through the Eisenstein series, which are the~$q$-analogues of single zeta values. A series such as $\sum_{n\geq1} q^{nk}/(1-q^n)^{2k}$, where~$k\in\mathbb N$, is a~$q$-analogue of~$\zeta(2k)$ (after multiplying by $(1-q)^{2k}$), and such series generate the algebra of quasimodular forms.  Notice that these are precisely~$F_{\mathfrak{g},0,2k}$ for~$\mathfrak g=\mathfrak{sl}_2$. 
The~$q$-series arising from MZVs form a larger class: Bachmann introduced a~$\Q$-algebra of such~$q$-series that contains the quasimodular forms as a proper subalgebra~\cite[Definition~2.1 and Theorem~A]{Bac19} (see also~\cite{BvI24}), and the same holds for the~$q$-analogues of multiple zeta values introduced by Okounkov in the study of Hilbert schemes of points~\cite[Section~1.3]{Okounkov.arXiv2014}. 
For surveys, see~\cite[Section~14]{Zud03} and~\cite{Bac20}.

We stress that one can easily construct related~$q$-MZV-type series by modifying the definition of $F^W_{\mathfrak g,\ell,\mathbf k}(q)$: we keep the sum over the cone, allow arbitrary exponents~$k_\alpha$, and omit the symmetrization over the Weyl group (see~\cite{Milas}). Such series are not quasimodular in general. This is consistent with the behaviour of~$q$-MZVs themselves: symmetrizing a~$q$-multiple zeta series over the permutations of its arguments expresses it through depth-one~$q$-zeta values, that is, through Eisenstein-type series, and if all arguments are even this yields quasimodular forms (see, e.g., \cite[Proposition~2.26]{Bac26}). This is the~$q$-analogue of the classical fact that symmetric sums of MZVs are polynomials in single zeta values. This parallel raises many questions, such as those studied for~$q$-analogues of multiple zeta values: the relations among these~$q$-series, the structure of the algebra they generate (generators, derivations), and their (quasi)modularity. In this paper we answer the last question, which is the~$q$-analogue of Zagier's result for root systems.

The paper is organized as follows. In Section~\ref{sec:prelim}, we first review the necessary facts about  polar quasi-Jacobi forms. Then we analyze various summations over lattices associated to root systems and related results.  In Section~\ref{sec:levelone}, we develop the combinatorics of level-one sets and prove the reduction results required for induction on the rank. This leads to the proof of Theorem~\ref{thm:main1} and various corollaries, including quasimodularity of $F_{\mathfrak g,\ell,k}(q)$ for~$\mathfrak{g}$ of type~$A$ and even~$k$.
We next introduce level-two pairs in Section~\ref{sec:leveltwo} and study the additional phenomena caused by half-integral lattice cosets and level-two elliptic transformations. After establishing the corresponding structural and reduction results, we prove Theorem~\ref{thm:main2} and explain how to prove quasimodularity of $F_{\mathfrak g,\ell,k}(q)$ for the~$B$, $C$ and~$D$ root systems for even $k$ and also for~$A$ type for odd~$k$. We finish with two appendices: Appendix~\ref{sec:appendix}, where we classify certain low rank level one sets, and Appendix~\ref{sec:appendixA2}, where we illustrate the algorithm suggested by the proofs of Theorem~\ref{thm:main1} and Corollary~\ref{cor:polyinsertion} for~$S=\Delta_+(A_2)$.

\subsection*{Acknowledgments}
We thank K.~Bringmann for suggestions and comments on previous versions of this work.  We thank D.~Zagier for pointing out a possible generalization of our quasimodularity result mentioned in the Introduction.

The first author received partial support from the Collaboration Grant for Mathematicians 709563 awarded by the Simons Foundation and the NSF grant DMS-2101844. 
This work is funded by the EU NextGeneration under the Juraj Dobrila University of Pula institutional research project number IIP\_UNIPU\_010159.

The second author's research is part of the project \emph{Modular forms and multiple zeta values in mathematical physics} with file number VI.Veni.242.073 of the research programme Veni, which is financed by the Dutch Research Council (NWO) under the grant \href{https://doi.org/10.61686/OJGCY95356}{\nolinkurl{doi.org/10.61686/OJGCY95356}}.

\subsection*{Use of AI}
We used LLMs for some fine editing of the paper, and Claude Code to write code that verifies the statements of the paper by computer. All concepts, main results, and key ideas in this paper are the original work of the authors.

\section{Preliminaries}\label{sec:prelim}

\subsection{Quasimodular forms} Let~$\Gamma$ be a congruence subgroup of ${\rm SL_2}(\mathbb{Z})$.  
\begin{defn}
Let \( f : \mathbb{H} \to \mathbb{C} \) be a holomorphic function, and let \( k, s \in \N_0 \).  
The function \( f \) is called a \emph{quasimodular function of weight \( k \) and depth \( s \) }if there exist holomorphic functions $ f_0, f_1, \dots, f_s : \mathbb{H} \to \mathbb{C}, $ with \( f_s \) not identically zero, such that
\begin{equation}
(c\tau + d)^{-k} \, f\!\left( \frac{a\tau + b}{c\tau + d} \right)
\= \sum_{j=0}^{s} f_j(\tau) \left( \frac{c}{c\tau + d} \right)^{\!j}
\end{equation}
for any \(
\begin{psmallmatrix}
a & b \\ c & d
\end{psmallmatrix} \in \Gamma
\) and any \( \tau \in \mathbb{H} \), and if the~$f_j$ are holomorphic at the cusps, i.e., $(c\tau+d)^{-(k-2j)}f_j\bigl(\frac{a\tau+b}{c\tau+d}\bigr)$ is bounded as~$\Im(\tau)\to\infty$ for every~$\abcd\in\SL_2(\Z)$ and~$0\leq j\leq s$.
\end{defn}
It is well-known that for any congruence subgroup~$\Gamma$, the ring of quasimodular forms satisfies $ \widetilde{M}(\Gamma) \= M_*(\Gamma)[E_2], $ where~$M_*(\Gamma)$ denotes the graded ring of modular forms on~$\Gamma$, and $E_2(\tau)\df 1-24\sum_{n\ge 1} \sum_{d \mid n} d\mspace{1mu} q^n$, with $q\df e^{2\pi i\tau}$ throughout, is the weight two Eisenstein series (see, e.g., \cite[Proposition~1]{KZ95}). It satisfies
$$E_2(\tau)\=\frac{12}{\pi i}\frac{d}{d\tau} \Log \left(\eta(\tau)\right),$$
where $\eta(\tau)\df q^{\frac{1}{24}}\prod_{n\geq1} (1-q^n)$ is the \emph{Dedekind eta function}.\footnote{Throughout we use the principal branch of the logarithm.} 
For the full modular group, the classical Eisenstein series
\[
E_4(\tau) \= 1 + 240 \sum_{n\ge1} \sum_{d\mid n} d^3q^n,\qquad
E_6(\tau) \= 1 - 504\sum_{n\ge1} \sum_{d\mid n} d^5q^n,
\]
generate all modular forms, and thus the quasimodular forms are $ {\widetilde{M} = \mathbb{C}\left[E_2, E_4, E_6\right]}. $

\subsection{Quasi-Jacobi forms}\mbox{}
We introduce Jacobi forms of matrix index following~\cite{BS23,BKMZ}, and quasi-Jacobi forms as in~\cite{OP19,vI23}. Let~$k\in \mathbb{Z}$ and let~$M$ be a symmetric~$n\times n$ matrix such that~$2M$ is integral with even diagonal entries.\footnote{Throughout we write vectors in bold letters.} 
\begin{defn} \label{def:jacobi}
A meromorphic function~$\phi$ on $\C^n\times \mathbb{H}$ is called \emph{polar} of index~$M$ if there are finitely many hyperplanes
\[
\vec{s}\cdot \vec{z} \in r\tau+t+\Z\tau+\Z \qquad (\vec{s}\in \Z^n,\ r,t\in\Q)
\]
which contain all poles of $\vec{z}\mapsto\phi(\vec{z};\tau)$ for every~$\tau\in\H$, and if for all $\vec{\lambda},\vec{\mu}\in\R^n$ for which $\vec{\lambda}\tau+\vec{\mu}$ does not lie on these hyperplanes the function
\[
\tau\mapsto e\big(\vec{\lambda}^TM\vec{\lambda}\,\tau\big)\,\phi(\vec{\lambda}\tau+\vec{\mu};\tau)
\]
is holomorphic on~$\H$ and of moderate growth at every cusp, i.e., for all~$C>0$, $\gamma\in\sltwoz$ and~$x\in\R$ it is~$O(e^{Cy})$ at~$\gamma(x+iy)$ as~$y\to\infty$. A \emph{polar Jacobi form} of weight~$k$ and index~$M$ is a polar function~$\phi$ of index~$M$ on $\C^n\times \mathbb{H}$ satisfying
\begin{align}
\phi\left(\frac{\vec{z}}{c\tau+d},\frac{a\tau+b}{c\tau+d}\right) &\= (c\tau+d)^{k} e\!\left(\frac{c\vec{z}^TM\vec{z}}{c\tau+d}\right)\phi(\vec{z};\tau), \label{eq:transfo.1} \\
\phi(\vec{z}+\vec{\lambda}\tau+\vec{\mu};\tau) &\= e\!\left(-\vec{\lambda}^T(M\vec{\lambda}\tau+2M\vec{z})\right)\phi(\vec{z};\tau), \label{eq:transfo}
\end{align}
for all~$\abcd \in \sltwoz$ and $\vec{\lambda},\vec{\mu} \in \Z^n$. Here~$e(w)\df e^{2\pi i w}$. 
\end{defn}

As our main example we consider the meromorphic Jacobi form of weight~$1$ and index $M=\frac12\left(\begin{smallmatrix}0&1\\1&0\end{smallmatrix}\right)$, i.e., $\vec z^TM\vec z=z_1z_2$, studied by Zagier in the context of periods of modular forms~\cite[Section~3]{Za}
\[
F(z_1,z_2;\tau) \= -\frac{2\pi\eta^3(\tau)\vartheta(z_1+z_2;\tau)}{\vartheta(z_1;\tau)\vartheta(z_2;\tau)},
\]
expressed in terms of the \emph{Jacobi theta function}
\[\vartheta(z;\tau)\df \sum_{n\in \Z+\frac{1}{2}} q^{\frac{n^2}{2}} e^{2\pi in\left(z+\frac{1}{2}\right)}.\]
Note that $[\frac{d}{dz}\vartheta(z;\tau)]_{z=0}=-2\pi\eta^3(\tau)$. Writing $\zeta_j\df e^{2\pi i z_j}$ throughout, $F$ can be written as~\cite[Equation~(14)]{Za}
	\begin{equation}\label{eq:FZagier}
			\frac{1}{2\pi i}F(z_{1},z_{2};\tau)
			\= \frac{1}{2}\left(\coth\left(\pi i z_1\right)+\coth\left(\pi i z_2\right)\right) - \sum_{n,m\geq 1} \left(\zeta_1^m\zeta_2^n-\zeta_1^{-m}\zeta_2^{-n}\right)q^{nm}
		\end{equation}
for $|\Im z_1|,|\Im z_2|<\Im\tau$. This identity goes back to Kronecker~\cite{Kro}; see also~\cite[pp.~70--71]{Wei76}.

Various functions of interest in this paper are not Jacobi forms, but lie in an extension, defined as follows.

\begin{defn}\label{def:quasijacobi}
A \emph{polar quasi-Jacobi form} of weight~$k$ and index~$M$ is a polar function~$\phi$ of index~$M$ for which there exist finitely many polar functions~$\phi_{i,\vec j}$ of index~$M$, indexed by $(i,\vec j)\in\Z_{\geq0}\times\Z_{\geq0}^n$, with~$\phi_{0,\vec 0}=\phi$, such that
\begin{align}
\phi\left(\frac{\vec{z}}{c\tau+d},\frac{a\tau+b}{c\tau+d}\right) &\= (c\tau+d)^{k} e\!\left(\frac{c\vec{z}^TM\vec{z}}{c\tau+d}\right)\sum_{i,\vec j}\phi_{i,\vec j}(\vec z;\tau)\left(\frac{c}{c\tau+d}\right)^{\!i+|\vec j|}\vec z^{\,\vec j}, \label{eq:quasitransfo.1} \\
\phi(\vec{z}+\vec{\lambda}\tau+\vec{\mu};\tau) &\= e\!\left(-\vec{\lambda}^T(M\vec{\lambda}\tau+2M\vec{z})\right)\sum_{\vec j}\phi_{0,\vec j}(\vec z;\tau)\,(-\vec\lambda)^{\vec j}, \label{eq:quasitransfo}
\end{align}
for all~$\abcd \in \sltwoz$ and $\vec{\lambda},\vec{\mu} \in \Z^n$, where $\vec z^{\,\vec j}\df z_1^{j_1}\cdots z_n^{j_n}$ and $|\vec j|\df j_1+\cdots+j_n$. A finite sum of polar quasi-Jacobi forms of different weights is called a polar quasi-Jacobi form of \emph{mixed weight}. Since \eqref{eq:quasitransfo.1} is required for all $\abcd\in\sltwoz$, we also say that these forms are of \emph{level one}.\footnote{Unless stated otherwise, we use \emph{quasi-Jacobi form} to mean \emph{polar quasi-Jacobi form}.}\textsuperscript{,}\footnote{Up to the normalization of the~$\phi_{i,\vec j}$, this is the characterization of strictly meromorphic quasi-Jacobi forms in \cite[Section~2.4]{vI23}. That the poles lie on finitely many rational hyperplanes is \cite[Theorem~1.6]{vI23}.} 
\end{defn}
The polar Jacobi forms are the polar quasi-Jacobi forms with~$\phi_{i,\vec j}=0$ for $(i,\vec j)\neq(0,\vec 0)$. Quasimodular forms are the polar quasi-Jacobi forms with the number~$n$ of elliptic variables equal to zero.

Two examples of polar quasi-Jacobi forms are $A(z;\tau)\df \frac{d}{dz} \Log \left(\vartheta(z;\tau)\right)$ and the Eisenstein series~$E_2$. We also use the normalization
\[
\mathcal A(z;\tau)\df \zeta\frac{d}{d\zeta}\Log\bigl(\vartheta(z;\tau)\bigr)\=\frac{1}{2\pi i}A(z;\tau).
\]
Its Fourier expansion, obtained by logarithmic differentiation of the Jacobi triple product for~$\vartheta$, is
\begin{equation}\label{eq:Aproduct}
\mathcal A(z;\tau)\=-\frac{1+\zeta}{2(1-\zeta)}+\sum_{r\geq1}\left(\frac{\zeta^{-1}q^r}{1-\zeta^{-1}q^r}-\frac{\zeta q^r}{1-\zeta q^r}\right).
\end{equation}
Note that the constant term $-\frac{1+\zeta}{2(1-\zeta)} = \frac12\coth(\pi i z)$ also occurs in $F$.
We record the following properties of a polar quasi-Jacobi form~$\phi$ of weight~$k$ and index~$M$: \vspace{3pt}
\begin{itemize}[leftmargin=*]\itemsep3pt
\item The function~$\phi$ can be written as a polynomial in~$E_2$ and $\mathcal A(z_1),\ldots,\mathcal A(z_n)$ with polar Jacobi forms of index~$M$ as coefficients.\footnote{Here and below we often omit the variable~$\tau$, or~$q$, from the notation.} Then, the coefficient of $E_2^{\,i}\mathcal A(z_1)^{j_1}\cdots\mathcal A(z_n)^{j_n}$ has weight $k-2i-|{\bf j}|$ \cite[Proposition~2.18]{vI23}.
\item The derivatives of~$\phi$ with respect to~$\tau$ and~$z_j$ are polar quasi-Jacobi forms of index~$M$ and of weight~$k+2$ and~$k+1$, respectively \cite[Section~2.4]{vI23}; see also \cite[Section~1.3.3]{OP19}.
\item Products of polar quasi-Jacobi forms are polar quasi-Jacobi forms, with weights and indices adding up. If~$L$ is an integral $n\times m$ matrix and $\phi$ is not identically singular on the image of~$L$, then $(\bm z,\tau)\mapsto\phi(L\bm z;\tau)$ is a polar quasi-Jacobi form on $\C^m\times\H$ of weight~$k$ and index~$L^TML$. Both statements follow directly from \eqref{eq:quasitransfo.1} and \eqref{eq:quasitransfo}.
\item The coefficient of~$z_n^{\ell}$ in the Laurent expansion of~$\phi$ around $z_n=0$ is a polar quasi-Jacobi form on $\C^{n-1}\times\H$ of weight~$k+\ell$, whose index is the upper left $(n-1)\times(n-1)$ block of~$M$; this follows by comparing the coefficients of~$z_n^{\ell}$ on both sides of \eqref{eq:quasitransfo.1} and \eqref{eq:quasitransfo}. Iterating, if $\phi$ admits a Laurent expansion around $\vec z=\vec 0$, the coefficient of~$\vec z^{\,\vec\ell}$ in it is a quasimodular form of weight $k+|\vec\ell|$ for $\SL_2(\Z)$, holomorphic at the cusp by the growth condition in Definition~\ref{def:jacobi} \cite[Theorem~2.33]{vI23}. For index $M=0$ the same holds for the expansion around every $\vec z\in(\Z\tau+\Z)^n$, by \eqref{eq:quasitransfo}.
\end{itemize}\vspace{3pt}
In particular, if~$\phi$ has index~$0$ and is holomorphic at $\vec z=\vec 0$, then $\phi(\vec 0;\tau)$ is a quasimodular form of weight~$k$ (cf.\ \cite[Section~1.3.5]{OP19}).

In Section~\ref{sec:leveltwo} we encounter functions which satisfy the Jacobi transformation laws only with respect to the congruence subgroup $\Gamma(2)$, and whose elliptic transformation law holds only up to a character. We therefore use the following extension of Definition~\ref{def:quasijacobi}.

\begin{defn}\label{def:quasijacobi2}
Let~$\phi$ be a polar function of index~$M$ for which there exist a character $\psi:\Z^n\times\Z^n\to\{\pm1\}$ and finitely many polar functions~$\phi_{i,\vec j}$ of index~$M$ with~$\phi_{0,\vec 0}=\phi$, such that~\eqref{eq:quasitransfo.1} holds for all~$\abcd\in\Gamma(2)$ and
\begin{equation}\label{eq:transfo-level2-ell}
\phi(\vec{z}+\vec{\lambda}\tau+\vec{\mu};\tau) \= \psi(\vec\lambda,\vec\mu)\,e\!\left(-\vec{\lambda}^T(M\vec{\lambda}\tau+2M\vec{z})\right)\sum_{\vec j}\phi_{0,\vec j}(\vec z;\tau)\,(-\vec\lambda)^{\vec j}
\end{equation}
for all $\vec\lambda,\vec\mu\in\Z^n$. A \emph{polar quasi-Jacobi form of level two} of weight~$k$ and index~$M$ is a finite sum of such functions, possibly with different characters~$\psi$. If~$n=0$, we speak of \emph{quasimodular forms of level two}.
\end{defn}

We use the following properties of a polar quasi-Jacobi form~$\phi$ of level two, of weight~$k$ and index~$M$, which are proven as for level one: \vspace{3pt}
\begin{itemize}[leftmargin=*]\itemsep3pt
\item The derivatives of~$\phi$ with respect to~$\tau$ and~$z_j$ are polar quasi-Jacobi forms of level two and index~$M$, of weight~$k+2$ and~$k+1$, respectively. Products of polar quasi-Jacobi forms of level two are again of level two, with weights and indices adding up. If~$L$ is an integral $n\times m$ matrix and $\phi$ is not identically singular on the image of~$L$, then $(\bm z,\tau)\mapsto\phi(L\bm z;\tau)$ is a polar quasi-Jacobi form of level two on $\C^m\times\H$ of weight~$k$ and index~$L^TML$.
\item If~$\phi$ is holomorphic at $\vec z=\vec 0$, the coefficient of~$\vec z^{\,\vec\ell}$ in its Taylor expansion around $\vec z=\vec 0$ is a quasimodular form of weight $k+|\vec\ell|$ on~$\Gamma(2)$.
\item We say that~$\phi$ is a form \emph{on~$\Gamma_0(2)$} if~\eqref{eq:quasitransfo.1} holds for all~$\abcd\in\Gamma_0(2)$. If the Fourier expansion of~$\phi$ is integral in~$q$ and in~$\vec{\zeta}$, then~$\phi$ is a form on~$\Gamma_0(2)$. Namely, $\Gamma(2)$ and $\left(\begin{smallmatrix}1&1\\0&1\end{smallmatrix}\right)$ generate~$\Gamma_0(2)$, and in that case~$\phi$ and the functions~$\phi_{i,\vec j}$, which are uniquely determined by~$\phi$, are invariant under~$\tau\mapsto\tau+1$.
\end{itemize}\vspace{3pt}

\subsection{Root lattices and the Weyl group} 
This section covers standard material, with full details available in~\cite{Carter}. We denote a root system by $\Phi\df \Delta_+ \cup \Delta_-$, where~$\Delta_+$ is the set of positive roots and~$\Delta_-=-\Delta_+$. The classical root systems can be described using orthonormal vectors $\{\varepsilon_j\}_{j=1}^n$ in $(\mathbb{R}^n, (\cdot,\cdot))$:
\begin{itemize}[leftmargin=*]\itemsep3pt
\item For type~$A_{n-1}$, the root system is $\Phi = \{\varepsilon_a - \varepsilon_b : 1 \leq a \ne b \leq n\}$, lying in the hyperplane $\sum_{a=1}^n \varepsilon_a = 0$. The simple roots are $\alpha_a = \varepsilon_a - \varepsilon_{a+1}$ for $a \in\{ 1, \dots, n-1\}$.

\item For type~$B_n$, the root system is $\Phi = \{\pm\varepsilon_a : 1 \leq a \leq n\} \cup \{ \pm\varepsilon_a \pm \varepsilon_b : 1 \leq a < b \leq n\}$. The simple roots are $\alpha_1 = \varepsilon_1$ and $\alpha_a = \varepsilon_a - \varepsilon_{a-1}$ for $a \in\{ 2, \dots, n\}$.

\item For type~$C_n$, the root system is $\Phi = \{\pm 2\varepsilon_a : 1 \leq a \leq n \} \cup \{\pm\varepsilon_a \pm \varepsilon_b : 1 \leq a < b \leq n\}$. The simple roots are $\alpha_1 = 2\varepsilon_1$ and $\alpha_a = \varepsilon_a - \varepsilon_{a-1}$ for $a \in\{ 2, \dots, n\}$.

\item For type~$D_n$, the root system is $\Phi = \{\pm\varepsilon_a \pm \varepsilon_b : 1 \leq a < b \leq n\}$. The simple roots are $\alpha_1 = \varepsilon_1 + \varepsilon_2$ and $\alpha_a = \varepsilon_a - \varepsilon_{a-1}$ for $a \in\{ 2, \dots, n\}$.

\end{itemize}

The root lattice is then defined as $Q=\bigoplus_{j=1}^n \mathbb{Z} \alpha_j$ and clearly $\mathbb{R}^n \cong Q \otimes_{\mathbb{Z}} \mathbb{R}$. Denote by $P=\bigoplus_{j=1}^n \mathbb{Z} \omega_j$ the weight lattice where $(\omega_a, \alpha_b)=\delta_{a,b}$ and by $P_+=\bigoplus_{j=1}^n \N_0 \omega_j$ the set of dominant weights.\footnote{Our normalization of the~$\omega_a$ (and hence our definition of~$P$) is slightly non-standard, in order to avoid the use of coroots (see~\cite[Definition~10.16]{Carter} for the usual normalization of the fundamental weights).}
The Weyl group \( W\df \langle s_{\alpha_j}: 1 \leq j \leq n \rangle  \) is generated by simple reflections~$s_{\alpha_j}$:
$$s_{\alpha_j} : x \mapsto x-2 \frac{(x,\alpha_j)}{(\alpha_j,\alpha_j)} \alpha_j,$$
where the reflecting hyperplane is given by $H_{\alpha_j}=\{ v : (v,\alpha_j)=0 \}$.
Clearly,  for every \( u, v  \in \mathbb{R}^n \) and~$w \in W$, we have~$(w u, w v) = (u, v).$

The \emph{length} of~$w\in W$, denoted by~$\ell(w)$, is the minimal integer~$m$ such that~$w$ can be expressed as a product of~$m$ simple reflections. Let
$$
C^0\df \bigoplus_{j=1}^n \mathbb{R}^{+}\omega_j
$$
denote the interior of the fundamental Weyl chamber. Its intersection with the weight lattice is
$$
P\cap C^0
\=
\mathbb{N}\omega_1+\cdots+\mathbb{N}\omega_n
\=
\rho+P_+,
$$
where 
\begin{equation}\label{eq:rho}
\rho\df \sum_{j=1}^n\omega_j\=\frac12\sum_{\alpha\in\Delta_+}\alpha^\vee, \qquad \alpha^\vee\df \frac{2\alpha}{(\alpha,\alpha)}.
\end{equation}

For~$\bm v\in\R^n$, let 
$$
H_{\bm v}\df \{\bm x \in\R^n:(\bm x,\bm v)=0\}, \ \ \mathcal H\df \bigcup_{\alpha\in\Delta_+}H_\alpha,
$$
so that~$H_\alpha$ for~$\alpha\in\Phi$ is the reflecting hyperplane associated with~$\alpha$ and~$\mathcal H$ is the union of the reflecting hyperplanes. The complement of the reflecting hyperplanes is a disjoint union of open Weyl chambers, and the Weyl group acts simply transitively on the set of such chambers. More precisely (see, for example, \cite[Section~12.3]{Carter})
$$
\mathbb{R}^n\setminus
\mathcal H
\=
\bigsqcup_{w\in W}C^w,
\qquad
C^w\df w\left(C^0\right).
$$

Recall also that~$W$ permutes the root system~$\Phi$. For~$w\in W$, its \emph{inversion set} is
$$
N(w)\df \{\alpha\in\Delta_+:w(\alpha)\in\Delta_-\}
\=
\Delta_+\cap w^{-1}(\Delta_-).
$$
A standard relation between the length function and the action of~$W$ on the roots is
$$
|N(w)|\=\ell(w),
$$
see, for example, \cite[Corollary~5.16]{Carter}. Since~$w$ maps~$N(w)$ bijectively onto $w(\Delta_+)\cap\Delta_-$, this also gives:

\begin{lem}\label{neg-roots}
For every~$w\in W$, we have
$$
\bigl|w(\Delta_+)\cap\Delta_-\bigr|\=\ell(w).
$$
\end{lem}
In the sections below we work with the coordinate representations of vectors. For the root systems~$B_n$, $C_n$, and~$D_n$ we use the orthonormal basis $\varepsilon_1,\ldots,\varepsilon_n$ above and write
\begin{equation}\label{eq:epscoord}
\bm{v}\=\sum_{i=1}^n v_i\varepsilon_i\=(v_1,\ldots,v_n),
\qquad
\bm{m}\=\bigl((\lambda,\varepsilon_1),\ldots,(\lambda,\varepsilon_n)\bigr)
\end{equation}
for a vector~$\bm{v}\in\R^n$ and a weight~$\lambda\in P$, respectively. Since the basis is orthonormal, the pairing of a weight with a vector is the standard inner product,
\begin{equation}\label{eq:epspairing}
(\lambda,\bm{v})\=\bm{v}\cdot\bm{m}\=\sum_{i=1}^n v_im_i.
\end{equation}
Hence the conditions $(\lambda,\alpha)\neq0$ cutting out the complement of the root hyperplanes, as well as all exponents of~$q$ in the series we study, are of the form~$\bm{v}\cdot\bm{m}$ with~$\bm{v}$ a positive root.

Note that the coordinates~$\bm m$ of a weight need not be integers. Namely, as $(\omega_a,\alpha_b)=\delta_{a,b}$, the weight lattice~$P$ is the dual of the root lattice~$Q$. For~$B_n$ we have~$Q=\Z^n$ and therefore~$P=\Z^n$, whereas for~$C_n$ and~$D_n$ the root lattice is $\{\bm{v}\in\Z^n:v_1+\cdots+v_n\in2\Z\}$ and therefore
\begin{equation}\label{eq:weightlattice}
P\=\Z^n\cup\left(\Z^n+\tfrac12\bm{1}\right),
\qquad
\bm{1}\df (1,\ldots,1).
\end{equation}
In the latter case a sum over~$P$ is the sum of two sums over cosets of~$\Z^n$.

For the root systems of type~$A$ the roots lie in the hyperplane $x_1+x_2+\cdots=0$, so the vectors~$\varepsilon_j$ do not give coordinates on the space spanned by the roots. Instead, we use the coordinates with respect to the simple roots. That is, for type~$A_n$ we write a root $\alpha=\sum_{a=1}^nv_a\alpha_a$ as $\bm v=(v_1,\ldots,v_n)$ and a weight~$\lambda$ as $\bm m=\big((\lambda,\alpha_1),\ldots,(\lambda,\alpha_n)\big)$. As $(\omega_a,\alpha_b)=\delta_{a,b}$, we then have $P=\Z^n$ and again $(\lambda,\alpha)=\bm v\cdot\bm m$. The positive roots are the vectors in~$\{0,1\}^n\setminus\{\bm 0\}$ whose entries equal to~$1$ have consecutive indices. 

Finally, the constructions below single out the last non-zero coordinate of a vector and hence depend on the ordering of the coordinates. For~$B_n$, $C_n$, and~$D_n$ we use the ordering in~\eqref{eq:epscoord}; with the simple roots above, the positive roots $\varepsilon_j\pm\varepsilon_i$ $(i<j)$ and~$\varepsilon_j$ of~$B_n$ then all have last non-zero coordinate equal to~$1$, whereas the long roots~$2\varepsilon_j$ of~$C_n$ have last non-zero coordinate equal to~$2$.

\subsection{The function of interest}\label{subsec:function}
Let  $\mathfrak{g}$ a simple Lie algebra with a root system~$\Phi$. For~$\lambda \in P$, the Weyl character $\mathrm{ch}(V_{\lambda})$ is expressed as a summation over the Weyl group~$W$:
$${\rm ch}(V_\lambda)(\boldsymbol{\zeta})\=\frac{\sum_{u \in W} (-1)^{\ell(u)} \boldsymbol{\zeta}^{u(\lambda + \rho)-\rho}}{\prod_{\alpha \in \Delta_+}\left(1-\boldsymbol{\zeta}^{-\alpha^\vee}\right)},$$
where for~$\mu \in P$, we write ${\boldsymbol \zeta}^{\mu}$ to denote $\boldsymbol{\zeta}^{\mu}=\prod_{j=1}^n \zeta^{(\mu,\alpha_j)}_j$, and $\rho$ is the Weyl vector~\eqref{eq:rho}. Note that here we allow~$\lambda \in P$ not just~$P_+$.  
\begin{lem} \label{Weyl} For~$w \in W$, we have
\begin{equation*}%
 \mathrm{ch}\!\left(V_{w(\lambda+\rho)-\rho}\right)(\boldsymbol{\zeta})\=(-1)^{\ell(w)}  \mathrm{ch}(V_{\lambda})(\boldsymbol{\zeta}).
 \end{equation*}
\end{lem}
\begin{proof}
The claim follows by substituting~$u\mapsto uw^{-1}$ in the sum over~$u\in W$ in the numerator of the left-hand side, since $(-1)^{\ell(uw^{-1})}=(-1)^{\ell(u)}(-1)^{\ell(w)}$.
\end{proof}

The Weyl dimension formula gives the dimension of an irreducible representation~$V_{\lambda}$ of a semisimple Lie algebra with highest weight~$\lambda \in P_+$:
\begin{equation} \label{weyl-dim}
\dim (V_{\lambda}) \= 
\lim_{\zeta_j \to 1} {\rm ch} (V_\lambda)(\boldsymbol{\zeta})\=\prod_{\alpha \in \Delta_+ } \frac{( \lambda + \rho, \alpha )}{( \rho, \alpha )}.
\end{equation}
Recall the series $F_{\mathfrak{g},\ell,k}(\boldsymbol{\zeta};q)$ of~\eqref{eq:Fzeta}. By the Weyl dimension formula~\eqref{weyl-dim}, its limit as $\zeta_j\to1$ is the convergent $q$-series $F_{\mathfrak{g},\ell,k}(q)$ of~\eqref{eq:Fq}.
Based on~\cite{Arakawa,BeemRastelli,Milas}, we make the following conjecture.
\begin{conj} For $k \equiv \ell \pmod 2$ the function 
$F_{\mathfrak{g},\ell,k}$ belongs to the ring of quasimodular forms $\widetilde{M}(\Gamma)$, where~$\Gamma$ is a congruence subgroup of~$\SL_2(\Z)$.
\end{conj}
We prove this conjecture for~$\mathfrak g$ of type~$A$ (Corollaries~\ref{k.unif} and~\ref{cor:Aodd}), and for~$\mathfrak g$ of type~$B$, $C$, or~$D$ when $k$ and~$\ell$ are even (Corollary~\ref{cor:BCD}).

\subsection{Symmetrization} 
Since $P\cap C^0=\rho+P_+$, we have
\[
F_{\mathfrak{g},\ell,k}(\boldsymbol{\zeta};q)\=  \sum_{\lambda \in P \cap C^0} {\rm ch}(V_{\lambda-\rho})(\boldsymbol{\zeta})^\ell \frac{q^{\frac{k}{2} \sum_{\alpha \in \Delta_+} (\lambda , \alpha)}}{\prod_{\alpha \in \Delta_+} \left(1 - q^{(\lambda, \alpha)}\right)^{\!k}}.
\]
Let $\mathbf{k}\df \{ k_{\alpha}: \alpha \in \Delta_+\}$ where $k_{\alpha} \in\mathbb{N}$ and let also $k_{-\alpha}\df k_{\alpha}$, for~$\alpha \in \Delta_+$. We introduce the symmetrized~$q$-series associated to the simple Lie algebra~$\mathfrak{g}$ with the vector of weight~$\mathbf{k}$ as
\[
F_{\mathfrak{g},\ell,\mathbf{k}}^W(\boldsymbol{\zeta};q)\df \sum_{w \in W} \sum_{\lambda \in P_+} {\rm ch}(V_\lambda)(\boldsymbol{\zeta})^\ell  \frac{q^{\frac{1}{2} \sum_{\alpha \in \Delta_+} k_{w \alpha} (\lambda + \rho, \alpha)}}{\prod_{\alpha \in \Delta_+} \left(1 - q^{(\lambda + \rho, \alpha)}\right)^{\!k_{w \alpha}}},
\]
where~$w \alpha$ indicates the Weyl group element action on~$\alpha$. The limit of $F^W_{\mathfrak{g},\ell,\mathbf{k}}(\boldsymbol{\zeta};q)$ as $\zeta_j\to1$ is the series $F^W_{\mathfrak{g},\ell,\mathbf{k}}(q)$ of~\eqref{eq:FWq}.

\begin{rem}
Let $\mathfrak{g}=\mathfrak{sl}(3)$ and choose~$k_{\alpha_1}=2$, $k_{\alpha_2}=4$, and $k_{\alpha_1+\alpha_2}=6$, so~$\mathbf{k}=\{2,4,6\}$  and~$\ell=0$. 
The Weyl group~$W \cong S_3$ is generated by two reflections $s_{\alpha_1} : (\alpha_1,\alpha_2,\alpha_1+\alpha_2) \mapsto (-\alpha_1,\alpha_1+\alpha_2,\alpha_2)$ and $s_{\alpha_2} : (\alpha_1,\alpha_2,\alpha_1+\alpha_2) \mapsto (\alpha_1+\alpha_2, -\alpha_2,\alpha_1)$.
It is not hard to see that the (un-symmetrized) sum 
$$\sum_{n_1,n_2 \geq 1} \frac{q^{n_1+2n_2+3(n_1+n_2)}}{(1-q^{n_1})^2(1-q^{n_2})^4 (1-q^{n_1+n_2})^6}$$
does not belong to~$\widetilde{M}$. For the symmetrization we get
\begin{align*}
 F_{\mathfrak{sl}_3,0,\{2,4,6 \}}^{S_3}&\=2 \sum_{n_1,n_2 \geq 1}\ \sum_{\substack{(a,b,c)\in\{(2,4,6),\\(4,6,2),(6,2,4)\}}} \frac{q^{\frac12(an_1+bn_2+c(n_1+n_2))}}{(1-q^{n_1})^a(1-q^{n_2})^b (1-q^{n_1+n_2})^c}\\
&\=\frac{1}{2615348736000}\bigl(693E_4^3-220E_6^2+16380E_4E_6+126126E_4^2\\
&\hphantom{\=\frac{1}{2615348736000}}+240240E_2E_6+6126120E_2E_4-1667380E_6+60540480E_2^2\\
&\hphantom{\=\frac{1}{2615348736000}}-26657631E_4-300420120E_2+261695312\bigr),
\end{align*}
which is quasimodular.
   \end{rem}
  Following the discussion in the Introduction and in~\cite{Milas}, we conjecture that Weyl-sym\-metrized cone sums satisfy quasimodularity. More precisely,
  \begin{conj} If $k_\alpha \equiv \ell \pmod 2$ for all~$\alpha \in \Delta_+$, then 
$F^W_{\mathfrak{g},\ell,\mathbf{k}}$ belongs to the ring of quasimodular forms $\widetilde{M}(\Gamma)$, for some congruence subgroup~$\Gamma$ of~$\SL_2(\Z)$.
\end{conj}
In Sections 3 and 4 we prove this conjecture for type~$A$, and for types~$B$, $C$, $D$ with even exponents. 

For $\mathbf k=\{k_\alpha:\alpha\in\Delta_+\}$ as above we set
\[
\widehat{F}_{\mathfrak{g},\ell,\mathbf{k}}(\boldsymbol{\zeta};q)\df \!\!\!\!\sum_{\lambda \in P \setminus \mathcal H}\!\!\!\! {\rm ch}\!\left(V_{\lambda-\rho}\right)(\boldsymbol{\zeta})^{\ell} \frac{q^{\frac{1}{2} \sum_{\alpha \in \Delta_+} k_{\alpha} (\lambda, \alpha)}}{\prod_{\alpha \in \Delta_+} \left(1 - q^{(\lambda, \alpha)}\right)^{\!k_\alpha}},
\]
and write $\widehat{F}_{\mathfrak{g},\ell,k}$ if $k_\alpha=k$ for all~$\alpha\in\Delta_+$.
\begin{thm} \label{thm.var} If $k_\alpha \equiv \ell  \pmod 2$ for all~$\alpha \in \Delta_+$, we have
\[
\widehat{F}_{\mathfrak{g}, \ell,\mathbf{k}}(\boldsymbol{\zeta};q) \= F^W_{\mathfrak{g},\ell,\mathbf{k}}(  \boldsymbol{\zeta};q).
\]
\end{thm}
\begin{proof} As we already point out above, $P \setminus \mathcal H=\bigcup_{w \in W} w(P \cap C^0)=\bigcup_{w \in W}w^{-1}(\rho+ P_+)$, so we can write 
\begin{align*}
\widehat{F}_{\mathfrak{g},\ell,\mathbf{k}}(\boldsymbol{\zeta};q)\= & \!\!\!\!\sum_{\lambda \in P \setminus \mathcal H}\!\!\!\! {\rm ch}(V_{\lambda-\rho})(\boldsymbol{\zeta})^{\ell} \frac{q^{\frac{1}{2} \sum_{\alpha \in \Delta_+} k_{\alpha} (\lambda, \alpha)}}{\prod_{\alpha \in \Delta_+} \left(1 - q^{(\lambda, \alpha)}\right)^{\!k_\alpha}}\\
\=&\sum_{w \in W} \sum_{\lambda \in w^{-1}(P_++\rho)} {\rm ch}(V_{\lambda-\rho})(\boldsymbol{\zeta})^{\ell} \frac{q^{\frac{1}{2} \sum_{\alpha \in \Delta_+} k_{\alpha} (\lambda, \alpha)}}{\prod_{\alpha \in \Delta_+} (1 - q^{(\lambda, \alpha)})^{k_\alpha}}\\
\=&\sum_{w \in W} \sum_{\lambda \in P_+} {\rm ch}\!\left(V_{w^{-1}(\lambda+\rho)-\rho}\right)(\boldsymbol{\zeta})^{\ell} \frac{q^{\frac{1}{2} \sum_{\alpha \in \Delta_+} k_{\alpha} (w^{-1}(\lambda+\rho), \alpha)}}{\prod_{\alpha \in \Delta_+} (1 - q^{(w^{-1}(\lambda+\rho), \alpha)})^{k_\alpha}}\\
 \=&\sum_{w \in W} \sum_{\lambda \in P_+}  {\rm ch}\!\left(V_{w^{-1}(\lambda+\rho)-\rho}\right)(\boldsymbol{\zeta})^{\ell}  \frac{q^{\frac{1}{2} \sum_{\alpha \in \Delta_+} k_{\alpha} (\lambda+\rho, w \alpha)}}{\prod_{\alpha \in \Delta_+} \left(1 - q^{(\lambda +\rho, w \alpha)}\right)^{\!k_\alpha}}.
 \end{align*}
Next we note that, by Lemma~\ref{neg-roots}, we can write $w(\Delta_+)=\Delta^w_+ \dot \cup \Delta^w_-$ and $\Delta_+=\Delta^w_+ \dot \cup (-\Delta^w_-)$, where $\Delta^w_{\pm} \subset \Delta_\pm$ and~$|\Delta^w_-|=\ell(w)$.  Together with 
\[\left[\frac{q^{\frac{k_\alpha}{2} (\lambda, \alpha)}}{(1-q^{(\lambda, \alpha)})^{k_\alpha}}\right]_{q \mapsto q^{-1}}\=(-1)^{k_{\alpha}} \frac{q^{\frac{k_{\alpha}}{2} (\lambda, \alpha)}}{\left(1-q^{(\lambda, \alpha)}\right)^{\!k_\alpha}}\] for~$k_{\alpha} >0$, we can rewrite the~$q$-product in the above sum as 
\begin{align*}
& \frac{q^{\frac{1}{2} \sum_{\alpha \in \Delta_+} k_{\alpha} (\lambda+\rho, w \alpha)}}{\prod_{\alpha \in \Delta_+} \left(1 - q^{(\lambda+\rho, w \alpha)}\right)^{\!k_{\alpha}}}\=\frac{q^{\frac{1}{2} \sum_{\alpha \in \Delta^w_+} k_{w^{-1} \alpha} (\lambda+\rho, \alpha)}}{\prod_{\alpha \in \Delta_+^w} \left(1 - q^{(\lambda+\rho, \alpha)}\right)^{\!k_{w^{-1} \alpha}}}
 \frac{q^{\frac{1}{2} \sum_{\alpha \in \Delta^w_-} k_{w^{-1} \alpha} (\lambda+\rho,  \alpha)}}{\prod_{\alpha \in \Delta^w_-} \left(1 - q^{(\lambda+\rho, \alpha)}\right)^{\!k_{w^{-1} \alpha}}} \\
& \hspace{1.2cm}\=\frac{q^{\frac{1}{2} \sum_{\alpha \in \Delta^w_+} k_{w^{-1} \alpha} (\lambda+\rho, \alpha)}}{\prod_{\alpha \in \Delta_+^w} \left(1 - q^{(\lambda+\rho, \alpha)}\right)^{\!k_{w^{-1} \alpha}}} 
 (-1)^{ \sum_{\alpha \in \Delta^w_-}  k_{w^{-1} \alpha}} \frac{q^{\frac{1}{2} \sum_{\alpha \in (-\Delta^w_-)} k_{-w^{-1} \alpha} (\lambda+\rho,  \alpha)}}{\prod_{\alpha \in (-\Delta^w_-)} \left(1 - q^{(\lambda+\rho, \alpha)}\right)^{\!k_{-w^{-1} \alpha}}}\\
&\hspace{1.2cm}\= (-1)^{\sum_{\alpha \in \Delta^w_-} k_{w^{-1} \alpha}}  \frac{ q^{\frac12\sum_{\alpha \in (-\Delta^w_-)} k_{-w^{-1} \alpha} (\lambda+\rho,  \alpha) +\frac12\sum_{\alpha \in \Delta^w_+} k_{w^{-1} \alpha} (\lambda+\rho,  \alpha)   }}{\prod_{\alpha \in \Delta_+} \left(1 - q^{(\lambda+\rho, \alpha)}\right)^{\!k_{w^{-1} \alpha}}} \\
&\hspace{1.2cm}\= (-1)^{\sum_{\alpha \in \Delta^w_-} k_{w^{-1} \alpha}}  \frac{ q^{\frac12\sum_{\alpha \in \Delta_+} \! k_{w^{-1} \alpha} (\lambda+\rho,  \alpha)   }}{\prod_{\alpha \in \Delta_+} \left(1 - q^{(\lambda+\rho, \alpha)}\right)^{\!k_{w^{-1} \alpha}}}.
\end{align*}

Plugging this into the sum and using Lemma~\ref{Weyl}, we obtain
\begin{align*}
\widehat{F}_{\mathfrak{g},\ell,\mathbf{k}}(\boldsymbol{\zeta};q)
&\=
\sum_{w \in W} \sum_{\lambda \in P_+}
(-1)^{\ell\cdot\ell(w)+\sum_{\alpha \in \Delta^w_-} k_{w^{-1}\alpha}}
{\rm ch}(V_\lambda)(\boldsymbol{\zeta})^\ell
\frac{
q^{\frac{1}{2}\sum_{\alpha \in \Delta_+}
k_{w^{-1}\alpha}(\lambda+\rho,\alpha)}
}{
\prod_{\alpha \in \Delta_+}
\left(1-q^{(\lambda+\rho,\alpha)}\right)^{\!k_{w^{-1}\alpha}}
}
\\
&\=
F^W_{\mathfrak{g},\ell,\mathbf{k}}(\boldsymbol{\zeta};q),
\end{align*}
where the sign $(-1)^{\ell\cdot\ell(w)}$ comes from Lemma~\ref{Weyl} (note that~$\ell(w^{-1})=\ell(w)$), and the total sign in the first line equals~$1$ because $k_\alpha\equiv \ell\pmod 2$
gives \[\sum_{\alpha \in \Delta^w_-}k_{w^{-1}\alpha}\equiv \ell\,|\Delta^w_-|=\ell\, \cdot \ell(w)\pmod{2}.\qedhere\]
\end{proof}

As an immediate consequence, when~$k$ and~$\ell$ have the same parity, the Weyl-group sum reduces to a multiple of the original series.

\begin{cor}\label{cor:symm}
For~$k\equiv \ell\pmod 2$, we have
$$
\widehat{F}_{\mathfrak{g},\ell,k}(\boldsymbol{\zeta};q)
\=
|W|\,F_{\mathfrak{g},\ell,k}(\boldsymbol{\zeta};q).
$$
\end{cor}

\begin{rem}

The symmetrization can be applied to any subset~$S \subset \Delta_+$. Let~$k \equiv \ell \pmod2$ and define
$$F^W_{S,\mathfrak{g}, \ell,k}(q)\df \sum_{w \in W} \sum_{ \lambda \in P_+} \prod_{\alpha \in S }  \left(\frac{( \lambda + \rho, w\alpha )}{( \rho, \alpha )}\right)^{\ell}  \frac{q^{\frac{1}{2} \sum_{\alpha \in w(S)} k (\lambda + \rho, \alpha)}}{\prod_{\alpha \in w(S)} \left(1 - q^{(\lambda + \rho, \alpha)}\right)^{\!k}},
$$
which converges if $S$ spans the same space as~$\Delta_+$.  
Then, as above, we have 
\begin{equation} \label{SW-cone}
F^W_{S,\mathfrak{g}, \ell,k }(q)\= \!\!\!\!\sum_{\lambda \in P \setminus \mathcal H}   \prod_{\alpha \in S } \left(\frac{( \lambda, \alpha )}{( \rho, \alpha )}\right)^{\ell}   \frac{q^{\frac{1}{2} \sum_{\alpha \in S} k (\lambda, \alpha)}}{\prod_{\alpha \in S} \left(1 - q^{(\lambda, \alpha)}\right)^{\!k}}.
\end{equation}

The construction in this section can be slightly generalized by introducing additional parameters $\boldsymbol{\ell} = \{\ell_\alpha : \alpha \in \Delta_+\},\ \ell_\alpha \equiv k_\alpha \pmod{2},$ namely
$$F^W_{\mathfrak{g}, \boldsymbol{\ell}, \mathbf{k}}(q)\df \sum_{w \in W} \sum_{ \lambda \in P_+} \prod_{\alpha \in \Delta_+ } \left(\frac{( \lambda + \rho, \alpha )}{( \rho, \alpha )}\right)^{\ell_{w \alpha}}  \frac{q^{\frac{1}{2} \sum_{\alpha \in \Delta_+} k_{w \alpha} (\lambda + \rho, \alpha)}}{\prod_{\alpha \in \Delta_+} \left(1 - q^{(\lambda + \rho, \alpha)}\right)^{\!k_{w \alpha}}}.
$$
However, we are not aware of a representation-theoretic interpretation of this series.
\end{rem}

\subsection{The functions~\texorpdfstring{$X$}{X} and~\texorpdfstring{$Y$}{Y}} 
Define
\begin{equation}\label{def:XY}
X(a)\df \frac{1+q^a}{1-q^a},
\qquad
Y(a)\df \frac{q^a}{(1-q^a)^2}
\=\frac14(X(a)^2-1).
\end{equation}
For~$a\not =b$, we have the elementary identity
\begin{equation}\label{eq:induction1}
X(a)X(b)
\=
X(a)X(b-a)-X(b)X(b-a)+1.
\end{equation}

We next record some elementary properties of~$X$ and~$Y$ that are used in the proof of Theorem~\ref{thm:main1}. For~$a, s\in \mathbb{N}$, define the Lambert series
\begin{equation}\label{eq:Ldef}
L_s(a)\df \sum_{m\geq1}m^s q^{am}.
\end{equation}
In particular, $L_1(a)=Y(a)$. It is convenient to introduce the normalized~$q$-derivative
\begin{equation}\label{eq:Dadef}
\mathcal D_a\df \frac1a q\frac{d}{dq}.
\end{equation}
Since $\mathcal D_a(q^{am})=mq^{am}$, we have
\begin{equation}\label{eq:Lderivative}
L_{s+1}(a)\=\mathcal D_a\left(L_s\right)(a).
\end{equation}
Moreover, direct differentiation gives
\begin{equation}\label{eq:DXY}
\mathcal D_a\left(X\right)(a)\=2Y(a),
\qquad
\mathcal D_a\left(Y\right)(a)\=X(a)Y(a).
\end{equation}

In the following lemma we fix~$a\in \mathbb{N}$ and abbreviate
$$
X\df X(a),\qquad Y\df Y(a),\qquad L_s\df L_s(a).
$$
As usual, $[x^s]P(x)$ denotes the coefficient of~$x^s$ in a polynomial $P(x)$.

\begin{lem}\label{lem:Lshape}
There exist polynomials~$f_r,g_r\in\Q[x]$ for $r\in \mathbb{N}$, determined recursively by
\begin{equation}\label{eq:fgrec}
f_1(x)\=x,\qquad
g_r(x)\=xf_r'(x),\qquad
f_{r+1}(x)\=2xg_r(x)+(1+4x)xg_r'(x),
\end{equation}
such that
\begin{equation}\label{eq:Lshape}
L_{2r-1}\=f_r(Y),
\qquad
L_{2r}\=Xg_r(Y).
\end{equation}
Moreover, $f_r$ and~$g_r$ have degree~$r$, vanish at~$x=0$, and their leading coefficients are nonzero.
\end{lem}

For example,
$$
f_1(x)\=g_1(x)\=x,\qquad
f_2(x)\=6x^2+x,\qquad
g_2(x)\=12x^2+x,
$$
and hence
$$
L_1\=Y,\qquad
L_2\=XY,\qquad
L_3\=6Y^2+Y,\qquad
L_4\=X\left(12Y^2+Y\right).
$$

\begin{proof}[Proof of Lemma~\ref{lem:Lshape}]
The identity~$L_1=Y=f_1(Y)$ gives the first assertion of \eqref{eq:Lshape} for~${r=1}$. Suppose that $ L_{2r-1}=f_r(Y). $ Using~\eqref{eq:Lderivative}, \eqref{eq:DXY}, and~\eqref{eq:fgrec}, we obtain
$$
L_{2r}
\=
\mathcal D_a(f_r)(Y)
\=
f_r'(Y)\mathcal D_a(Y)
\=
XYf_r'(Y)
\=
Xg_r(Y).
$$
Applying~$\mathcal D_a$ once more gives
\begin{align*}
L_{2r+1}
&\=
\mathcal D_a\bigl(Xg_r(Y)\bigr)\=
\mathcal D_a(X)g_r(Y)
+
Xg_r'(Y)\mathcal D_aY\=
2Yg_r(Y)+X^2Yg_r'(Y).
\end{align*}
Since~\eqref{def:XY} gives~$X^2=1+4Y$, we conclude that
$$
L_{2r+1}
\=
2Yg_r(Y)+(1+4Y)Yg_r'(Y)
\=
f_{r+1}(Y),
$$
which proves~\eqref{eq:Lshape} by induction.

Both recursions in~\eqref{eq:fgrec} have an explicit factor~$x$ on the right-hand side, so~$f_r$ and~$g_r$ vanish at~$x=0$. Finally, if~$\deg (f_r)=r$, then $\deg (g_r)=\deg(xf_r')=r$ and~$\deg (f_{r+1})=r+1$; comparing leading coefficients in~\eqref{eq:fgrec} yields
\begin{equation}\label{eq:leadcoeff}
[x^r]g_r\=r[x^r]f_r,
\qquad
[x^{r+1}]f_{r+1}
\=
2r(2r+1)[x^r]f_r.
\end{equation}
 Since~$\deg f_1=1$ and~$[x]f_1=1$, all the leading coefficients occurring are nonzero.
\end{proof}

\subsection{Some quasi-Jacobi forms}

Here we record some quasi-Jacobi forms needed later. The result is probably known, but we include a proof for completeness.

\begin{lem}\label{lem:XYjacobi}
Let~$z\in\C$ and $\zeta\df e^{2\pi i z}$. For every~$r\in \mathbb{N}$ 
\[
\mathcal Y_r(z;\tau)\df \!\sum_{a \in \mathbb{Z} \setminus \{0\}}\! Y(a)^{r}\,\zeta^{a},\qquad
\mathcal X_r(z;\tau)\df \!\sum_{a \in \mathbb{Z} \setminus \{0\}}\!X(a)\,Y(a)^{r}\,\zeta^{a} \qquad (0<\Im z<\Im\tau)
\]
are polar quasi-Jacobi forms of level one and index~$0$, of mixed weight at most~$2r$ and~$2r+1$, respectively. 
\end{lem}
\begin{proof}
Since~$Y(-a)=Y(a)$ and~$X(-a)=-X(a)$ by~\eqref{def:XY}, we have
\begin{equation}\label{eq:XYsym}
\mathcal Y_r\=\sum_{a\geq1}Y(a)^{r}\big(\zeta^{a}+\zeta^{-a}\big),\qquad
\mathcal X_r\=\sum_{a\geq1}X(a)Y(a)^{r}\big(\zeta^{a}-\zeta^{-a}\big).
\end{equation}
For~$\ell\geq2,$ set
$$
\mathcal F_\ell(z;\tau)\df \left[\left(\frac{1}{2\pi i}\frac{\partial}{\partial z_1}\right)^{\!\ell-1}\left(\frac{1}{2\pi i}F(z_1,z;\tau)-\frac12\coth\big(\pi i z_1\big)\right)\right]_{z_1=0}.
$$
Applying $(\frac{1}{2\pi i}\frac{\partial}{\partial z_1})^{\ell-1}$ termwise to~\eqref{eq:FZagier} and specializing to~$z_1=0$ gives
\begin{equation}\label{eq:FellG}
\mathcal F_\ell(z;\tau)\=-\sum_{a\geq1}L_{\ell-1}(a)\big(\zeta^{a}+(-1)^{\ell}\zeta^{-a}\big).
\end{equation}
Hence~$\mathcal F_\ell$ agrees, up to an additive constant and the factor~$(\ell-1)!/(2\pi i)^{\ell-1}$, with a Laurent coefficient at~$z_1=0$ of~$\frac{1}{2\pi i}F$. This Laurent coefficient is a polar quasi-Jacobi form in~$z$ of index~$0$ and weight~$\ell$ by the properties listed above, and the constant has weight~$0$. Hence~$\mathcal F_\ell$ is a polar quasi-Jacobi form of index~$0$ and of mixed weight at most~$\ell$.

By~\eqref{eq:Lshape}, \eqref{eq:XYsym}, and~\eqref{eq:FellG} we obtain, for~$r\in \mathbb{N}$,
$$
-\mathcal F_{2r}\=\sum_{s=1}^{r}\big([x^s]f_r\big)\,\mathcal Y_s,\qquad
-\mathcal F_{2r+1}\=\sum_{s=1}^{r}\big([x^s]g_r\big)\,\mathcal X_s,
$$
where the sums start at~$s=1$ because~$f_r$ and~$g_r$ vanish at~$x=0$. By~\eqref{eq:leadcoeff} both systems are triangular with nonzero diagonal entries, hence invertible: each~$\mathcal Y_r$ is a~$\C$-linear combination of $\mathcal F_2,\ldots,\mathcal F_{2r}$, and each~$\mathcal X_r$ of $\mathcal F_3,\ldots,\mathcal F_{2r+1}$, which gives the bounds on the weights. In particular~$\mathcal Y_r$ and~$\mathcal X_r$ are polar quasi-Jacobi forms of level one. 
\end{proof}

\section{Level one}\label{sec:levelone}
In this section, we prove the main theorem in case the resulting form is of level~one.
\subsection{Sets of level one}
For $\vec{v}\in \Z^n\setminus\{{\bm0}\}$, we set $\iota(\vec{v})\df \max\{j:1\le j\le n, v_j\neq 0\}$ and write $\kappa(\vec v)=v_{\iota(\vec v)}$ for the value of the component of~$\vec{v}$ of index~$\iota(\vec{v})$, and we extend these to $\iota(\vec 0)\df 0$ and $\kappa(\vec 0)\df 0$. We now consider the following partial order on~$\Z^n$.
\begin{defn}
For $\vec{v}, \vec{w}\in \Z^n$, we write~$\vec{v}>\vec{w}$ if either $\iota(\vec{v})>\iota(\vec{w})$ or  ($\iota(\vec{v})=\iota(\vec{w})$ and $\left|\kappa(\vec{v})\right|>\left|\kappa(\vec{w})\right|$).
Moreover, we write $\vec{v}\parallel \vec{w}$ if~$\vec{v}$ and~$\vec{w}$ are incomparable, that is~$\vec{v}\not> \vec{w}$ and~$\vec{w}\not>\vec{v}$.
\end{defn}
Note that $\vec{v}>\vec{0}$ for non-zero~$\vec{v}$, and $\vec{v}\parallel \vec{w}$ if and only if $\iota(\vec{v})=\iota(\vec{w})$ and $|\kappa(\vec{v})|=|\kappa(\vec{w})|$.
For a set~$S$, we write $-S\df \{-\vec{v} : \vec{v}\in S\}$ and $S^*\df S\cup (-S)\cup \{\vec{0}\}$. 
\begin{defn}\label{def:level1}
For~$n\in\N_0$, we call a set~$S\subseteq\Z^n$ of \emph{level one} if the following hold:
\begin{enumerate}[label = (\roman*),leftmargin=*]
\item\label{def:level1:kappa} $\kappa(\vec{v})= 1$ for all~$\vec{v} \in S$;
\item\label{def:level1:diff} $\vec{v}-\vec{w}\in S^*$  for all $\vec{v},\vec{w}\in S^*$ for which $\vec{v}>\vec{v}-\vec{w}$;
\item\label{def:level1:span} The dimension of the vector space generated by~$S$ is~$n$.
\end{enumerate}
\end{defn}
Condition~\ref{def:level1:diff} simplifies for sets for which \ref{def:level1:kappa} holds. Let $\bm{v},\bm{w}\in\Z^n$ with $|\kappa(\bm{v})|=|\kappa(\bm{w})|=1$ and $\bm v\neq\bm w$. We distinguish four cases.
\begin{itemize}[leftmargin=*]
\item If $\iota(\bm{w})<\iota(\bm{v})$, then $\iota(\bm v-\bm w)=\iota(\bm v)$ and $\kappa(\bm v-\bm w)=\kappa(\bm v)$, so that $\bm{v}\parallel\bm{v}-\bm{w}$.
\item If $\iota(\bm{w})>\iota(\bm{v})$, then $\iota(\bm v-\bm w)=\iota(\bm w)$, so that $\bm{v}-\bm{w}>\bm{v}$.
\item If $\iota(\bm{v})=\iota(\bm{w})$ and $\kappa(\bm v)=\kappa(\bm w)$, then $\iota(\bm{v}-\bm{w})<\iota(\bm{v})$, so that $\bm v>\bm v-\bm w$. 
\item If $\iota(\bm{v})=\iota(\bm{w})$ and $\kappa(\bm v)=-\kappa(\bm w)$, then $|\kappa(\bm{v}-\bm{w})|=2$, so that $\bm{v}-\bm{w}>\bm{v}$.
\end{itemize}
Hence, in condition~\ref{def:level1:diff}, only the pairs with $\iota(\bm v)=\iota(\bm w)$ and $\kappa(\bm v)=\kappa(\bm w)$ impose a condition. Replacing such a pair~$(\bm v,\bm w)$ by~$(-\bm v,-\bm w)$ gives the same condition, as $S^*=-S^*$. Hence a set~$S$ with~$\kappa(\bm{v})=1$ for all~$\bm{v}\in S$ satisfies condition~\ref{def:level1:diff} if and only if
\begin{enumerate}[label=(\roman*$^\dagger$),ref=(\roman*$^\dagger$),start=2,leftmargin=*,labelindent=1pt]
\item\label{eq:closure1} $\bm{v}-\bm{w}\in S^*$ for all~$\bm{v},\bm{w}\in S$ with $\iota(\bm{v})=\iota(\bm{w})$.
\end{enumerate}
\begin{example}[Root system~$A_n$]\label{ex:An}
Let~$S=\Delta_+(A_n)\subset\Z^n$, in the coordinates with respect to the simple roots of Section~\ref{sec:prelim}. To a subset~$I$ of~$\{1,\ldots,n\}$ associate the vector~$\vec{v}$ for which~$v_j\df 1$ if~$j\in I$ and~$v_j\df 0$ otherwise. Then~$S$ consists of the vectors associated to sets~$I$ of consecutive integers.  
This set~$S$ is of level~one.  Namely, the first condition follows directly from the fact that such~$\vec{v}$ satisfy $\vec{v} \in \{0,1\}^n$,  the second from the fact that the difference of two sets of consecutive integers ending with the same integer is again a set of consecutive integers and the last from the fact that the standard basis vectors~$\varepsilon_j$ are obtained from the singleton sets~$I=\{j\}$.
\end{example}
 
\begin{example}
The set $S \subset \mathbb{Z}^n$ consisting of the vectors~$\bm v_1=(1,0,\ldots,0)$, $\bm v_2=(1,1,0,\ldots,0)$, \ldots, $\bm v_n=(1,1,\ldots,1)$ is of level one. Condition~\ref{def:level1:kappa} holds as $\kappa(\bm v_j)=1$. Moreover, $\bm v_j \parallel \bm v_j - \bm v_k$ for~$j>k$ and~$\bm v_j \not> \bm v_j-\bm v_k$ for~$j<k$, so the second condition holds. Finally, the~$\bm v_j$ span~$\R^n$.  
\end{example}
This example is essentially the set $\{\varepsilon_1,\ldots,\varepsilon_n\}$ of positive roots of~$A_1 \times \cdots \times A_1$. Namely, the upper triangular matrix all of whose entries on and above the diagonal are equal to~$1$ maps~$\varepsilon_j$ to~$\bm v_j$. To make this precise, we write~$\mathrm{T}_n(\Z)$ for the subgroup of~$\mathrm{GL}_n(\Z)$ of upper triangular matrices (their diagonal entries are~$\pm1$) and~$\mathrm{UT}_n(\Z)$ for the subgroup of those with ones on the diagonal. For~$M\in\mathrm{T}_n(\Z)$ and~$\vec v\neq\vec0$ we have $\iota(M\vec v)=\iota(\vec v)$ and $|\kappa(M\vec v)|=|\kappa(\vec v)|$, so that~$M$ preserves the partial order. 
Moreover, $MS^*$ is again of the form~$T^*$ for a unique set~$T$ of vectors with~$\kappa=1$, namely the set of~$\vec v\in MS^*$ with~$\kappa(\vec v)=1$; if $M\in\mathrm{UT}_n(\Z)$, then simply~$T=MS$.

\begin{lem}
Let~$M\in\mathrm{T}_n(\Z)$ and let~$T$ be the set of~$\vec v\in MS^*$ with~$\kappa(\vec v)=1$. Then~$S$ is of level~one if and only if~$T$ is so.
\end{lem}
\begin{proof}
Condition~\ref{def:level1:kappa} holds for~$T$ by construction. Since~$M$ preserves the partial order and~$T^*=MS^*$, for~$\vec v,\vec w\in T^*$ the relation~$\vec v>\vec v-\vec w$ is equivalent to $M^{-1}\vec v>M^{-1}\vec v-M^{-1}\vec w$, and~$\vec v-\vec w\in T^*$ to $M^{-1}\vec v-M^{-1}\vec w\in S^*$; hence condition~\ref{def:level1:diff} for~$S$ implies condition~\ref{def:level1:diff} for~$T$. Finally $\langle T\rangle=M\langle S\rangle$. The converse follows by applying this to~$M^{-1}$.
\end{proof}
In view of this lemma, we define the following notion of equivalence:
\begin{defn}
We say sets~$S$ and~$T$ are \emph{equivalent} if there exists $M\in \mathrm{T}_n(\Z)$ such that~${MS^*=T^*}$. 
\end{defn}

For a set~$S\subseteq\Z^n$ and~$r\in\N_0$, we denote 
\[
S_r\df \{\vec{v}\in S : \iota(\vec{v})\leq r\}.
\]
Note that~$S_r$ naturally is a subset of~$\Z^r$, since the last~$n-r$ components of~$\vec{v}\in S_r$ are zero. We call $S_r\backslash S_{r-1}=\{\vec v\in S:\iota(\vec v)=r\}$ the \emph{$r$-th layer} of~$S$. In the sequel we need the following properties of these subsets~$S_r$:

\begin{lem}\label{lem:parallel} Let~$S$ be of level~one and $1\leq r\leq n$.
\begin{enumerate}[label={\upshape(\arabic*)},leftmargin=*]
\item\label{lem:parallel:nonempty} The layer $S_r\backslash S_{r-1}$ is non-empty.
\item\label{lem:parallel:sub} The set~$S_{r}$ is of level~one.
\item\label{lem:parallel:diff} If $\iota(\vec{v})=\iota(\vec{w})=r$, then 
$\vec{v}-\vec{w}\in S_{r-1}^*$.
\end{enumerate}
\end{lem}
\begin{proof} 
(1) We show the stronger statement that $\dim(\langle S_r\rangle) = r$ for all~$0\leq r\leq n$. Since then $S_{r-1}\subseteq \Z^{r-1}$ satisfies $\dim(\langle S_{r-1}\rangle) = r-1<r=\dim(\langle S_r\rangle)$, this yields $S_r\backslash S_{r-1}\neq\emptyset$. We argue by descending induction on~$r$. For~$r=n$ the claim is property~\ref{def:level1:span} of Definition~\ref{def:level1}. Assume next $\dim(\langle S_r\rangle)=r$ for some~$1\leq r\leq n$. As~$S_r\subseteq \Z^r$ spans an~$r$-dimensional space, there exists~$\vec u\in S_r$ with~$\iota(\vec u)=r$. For any~$\vec v\in S_r$ with~$\iota(\vec v)=r$, we have $\kappa(\vec v)=\kappa(\vec u)=1$ by property~\ref{def:level1:kappa}, so the~$r$-th entries of~$\vec v$ and~$\vec u$ agree. Hence $\iota(\vec v-\vec u)<r$, so~$\vec u>\vec u-\vec v$ and property~\ref{def:level1:diff} gives $\vec v-\vec u\in S_{r-1}^*\subseteq\langle S_{r-1}\rangle$. Thus every~$\vec v\in S_r$ lies in $\R\vec u+\langle S_{r-1}\rangle$ (either~$\iota(\vec v)<r$, so $\vec v\in\langle S_{r-1}\rangle$, or~$\iota(\vec v)=r$, so $\vec v\in\vec u+\langle S_{r-1}\rangle$). Hence $r=\dim(\langle S_r\rangle)\leq 1+\dim(\langle S_{r-1}\rangle)$, and together with $\dim(\langle S_{r-1}\rangle)\leq r-1$ we obtain $\dim(\langle S_{r-1}\rangle)=r-1$, completing the induction.

\noindent(2) Since~$S$ is of level~one, for all $\vec{v}, \vec{w}\in S_r^*$ for which $\vec{v}>\vec{v}-\vec{w}$, we have that $\vec{v}-\vec{w}\in S^*$. Since $\iota(\vec{v}-\vec{w})\leq \max\{\iota(\vec{v}),\iota(\vec{w})\}\leq r$, we conclude that $\vec{v}-\vec{w}\in S_r^*$, which is the second property of Definition~\ref{def:level1} for~$S_r$. The first property is inherited from~$S$. Finally, considering an element of $S_m\backslash S_{m-1}$ for all~$m\leq r$, which exist by~\ref{lem:parallel:nonempty}, shows that the dimension of the vector space generated by the elements of~$S_r$ is~$r$.

\noindent(3) For a set of level~one, $\iota(\vec{v})=\iota(\vec{w})$ implies $\kappa(\vec{v})=\kappa(\vec{w})=1$, so the~$\iota(\vec{v})$-th entries of~$\vec{v}$ and~$\vec{w}$ agree and $\iota(\vec{v}-\vec{w})<\iota(\vec{v})=r$. Hence $\vec{v}>\vec{v}-\vec{w}$, and property~\ref{def:level1:diff} gives $\vec{v}-\vec{w}\in S^*$. As $\iota(\vec{v}-\vec{w})<r$ we conclude $\vec{v}-\vec{w}\in S_{r-1}^*$.
\end{proof}

\begin{example}
Consider the set of level~one given by
\[
S\=\{(1,0,0),(1,1,0),(0,1,0),(0,1,1),(0,0,1)\}.
\]
This is the set~$\Delta_+(A_3)$ of Example~\ref{ex:An} without~$(1,1,1)$. It is not equivalent to~$\Delta_+(A_3)$, $\Delta_+(A_1\times A_2)$ or $\Delta_+(A_1\times A_1\times A_1)$. All sets of level one of rank three and four that are subsets of~$\Delta_+(A_n)$ are listed in Appendix~\ref{sec:appendix}. \end{example}

\begin{prop}\label{prop:An}
All sets of level~one are subsets of~$\Delta_+(A_n)$ up to equivalence.
\end{prop}
\begin{proof}
Let~$S$ be of level~one. First of all, suppose $\vec{v},\vec{w},\vec{a},\vec{b}\in S$ with $\vec v \neq \vec w$ are all pairwise incomparable and satisfy $\vec{v}-\vec{w}=\vec{a}-\vec{b}$. We show that~$\vec{v}=\vec{a}$ and~$\vec{w}=\vec{b}$. Since~$\vec{w}$ and~$\vec{a}$ are incomparable, we have $\vec{w}>\vec{w}-\vec{a}$, so $\vec{w}-\vec{a}\in S^*$ by condition~\ref{def:level1:diff}; similarly $\vec v-\vec b\in S^*$. Write $\ell=\iota(\vec{v}-\vec{b})$ and $j=\iota(\vec{w}-\vec{a})$; not both are~$0$, since~$\vec v=\vec b$ and~$\vec w=\vec a$ would give $\vec v-\vec w=\vec b-\vec a=-(\vec v-\vec w)$, i.e.~$\vec v=\vec w$. Suppose that~$j>\ell$. Then $v_j-b_j=0$ by definition of~$\ell$. Since~$v_j-w_j=a_j-b_j$, we find that~$a_j+w_j=2v_j$. Moreover, $w_j-a_j=\pm 1$, as $\vec{w}-\vec{a}$ is a non-zero element of~$S^*$. This is a contradiction since now $1 \equiv w_j-a_j\equiv w_j+a_j\equiv 0 \pmod2$. 

Next, assume that~$j=\ell$. Then $w_j-a_j=\pm1$ and~$v_j-b_j=\pm1$. If~$w_j-a_j=v_j-b_j$, then $\vec w-\vec a>(\vec w-\vec a)-(\vec v-\vec b)$, so that condition~\ref{def:level1:diff} gives
\[
(\vec{w}-\vec{a})-(\vec{v}-\vec{b}) \= (\vec{w}-\vec{v}) - (\vec{a}-\vec{b}) \= 2(\vec{w}-\vec{v}) \in S^*.
\]
If~$w_j-a_j=-(v_j-b_j)$, then in the same way $(\vec w-\vec a)+(\vec v-\vec b)=2(\vec w-\vec b)\in S^*$. Since $\kappa(\vec{u})=\pm 1$ for a non-zero element~$\vec{u}\in S^*$, an element of~$S^*$ with all entries even is~$\vec 0$. In the first case we obtain~$\vec w=\vec v$, which we have excluded, and in the second case~$\vec w=\vec b$ and hence~$\vec v=\vec a$.  Finally, note that the case~$\ell>j$ goes analogously to the case~$j>\ell$ by interchanging the roles of~$\vec{v}$ and~$\vec{w}$ and~$\vec{a}$ and~$\vec{b}$, respectively. We conclude that the differences between distinct incomparable elements in~$S$ are all different. 

Next, we claim that $|S_m\backslash S_{m-1}|\leq m$. We prove the claim inductively. For~$m=1$ the claim is clear, since there is only one vector~$\vec v=(1)$ in~$\Z$ with~$\kappa(\vec v)=1$. Now suppose that $|S_\ell\backslash S_{\ell-1}|\leq \ell$ for all $\ell\in\{1,\ldots,m-1\}$. Then, we find that
\[ |S_{m-1}| \= |S_{m-1}\backslash S_{m-2}| + |S_{m-2}\backslash S_{m-3}| + \ldots + |S_{1}\backslash S_{0}| \leq (m-1)+(m-2) + \ldots + 1 \= \frac{m(m-1)}{2}.\]
If $|S_{m}\backslash S_{m-1}|>m$, then there would be more than $\frac{1}{2}m(m-1)$ distinct differences in~$S_{m-1}$ (for each pair of distinct $\vec v,\vec w\in S_m\backslash S_{m-1}$ one of~$\pm(\vec v-\vec w)$ lies in~$S_{m-1}$ by Lemma~\ref{lem:parallel}~\ref{lem:parallel:diff}, and these are distinct by the first claim). Hence, the claim is proven inductively. Note that for~$S=\Delta_+(A_m)$ we have equality in the inequality of the claim.

To conclude the proof, we first show that~$S$ is equivalent to a subset of~$\Delta_+(A_n)$ if the equality $|S_{n}\backslash S_{n-1}|=n$ holds. For~$n=1$ there is nothing to prove, so let~$n\geq2$. By the first claim, the~$n(n-1)$ ordered differences of distinct elements of~$S_n\setminus S_{n-1}$ are distinct elements of $S_{n-1}^*\setminus\{\vec 0\}$, which has at most~$2\binom n2=n(n-1)$ elements by the second claim. Hence every element of~$S_{n-1}^*$ is such a difference, and $|S_m\setminus S_{m-1}|=m$ for all~$m\in\{1,\ldots,n\}$. Consider the~$(n-1)$-th entries of the elements of~$S_n\setminus S_{n-1}$. They take at least two values, since otherwise all differences would lie in~$S_{n-2}^*$, contradicting $|S_{n-1}\setminus S_{n-2}|=n-1$; and they take at most two values, which then differ by one, since the difference of two elements with distinct~$(n-1)$-th entries has~$\iota=n-1$ and~$\kappa=\pm1$.  If the two values occur with multiplicities~$j$ and~$n-j$, then the elements of~$S_{n-1}^*$ with~$\iota=n-1$ are the~$2j(n-j)$ differences of elements with distinct~$(n-1)$-th entries, so that~$n-1=j(n-j)$ and~$j\in\{1,n-1\}$. Let~$\vec s^{[1]}$ be the element with the unique value. Replacing~$S$ by its image under $x_{n-1}\mapsto-x_{n-1}$ if necessary, which is an equivalence, we may assume that this value is the smaller one, so that $(\vec t-\vec s^{[1]})_{n-1}=1$ for the other~$n-1$ elements~$\vec t$. These elements agree in their last two entries, their~$(n-1)(n-2)$ ordered differences are distinct elements of $S_{n-2}^*\setminus\{\vec0\}$, which has~$2\binom{n-1}{2}$ elements, so they exhaust it, and the same argument applies to their~$(n-2)$-th entries, and so on. Since the sign change $x_{n-m}\mapsto-x_{n-m}$ does not affect the entries with index larger than~$n-m$, we obtain, after replacing~$S$ by an equivalent set, an ordering $\vec s^{[1]},\ldots,\vec s^{[n]}$ of~$S_n\setminus S_{n-1}$ such that~$\vec s^{[m]}$ and~$\vec s^{[r]}$ agree in the entries with index larger than~$n-m$ and $(\vec s^{[r]}-\vec s^{[m]})_{n-m}=1$ for all~$r>m$. Hence the vectors $\vec s^{[m+1]}-\vec s^{[m]}$ $(1\leq m<n)$ and~$\vec s^{[1]}$ have~$\iota=n-m$ and~$\iota=n$, respectively, and~$\kappa=1$, so the matrix~$B$ with these vectors as columns~$n-m$ and~$n$ lies in~$\mathrm{UT}_n(\Z)$, and so does~$M\df B^{-1}$, which satisfies $M\vec s^{[1]}=\varepsilon_n$ and $M(\vec s^{[m+1]}-\vec s^{[m]})=\varepsilon_{n-m}$. Thus $M\vec s^{[m]}=\varepsilon_{n-m+1}+\cdots+\varepsilon_n$, and these are exactly the elements of~$\Delta_+(A_n)\setminus\Delta_+(A_{n-1})$. Since every element of~$S_{n-1}$ is, up to sign, a difference of two elements of~$S_n\setminus S_{n-1}$, and since~$M$ maps such a difference to an interval, $M$ maps~$S$ into~$\Delta_+(A_n)$. 

It remains to treat the case $|S_n\setminus S_{n-1}|<n$, which we do by induction on~$n$. For~$n=1$ there is nothing to prove. By Lemma~\ref{lem:parallel}~\ref{lem:parallel:sub}, the set~$S_{n-1}$ is of level one and of rank~$n-1$, so by the induction hypothesis we may assume that $S_{n-1}\subseteq\Delta_+(A_{n-1})$, after applying an equivalence which acts only on the first~$n-1$ coordinates. Write
\[
S_n\setminus S_{n-1}\=\{\vec y_1+\varepsilon_n,\ldots,\vec y_k+\varepsilon_n\}\qquad(k<n),
\]
where the~$\vec y_i$ have last coordinate~$0$. Condition~\ref{def:level1:diff} gives $\vec y_i-\vec y_j\in S_{n-1}^*\subseteq \Delta_+(A_{n-1})^*$ for all~$i,j$. The matrix in~$\mathrm{UT}_n(\Z)$ which differs from the identity only in its last column, this column being $\varepsilon_n-\vec y_1$, fixes~$S_{n-1}$ and maps $\vec y_i+\varepsilon_n$ to $\vec y_i-\vec y_1+\varepsilon_n$. Hence, we may assume~$\vec y_1=\vec 0$.

Put~$\vec r_0\df \vec 0$ and $\vec r_j\df \varepsilon_1+\cdots+\varepsilon_j$ for~$1\leq j\leq n-1$, so that $\Delta_+(A_{n-1})^*=\{\vec r_a-\vec r_b:0\leq a,b\leq n-1\}$. Let~$a\neq b$ and~$c\neq d$. Then $(\vec r_a-\vec r_b)-(\vec r_c-\vec r_d)=\vec r_e-\vec r_f$ means $\vec r_a+\vec r_d+\vec r_f=\vec r_b+\vec r_c+\vec r_e$; as $\vec r_1,\ldots,\vec r_{n-1}$ are linearly independent and~$\vec r_0=\vec 0$, this holds if and only if~$\{a,d,f\}=\{b,c,e\}$ as multisets, which forces~$a=c$ or~$b=d$. Conversely, $(\vec r_a-\vec r_b)-(\vec r_a-\vec r_d)=\vec r_d-\vec r_b$ and $(\vec r_a-\vec r_b)-(\vec r_c-\vec r_b)=\vec r_a-\vec r_c$. Hence the non-zero $\vec y_i=\vec r_{a_i}-\vec r_{b_i}$ pairwise share their first or their second index. If two of them share the first index, then a third one cannot share the second index with both, so it has the same first index as well; similarly for the second index. Thus all non-zero~$\vec y_i$ have a common first index~$a$, or all have a common second index~$b$. Adding the missing vectors~$\vec r_a-\vec r_c$ $(c\neq a)$, respectively~$\vec r_c-\vec r_b$ $(c\neq b)$, we obtain~$n$ vectors $\vec y_1,\ldots,\vec y_n$ with $\vec y_i-\vec y_j\in \Delta_+(A_{n-1})^*$ for all~$i,j$.

The set $\widetilde S\df\Delta_+(A_{n-1})\cup\{\vec y_i+\varepsilon_n:1\leq i\leq n\}$ contains~$S$, and it is of level one by~\ref{eq:closure1}: its layers below~$n$ are those of~$\Delta_+(A_{n-1})$, and the differences within its top layer lie in~$\Delta_+(A_{n-1})^*$. Its top layer has~$n$ elements, so by the case treated above~$\widetilde S$, and hence~$S$, is equivalent to a subset of~$\Delta_+(A_n)$.
\end{proof}

\subsection{Admissible polynomials}\label{sec:admissible}
Recall the lattice sums~$\Psi_{S,\vec k}$ and~$\Psi_{S,p}$ from the introduction. Since $Y(a)=\frac14(X(a)^2-1)$ by~\eqref{def:XY}, we have $\Psi_{S,\vec k}=\Psi_{S,p_{\vec k}}$ with
\[
p_{\vec k}\df \prod_{\vec v\in S}\big(\tfrac14(X_{\vec v}^2-1)\big)^{k_{\vec v}}.
\]
Since $X(a)=\mathrm{sgn}(a)+O\big(q^{|a|}\big)$ for~$a\in\Z\setminus\{0\}$, the coefficient of~$q^0$ in $p\big((X_{\vec v}(\vec m))_{\vec v}\big)$ equals $p\big((\mathrm{sgn}(\vec v\cdot\vec m))_{\vec v}\big)$, and more generally the low-order behaviour in~$q$ of the summand is governed by the signs $\mathrm{sgn}(\vec v\cdot\vec m)$. Note that for $a=0$ the series $X(a)$ is ill-defined. This motivates the following definition, which will make convergence, and the absence of poles in~$\bm\zeta$, manifest.
\begin{defn}\label{def:admissible}
A polynomial $p\in\C[X_{\vec v}:\vec v\in S]$ is \emph{admissible} (for~$S$) if for every $\vec d\in\R^n\setminus\{\vec 0\}$ the polynomial obtained from~$p$ by the substitution
\[
X_{\vec v}\mapsto \mathrm{sgn}(\vec v\cdot\vec d)\qquad\text{for all }\vec v\in S\text{ with }\vec v\cdot\vec d\neq0
\]
is the zero polynomial in the remaining variables $X_{\vec u}$ for $\vec u\in S$ with $\vec u\cdot\vec d=0$.
\end{defn}

The substitution depends only on the sign vector $(\mathrm{sgn}(\vec v\cdot\vec d))_{\vec v \in S}$, that is, on the face of the hyperplane arrangement $\{\vec v^\perp : \vec v\in S\}$ containing $\vec d$.
Hence admissibility amounts to finitely many conditions, one for each face other than $\{\vec 0\}$.
For $\vec d$ in an open chamber no variables remain, and the condition says that $p$ vanishes at the corresponding sign vector $(\pm1)_{\vec v\in S}$. For~$\vec d$ in a lower-dimensional face, the condition is a polynomial identity in the variables~$X_{\vec u}$ with $\vec u\cdot \vec d=0$.

Since different~$\bm m$ give different monomials~$\bm\zeta^{\bm m}$, the lattice sum~$\Psi_{S,p}$ always defines a formal series in~$q$ whose coefficients are formal Laurent series in~$\zeta_1,\ldots,\zeta_n$. The point of admissibility is that these coefficients are Laurent polynomials.
\begin{lem}\label{lem:admissible}
Let $p\in\C[X_{\vec v}:\vec v\in S]$.
\begin{enumerate}[label={\upshape(\arabic*)},leftmargin=*]
\item\label{lem:admissible:laurent} If~$p$ is admissible, then $\Psi_{S,p}$ can be expanded formally in~$\C[\zeta_1^{\pm1},\ldots,\zeta_n^{\pm1}][\![q]\!]$. 
\item\label{lem:admissible:conv} If~$p$ is admissible, then there is a constant $c_S>0$, depending only on~$S$, such that the lattice sum~$\Psi_{S,p}$ converges absolutely and locally uniformly on
\[
\mathcal T_S\df\big\{(\bm z,\tau)\in\C^n\times\H:|q|^{c_S}<|\zeta_j|<|q|^{-c_S}\text{ for }1\leq j\leq n\big\},
\]
and agrees with the formal expansion of~{\upshape(1)}.
\item\label{lem:admissible:prod} If~$B\subseteq S$ spans~$\R^n$, then $\prod_{\vec a\in B}(X_{\vec a}^2-1)$ is admissible.
\item\label{lem:admissible:rewrite} Let $p^\prime$ be obtained from $p$ by replacing, in each monomial, a factor $X_{\vec a}X_{\vec b}$ by $X_{\vec a}X_{\vec b-\vec a}-X_{\vec b}X_{\vec b-\vec a}+1$ with $\vec a,\vec b,\vec b-\vec a\in S$. Then $\Psi_{S,p^\prime}=\Psi_{S,p}$, and $p^\prime$ is admissible if $p$ is.
\end{enumerate}
\end{lem}
\begin{proof}
\noindent (1) Write $\|\vec m\|\df|m_1|+\cdots+|m_n|$ for $\vec m=(m_1,\ldots,m_n) \in\R^n$. We show that there is a constant $c_S>0$, depending only on~$S$, such that
\begin{equation}\label{eq:qorder}
p\big((X(\vec v\cdot\vec m))_{\vec v\in S}\big)\,\in\, q^{\lceil 2c_S\|\vec m\|\rceil}\,\C[\![q]\!]
\end{equation}
for all admissible~$p$ and all $\vec m\in\Z^n$ with $\vec v\cdot\vec m\neq0$ for $\vec v\in S$. Then only finitely many~$\vec m$ contribute to a given power of~$q$ in~$\Psi_{S,p}$, which is~\ref{lem:admissible:laurent}.

Let $\vec d\in\R^n$ with $\|\vec d\|=1$, and let
\[
c_{\vec d}\df\tfrac12\min\{|\vec v\cdot\vec d|:\vec v\in S,\ \vec v\cdot\vec d\neq0\}.
\]
Let~$\vec m$ be such that \[\Bigl|\vec v\cdot\Bigl(\frac{\vec m}{\|\vec m\|}-\vec d\Bigr)\Bigr|<c_{\vec d} \qquad \text{for all } \vec v\in S.\] 
Then $\mathrm{sgn}(\vec v\cdot\vec m)=\mathrm{sgn}(\vec v\cdot\vec d)$ and $|\vec v\cdot\vec m|>c_{\vec d}\|\vec m\|$ for all $\vec v\in S$ with $\vec v\cdot\vec d\neq0$. By admissibility of~$p$ for~$\vec d$, the coefficients of~$p$ as a polynomial in the~$X_{\vec u}$ with $\vec u\cdot\vec d=0$ vanish at $X_{\vec v}=\mathrm{sgn}(\vec v\cdot\vec d)$. As $X(a)\equiv\mathrm{sgn}(a)\pmod{q^{|a|}}$, it follows that
\[
p\big((X(\vec v\cdot\vec m))_{\vec v\in S}\big)\,\in\, q^{\lceil c_{\vec d}\|\vec m\|\rceil}\,\C[\![q]\!].
\]
The condition on~$\vec m$ defines an open neighbourhood of~$\vec d$ in the unit sphere $\{\vec x\in\R^n:\|\vec x\|=1\}$. Since the unit sphere is compact, finitely many of these neighbourhoods cover it. Hence~\eqref{eq:qorder} holds with~$2c_S$ the minimum of the corresponding~$c_{\vec d}$.

\noindent (2) Let~$c_S$ be as in the proof of~\ref{lem:admissible:laurent}, and let~$\tilde p$ be the polynomial obtained from~$p$ by replacing its coefficients by their absolute values. Fix $r\in(0,1)$. As $X(a)=\mathrm{sgn}(a)X(|a|)$, and~$X(|a|)$ is a power series in~$q$ with nonnegative coefficients whose value at~$r$ is $\frac{1+r^{|a|}}{1-r^{|a|}}\leq\frac{1+r}{1-r}$, we have
\[
\sum_{N\geq0}\Big|[q^N]\,p\big((X(\vec v\cdot\vec m))_{\vec v\in S}\big)\Big|\,r^N\,\leq\,\tilde p\Big(\frac{1+r}{1-r},\ldots,\frac{1+r}{1-r}\Big).
\]
Hence, by~\eqref{eq:qorder}, for $|q|<r$
\[
\sum_{N\geq0}\Big|[q^N]\,p\big((X(\vec v\cdot\vec m))_{\vec v\in S}\big)\Big|\,|q|^N\,\leq\,\Big(\frac{|q|}{r}\Big)^{2c_S\|\vec m\|}\,\tilde p\Big(\frac{1+r}{1-r},\ldots,\frac{1+r}{1-r}\Big).
\]
For $(\bm z,\tau)\in\mathcal T_S$ we have $|\zeta_j|^{m_j}\leq|q|^{-c_S|m_j|}$ for all~$j$, hence $|\bm\zeta^{\vec m}|\leq|q|^{-c_S\|\vec m\|}$. Taking $|q|^{1/2}<r<1$, the double series obtained by expanding each summand of~$\Psi_{S,p}$ in powers of~$q$ is therefore dominated by
\[
\tilde p\Big(\frac{1+r}{1-r},\ldots,\frac{1+r}{1-r}\Big)\sum_{\vec m\in\Z^n}\Big(\frac{|q|^{1/2}}{r}\Big)^{2c_S\|\vec m\|}\,<\,\infty,
\]
uniformly in a neighbourhood of~$(\bm z,\tau)$. Hence this double series converges absolutely and locally uniformly on~$\mathcal T_S$, so it may be summed in either order. This gives both the convergence of the lattice sum and its agreement with the formal expansion of~\ref{lem:admissible:laurent}.

\noindent (3) For~$\vec d\neq\vec0$, since~$B$ spans there is~$\vec a\in B$ with $\vec a\cdot\vec d\neq0$. Then the factor~$X_{\vec a}^2-1$ becomes $\mathrm{sgn}(\vec a\cdot\vec d)^2-1=0$ under the substitution, so the product vanishes identically.

\noindent(4) As $\vec a\cdot\vec m$, $\vec b\cdot\vec m$ and $(\vec b-\vec a)\cdot\vec m$ are non-zero in the range of summation, $\Psi_{S,p^\prime}=\Psi_{S,p}$ by~\eqref{eq:induction1}. For the second claim, write $p=X_{\vec a}X_{\vec b}\,f+g$, where no monomial of~$g$ is divisible by~$X_{\vec a}X_{\vec b}$, so that
\[
p^\prime\=p+Q\,f,\qquad Q\df -X_{\vec a}X_{\vec b} + X_{\vec a}X_{\vec b-\vec a}-X_{\vec b}X_{\vec b-\vec a}+1.
\]
Let~$p$ be admissible, let $\vec d\neq\vec0$, and let a bar denote the substitution in Definition~\ref{def:admissible}. It suffices to show that $\bar Q\,\bar f=0$. If $\vec a\cdot\vec d=\vec b\cdot\vec d=0$, then~$X_{\vec a}$ and~$X_{\vec b}$ are not substituted, so $0=\bar p=X_{\vec a}X_{\vec b}\,\bar f+\bar g$ gives $\bar f=0$. Otherwise at most one of $\vec a\cdot\vec d$, $\vec b\cdot\vec d$ and $(\vec b-\vec a)\cdot\vec d$ vanishes, and $\bar Q=0$. If none vanishes, this is the limit $q\to0$ of~\eqref{eq:induction1}. If one vanishes, it follows from $\bar X_{\vec b-\vec a}=\bar X_{\vec b}$, $\bar X_{\vec b-\vec a}=-\bar X_{\vec a}$ or $\bar X_{\vec a}=\bar X_{\vec b}$, respectively, together with $\bar X_{\vec v}^{\,2}=1$ for the substituted variables.
\end{proof}

In particular, $p_{\vec k}$ is admissible by~\ref{lem:admissible:prod} with~$B=S$, since~$S$ spans by condition~\ref{def:level1:span} of Definition~\ref{def:level1}, so that $\Psi_{S,\vec k}=\Psi_{S,p_{\vec k}}$ is a well-defined $q$-series by~\ref{lem:admissible:laurent}.

\subsection{The main theorem for level one}
Before turning to the proof of Theorem~\ref{thm:main1}, we record the identity that resolves the one boundary term in the reduction which is not manifestly convergent. Recall from the preliminaries the quasi-Jacobi form $\mathcal A(z;\tau)=\frac{1}{2\pi i}A(z;\tau)$ and its expansion~\eqref{eq:Aproduct}. Let~$\iota_{\mathfrak c}$ denote the expansion of a meromorphic function of~$z$ in the chamber $\mathfrak c=\{0<\Im(z)<\Im(\tau)\}$, i.e.,  in nonnegative powers of~$\zeta$.

 Let us view now $\zeta$ as a formal variable and let
\[
\delta(\zeta)\df \sum_{m\in\Z}\zeta^m \,\in\, \C[\![\zeta,\zeta^{-1}]\!]
\]
be the formal delta distribution (see for instance~\cite{LL}), characterized by $\delta(\zeta)f(\zeta)=\delta(\zeta)f(1)$ for every formal series~$f$ whose specialization~$f(1)$ is defined (e.g., a Laurent polynomial in~$\zeta$). Finally, write
\[
H(\zeta;q)\df \sum_{m \in \mathbb{Z}\setminus\{0\}} X(m)\zeta^m
\]
for the boundary series.
\begin{lem}\label{lem:regularization}
As formal series in~$\zeta$ with coefficients in~$\C[\![q]\!]$,
\begin{equation}\label{eq:Hidentity}
H(\zeta;q)\=-2\,\iota_{\mathfrak c}\mathcal A(z;\tau)-\delta(\zeta).
\end{equation}
Consequently, if~$Z=Z(\zeta;q)$ is a series whose coefficients are Laurent polynomials in~$\zeta$ and which vanishes at~$\zeta=1$, then
\begin{equation}\label{eq:HDidentity}
H(\zeta;q)\,Z(\zeta;q)\=-2\,\iota_{\mathfrak c}\mathcal A(z;\tau)\,Z(\zeta;q).
\end{equation}
\end{lem}
\begin{proof}
Using $X(m)=\sgn(m)+2\,\sgn(m)\frac{q^{|m|}}{1-q^{|m|}}$, we split
\[
H(\zeta;q)\=\sum_{m\geq1}\left(\zeta^m-\zeta^{-m}\right)+2\sum_{m\geq1}\left(\zeta^m-\zeta^{-m}\right)\frac{q^m}{1-q^m}.
\]
Expanding~\eqref{eq:Aproduct} in nonnegative powers of~$\zeta$, and using \[\frac{\zeta q^r}{1-\zeta q^r}\=\sum_{m\geq1}\zeta^m q^{mr} \quad \text{and} \quad \frac{\zeta^{-1}q^r}{1-\zeta^{-1}q^r}\=\sum_{m\geq1}\zeta^{-m}q^{mr},\]
 gives
\[
\iota_{\mathfrak c}\mathcal A(z;\tau)\=-\frac12-\sum_{m\geq1}\zeta^m-\sum_{m\geq1}\left(\zeta^m-\zeta^{-m}\right)\frac{q^m}{1-q^m}.
\]
Adding twice this to~$H$, the~$q$-dependent series cancel, and
\[
H(\zeta;q)+2\,\iota_{\mathfrak c}\mathcal A(z;\tau)\=\sum_{m\geq1}\left(\zeta^m-\zeta^{-m}\right)-1-2\sum_{m\geq1}\zeta^m\=-1-\sum_{m\geq1}\zeta^m-\sum_{m\geq1}\zeta^{-m}\=-\delta(\zeta),
\]
which is~\eqref{eq:Hidentity}. For~\eqref{eq:HDidentity}, the delta property gives $\delta(\zeta)Z(\zeta;q)=\delta(\zeta)Z(1;q)=0$, since~$Z$ vanishes at~$\zeta=1$. Substituting~\eqref{eq:Hidentity} yields the claim.
\end{proof}

We now recursively define the chamber in which the induction in the proof of Theorem~\ref{thm:main1} takes place. If~$n=1$, then $S=\{1\}$ and we let 
\[\mathcal C_{\{1\}}\df\{(z_1,\tau)\in\C\times\H: 0<\Im(z_1)<\Im(\tau)\}.\]
 If~$n>1$, for $\vec w=(w_1,\ldots,w_{n-1},1)\in S\setminus S_{n-1}$ let
\[
\Lambda_{\vec w}(z_1,\ldots,z_n)\df (z_1-w_1z_n,\ldots,z_{n-1}-w_{n-1}z_n).
\]
As $S_{n-1}$ is of level one by Lemma~\ref{lem:parallel}~\ref{lem:parallel:sub}, we may define
\begin{equation}\label{eq:chamber}
\mathcal C_S \df \left\{ (\bm z,\tau)\in\C^n\times\H: 0<\Im(z_n)<\Im(\tau), \; (\Lambda_{\vec w}(\bm z),\tau)\in
\mathcal C_{S_{n-1}} \text{ for all } \vec w\in S\setminus S_{n-1}
\right\}.
\end{equation}
The set $\mathcal C_S$ is non-empty, more precisely:
\begin{lem}\label{lem:chamber}
If~$S$ is of level one, then $\mathcal C_S\cap\mathcal T_S$ is open and non-empty.
\end{lem}
\begin{proof}
The chamber~$\mathcal C_S$ is defined in~\eqref{eq:chamber} by finitely many strict inequalities of the form $0<\ell(\Im\bm z)<\Im(\tau)$ with~$\ell$ linear, so it is open. One checks by induction on~$n$ that it contains every $(\bm z,\tau)$ with $\Im(z_j)=\epsilon^j\,\Im(\tau)$ for $1\leq j\leq n$, provided $\epsilon>0$ is small enough. For~$\epsilon$ small, these points also lie in~$\mathcal T_S$.
\end{proof}

We are now ready to prove Theorem~\ref{thm:main1}. As the complete proof involves several technicalities, the reader may consult Appendix~\ref{sec:A2reduction}, where we illustrate the algorithm suggested by the proof for~$S=\Delta_+(A_2)$ and~$k_{\vec v}=1$ for~$\vec v\in S$.

\begin{proof}[Proof of Theorem~\ref{thm:main1}]
By Lemma~\ref{lem:admissible}~\ref{lem:admissible:laurent}, $\Psi_{S,p}$ is a $q$-series with coefficients in $\C[\bm\zeta^{\pm1}]$. We prove by induction on~$n$ that, for every admissible~$p$, this $q$-series converges on the chamber~$\mathcal C_S$ of~\eqref{eq:chamber} and represents there the Fourier expansion of a polar quasi-Jacobi form~$\phi$ of level one and index~$0$. At the end of the proof we pass from~$\mathcal C_S$ to the region~$\mathcal T_S$ of the theorem (see Lemma~\ref{lem:admissible}~\ref{lem:admissible:conv}) and determine the weight. 

If~$n=0$, we have~$S=\emptyset$ and~$\Psi_{S,p}(q) = 1$, since we sum the empty product (equal to~$1$) over the empty tuple~$\vec{m}=()$. 

Let~$n\geq1$ and assume the claim for level-one sets of rank~$n-1$. Write
\[
T\df S\setminus S_{n-1}\=\{\vec v\in S:\iota(\vec v)=n\}
\]
for the set of \emph{top vectors}. By Lemma~\ref{lem:parallel}~\ref{lem:parallel:nonempty} it is non-empty.
If two distinct top vectors~$\vec{v},\vec{w}\in T$ occur in a monomial of~$p$ (i.e., $X_{\vec{v}}$ and~$X_{\vec{w}}$ divide this monomial), Lemma~\ref{lem:parallel}~\ref{lem:parallel:diff} gives $\vec{v}-\vec{w}\in S_{n-1}^*$; without loss of generality $\vec{v}-\vec{w}\in S_{n-1}$.
Then, we reduce their product by~\eqref{eq:induction1}, which leaves~$\Psi_{S,p}$ unchanged and~$p$ admissible by Lemma~\ref{lem:admissible}~\ref{lem:admissible:rewrite}.
Moreover, the total degree in the top variables strictly decreases under this rewriting, so it terminates.

We continue until no monomial contains two distinct top vectors, and group the monomials by the top variable they contain. Writing $X_{\vec{w}}^{[h]}\df x^{[h]}(X_{\vec w})$ with
\[
x^{[h]} \df \begin{cases} x & \text{if }h\=1, \\
x^{2j}-1 & \text{if }h\=2j, \\
x^{2j-1}-x & \text{if }h\=2j-1, j\geq 2,
\end{cases}
\]
we obtain, since $1$ and the~$x^{[h]}$ ($h\geq1$) form a basis of~$\C[x]$,
\[
p\big((X_{\vec v})_{\vec v\in S}\big)
\= \!\!\!\!\sum_{\vec{w}\in S\setminus S_{n-1}}\!\!\Big(p_{\vec w,0}\big((X_{\vec v})_{\vec v\in S_{n-1}}\big)+\sum_{h\geq 1}
    X_{\vec{w}}^{[h]}\,
    p_{\vec{w},h}\big((X_{\vec v})_{\vec v\in S_{n-1}}\big)\Big),
\]
where each coefficient~$p_{\vec{w},h}$ is a polynomial in the lower variables $(X_{\vec v})_{\vec v\in S_{n-1}}$ only.
Under the substitution $\vec{d}=\pm \varepsilon_n$ of the admissibility condition, each top variable becomes $\mathrm{sgn}(\pm w_n)=\pm 1$ (as $\kappa(\vec{w})=w_n=1$), so that $X_{\vec{w}}^{[1]}\mapsto\pm1$ and $X_{\vec{w}}^{[h]}\mapsto0$ for~$h\geq2$. Hence the admissibility of~$p$ for both signs gives
\begin{equation}\label{eq:sumzero1}
\sum_{\vec{w}\in T}p_{\vec{w},0}\=\sum_{\vec{w}\in T}p_{\vec{w},1}\=0,
\end{equation}
and by the first relation the terms~$p_{\vec w,0}$ cancel, so that
\begin{equation}\label{E:Before}
p\big((X_{\vec v})_{\vec v\in S}\big)
\= \!\!\!\!\sum_{\vec{w}\in S\setminus S_{n-1}}\sum_{h\geq 1}
    X_{\vec{w}}^{[h]}\,
    p_{\vec{w},h}\big((X_{\vec v})_{\vec v\in S_{n-1}}\big).
\end{equation}

Each coefficient~$p_{\vec{w},h}$ is admissible for~$S_{n-1}$. To see this, apply the admissibility of~$p$ in a direction $\vec{d}\in\vec{w}^{\perp}$: then $\vec{w}\cdot\vec{d}=0$, so~$X_{\vec{w}}$ remains a free variable; moreover any $\vec{\delta}\in\R^{n-1}\setminus\{\vec 0\}$ extends to such a $\vec{d}=(\vec{\delta},-\sum_{j<n}w_j\delta_j)\in\vec{w}^{\perp}$ with the same signs on~$S_{n-1}$. Under the substitution, the blocks~$X_{\vec{z}}^{[h]}$ with~$\vec{z}\neq\vec{w}$ become constants or stay free, and in both cases contribute only to the coefficient of~$X_{\vec{w}}^{0}$. For~$h\geq2$, only the block~$X_{\vec{w}}^{[h]}$ reaches degree~$h$ in~$X_{\vec{w}}$, so the coefficient of~$X_{\vec{w}}^{\,h}$ equals~$p_{\vec{w},h}$ evaluated at the signs of~$S_{n-1}$, which must vanish; hence~$p_{\vec{w},h}$ is admissible for~$h\geq2$. Finally, the coefficient of~$X_{\vec{w}}$ equals $p_{\vec{w},1}-\sum_{j\geq2}p_{\vec{w},2j-1}$ at these signs, and since the~$p_{\vec{w},2j-1}$ ($j\geq2$) already vanish there, also~$p_{\vec{w},1}$ is admissible.
 
Next, for each~$\vec{w}$, we make the change of variables
 \[\nu: m_n\mapsto  m_n -\sum_{j<n} w_j m_j.\]
 In particular, $\vec{w}\cdot \vec{m}\mapsto m_n$ and for~$\vec{u}\in S_{n-1}$, we have $\vec{u}\cdot \vec{m}\mapsto \vec{u}\cdot \vec{m}$ (as~$u_n=0$). Note that 
 \[\vec{\zeta}^{\nu(\vec{m})} \= \vec{\xi}^{\vec{m}}\zeta_n^{m_n} \quad \text{with} \quad \vec{\xi} \= (\zeta_1 \zeta_n^{-w_1}, \ldots, \zeta_{n-1} \zeta_n^{-w_{n-1}}) \quad\text{and}\quad \vec{\xi}^{\vec{m}}\df\prod_{j<n}\xi_j^{m_j}.\]
We now consider the effect of~$\nu$ on the range of summation~$\vec{m}\in \Z^n$ with $\vec{v}\cdot \vec{m}\neq 0$ for~$\vec{v}\in S$. 
Note that
\begin{equation*}
\nu(\vec{v}\cdot \vec{m}) \= \sum_{j<n} v_j m_j +v_n\nu(m_n) \= \sum_{j<n} (v_j-v_n w_j) m_j + v_n m_n.
\end{equation*}
Hence, the inequalities $\vec{v}\cdot \vec{m}\neq 0$ for all~$\vec{v}\in S_n$ turn into
\begin{align}\begin{split}\label{eq:casesineqS}
\sum_{j<n} (v_j-w_j) m_j + m_n &\neq 0 \qquad (\vec{v}\in S_n\setminus S_{n-1}) \\
\vec{v}\cdot \vec{m}&\neq 0 \qquad (\vec{v}\in S_{n-1})
\end{split}
\end{align}
under~$\nu$. Note that for~$\vec{w}=\vec{v}$, the first inequality turns into~$m_n\neq 0$.

Let~$\Omega$ be an (infinite) set and~$\mathcal P$ a finite set of properties that the elements of~$\Omega$ may or may not have. Write $\mathcal P(m)\subseteq \mathcal P$ for the set of properties of~$m\in \Omega$. The inclusion--exclusion principle (see~\cite[Section~2.1]{Sta}) states that
\begin{equation}\label{E:inex}
\sum_{\substack{m\in \Omega \\ \mathcal P(m)=\emptyset}}\!\!\!\! f(m) \= \sum_{j=0}^{|\mathcal P|} (-1)^j \!\!\!\!\sum_{\substack{I\subseteq \mathcal P\\ |I|=j}}\ \sum_{\substack{m\in \Omega \\ I\subseteq \mathcal P(m)}}\!\!\!\! f(m).
\end{equation}
In our application every~$m\in \Omega$ has at most one property (as we verify below). So only the summands~$j=0$ and~$j=1$ contribute, and~\eqref{E:inex} reduces to
\begin{equation}\label{E:inex2}
\sum_{\substack{m\in \Omega \\ \mathcal P(m)=\emptyset}}\!\!\!\! f(m) \= \sum_{m\in \Omega} f(m)-\sum_{\sigma\in \mathcal P}\ \sum_{\substack{m\in \Omega \\ \sigma\in \mathcal P(m)}}\!\!\!\! f(m).
\end{equation}

Now, let
\begin{align*}
\Omega_S\df \Omega&\df \left\{\bm{m}\in \mathbb{Z}^n : m_n\neq 0,\ \vec{u}\cdot \vec{m}\neq 0 \text{ for all } \vec{u}\in S_{n-1} \right\},
\end{align*}
and $\mathcal P_S\df \mathcal P\df T=S_n\backslash S_{n-1}$. For~$\bm m\in \Omega$ the property $\vec{v}\in \mathcal P$ is satisfied if
\[
{\sum_{j<n} (v_j-w_j) m_j + m_n = 0},
\]
that is,
\begin{equation}\label{eq:Cv}
m_n \= (\vec{w}-\vec{v})\cdot\vec{m}.
\end{equation}
Write $\Omega_{\vec{w}}(\vec{v})=\Omega(\vec{v})$ for the set of all elements in~$\Omega$ for which the above property holds. 
In particular, note that the property corresponding to~$\vec{w}=\vec{v}$ is never satisfied (since then both~$m_n=0$ and~$m_n\neq 0$). By definition, the set of~$\bm m$ satisfying~\eqref{eq:casesineqS} is the subset of~$\Omega$ for which no property is satisfied.  Moreover, observe that for all~$\bm m\in \Omega$ at most one property is satisfied. Namely, suppose that two properties~$\vec{v}$ and~$\vec{t}$ are satisfied. Then we have
\begin{equation*}
0\=\sum_{j<n} (v_j-w_j) m_j + m_n - \left(\sum_{j<n} (t_j-w_j) m_j + m_n\right) \=\sum_{j<n} (v_j-t_j)  m_j\=(\vec{v}-\vec{t})\cdot \vec{m}.
\end{equation*}
Since $\vec{v}, \vec{t}\in S_n\backslash S_{n-1}$, we have $\vec{v}-\vec{t}\in S_{n-1}^*$ by Lemma~\ref{lem:parallel}~\ref{lem:parallel:diff}. This contradicts the definition of the set~$\Omega$.

By~\eqref{E:Before} we have
\begin{align*}
\Psi_{S,p}(\bm\zeta;q)&\=\!\!\!\!\sum_{\substack{\bm m\in\Z^n \\ \vec v\cdot\vec m\neq0\ (\vec v\in S)}}\sum_{\substack{\vec{w}\in S\setminus S_{n-1}\\h\geq 1}}\!\!\!\! \bm\zeta^{\bm m} X_{\vec{w}}^{[h]}(\vec{m})\,p_{\vec{w},h}\big((X_{\vec v}(\vec m))_{\vec v\in S_{n-1}}\big). 
\end{align*}
Write $X^{[h]}(a)\df x^{[h]}(X(a))$ for~$a\in\Z\setminus\{0\}$.
By~\eqref{E:inex2}, this equals (for~$h=1$ as a formal series in~$\zeta_n$, see below)
\begin{align*}&\!\!\!\!\sum_{\substack{\vec{w}\in S\setminus S_{n-1}\\h\geq 1}} \left(
 \sum_{\bm{m}\in \Omega}  - \sum_{\vec{v}\in S_{n}\backslash S_{n-1}}\sum_{\substack{\bm{m}\in \Omega(\vec{v})}} \right)\zeta_n^{m_n} X^{[h]}(m_n)  \vec{\xi}^{\vec{m}} p_{\vec{w},h}\big((X_{\vec v}(\vec{m}))_{\vec v\in S_{n-1}}\big)
 \\
&\=\!\!\!\!\sum_{\substack{\vec{w}\in S\setminus S_{n-1}\\h\geq 1}} \left(
 \sum_{\bm{m}\in \Omega} \!\! \zeta_n^{m_n} X^{[h]}(m_n) \vec{\xi}^{\vec{m}} p_{\vec{w},h}\big((X_{\vec v}(\vec{m}))_{\vec v\in S_{n-1}}\big)\right.\\
&\hspace{80pt}\left.- \sum_{\vec{v}\in S_{n}\backslash S_{n-1}}\sum_{\substack{\bm{m}\in \Omega(\vec{v})}}\!\!\!\!\zeta_n^{(\vec{w}-\vec{v})\cdot\vec{m}} X^{[h]}_{\vec{w}-\vec{v}}(\vec{m})   \vec{\xi}^{\vec{m}} p_{\vec{w},h}\big((X_{\vec v}(\vec{m}))_{\vec v\in S_{n-1}}\big) \right) 
\\
&\=\!\!\!\!\sum_{\substack{\vec{w}\in S\setminus S_{n-1}\\h\geq 1}} \left(
 \sum_{\substack{m_n\in \mathbb{Z}\backslash\{0\} }}\!\!\!\! \zeta_n^{m_n} X^{[h]}(m_n) \sum_{\substack{\vec m'\in\Z^{n-1} \\ \vec u\cdot\vec m'\neq0\ (\vec u\in S_{n-1})}}\!\!\!\! \vec{\xi}^{\vec{m}'} p_{\vec{w},h}\big((X_{\vec u}(\vec{m}'))_{\vec u\in S_{n-1}}\big)\right.\\[-.75em]
&\hspace{64pt}\left.- \sum_{\substack{\vec{v}\in S_{n}\backslash S_{n-1}\\ \vec v\neq\vec w}}\ \sum_{\substack{\vec m'\in\Z^{n-1} \\ \vec u\cdot\vec m'\neq0\ (\vec u\in S_{n-1})}}\!\!\!\!\zeta_n^{(\vec{w}-\vec{v})\cdot\vec{m}'} X^{[h]}_{\vec{w}-\vec{v}}(\vec{m}')\,\vec{\xi}^{\vec{m}'} p_{\vec{w},h}\big((X_{\vec u}(\vec{m}'))_{\vec u\in S_{n-1}}\big) \right),
\end{align*}
where in the last step we identified~$\Omega(\vec v)$, for~$\vec v\neq\vec w$, with the set of $\vec m'=(m_1,\ldots,m_{n-1})$ such that~$\vec u\cdot\vec m'\neq0$ for~$\vec u\in S_{n-1}$ (then~$m_n=(\vec w-\vec v)\cdot\vec m'\neq0$ automatically, as~$\vec w-\vec v\in S_{n-1}^*\setminus\{\vec 0\}$), and omitted~$\vec v=\vec w$, as~$\Omega(\vec w)=\emptyset$.

Note that the~$(n-1)$-dimensional pieces
\begin{align}\label{eq:2forms}
\vec{\xi}^{\vec{m}} p_{\vec{w},h}\big((X_{\vec u}(\vec m))_{\vec u\in S_{n-1}}\big)\quad\text{and}\quad\vec{\xi}^{\vec{m}}\, X^{[h]}_{\vec{w}-\vec{v}}(\vec{m})\, p_{\vec{w},h}\big((X_{\vec u}(\vec m))_{\vec u\in S_{n-1}}\big)
\end{align}
 are again of the form of the induction hypothesis for~$n-1$. Indeed $\vec{w}-\vec{v}\in S^*_{n-1}$ by Lemma~\ref{lem:parallel}~\ref{lem:parallel:diff}, so $X^{[h]}_{\vec{w}-\vec{v}}(\vec{m})$ is a polynomial in the variables $(X_{\vec u})_{\vec u\in S_{n-1}}$. Since~$p_{\vec{w},h}$ is admissible and the product of an admissible polynomial with any polynomial is admissible, both pieces are admissible for~$S_{n-1}$ (after absorbing the shift $\zeta_n^{(\vec{w}-\vec{v})\cdot\vec{m}}$ into the elliptic variables).

The other factor is the one-dimensional sum
\begin{equation}\label{E:ReductionFactor}
\mathcal J_h(z_n;\tau)\df\sum_{m_n \in \mathbb{Z}\setminus\{0\}} X^{[h]}(m_n)\zeta_n^{m_n} \= \sum_{m> 0} X^{[h]}(m)\big(\zeta_n^m+(-1)^{h}\zeta_n^{-m}\big).
\end{equation}
For~$h=1$ this is the formal series~$H(\zeta_n;q)$ of Lemma~\ref{lem:regularization}, which is treated below. For~$h\geq2$ the coefficient of~$q^0$ vanishes, because $X^{[h]}(m)=X(m)^{2j}-1$ or~$X(m)^{2j-1}-X(m)$ vanishes for~$X(m)=\pm1$; hence these sums converge. Since $X(m)^2-1=4Y(m)$ by~\eqref{def:XY}, $X^{[2j]}(m)$ is a polynomial in~$Y(m)$ divisible by~$Y(m)$, and $X^{[2j-1]}(m)=X(m)\big(X(m)^{2j-2}-1\big)$ is~$X(m)$ times such a polynomial; hence~\eqref{E:ReductionFactor} is a~$\C$-linear combination of the series $\sum_{m\neq0}\zeta_n^{m} Y(m)^{r}$ and $\sum_{m \in \mathbb{Z}\setminus\{0\}} X(m)Y(m)^{r}\zeta_n^{m}$ with~$r\geq1$. These two families are polar quasi-Jacobi forms of level one by Lemma~\ref{lem:XYjacobi}, applied with~$z=z_n$, and as~$r\leq\frac h2$, respectively $r\leq\frac{h-1}2$, the form~$\mathcal J_h$ ($h\geq2$) has mixed weight at most~$h$.

Let
$
\bm m=(\bm m_0,m_n),
\bm z=(\bm z_0,z_n),
\bm\zeta=(\bm\zeta_0,\zeta_n),$
and $\vec w=(\vec w_0,1)$ for~$\vec w\in T$,
where
$
\bm m_0=(m_1,\ldots,m_{n-1}),
\bm z_0=(z_1,\ldots,z_{n-1}),
\bm\zeta_0=(\zeta_1,\ldots,\zeta_{n-1}),
\vec w_0=(w_1,\ldots,w_{n-1}).
$
For~$\vec u\in T$ put $\bm\xi^{(\vec u)}\df(\zeta_1\zeta_n^{-u_1},\ldots,\zeta_{n-1}\zeta_n^{-u_{n-1}})$, so that $\bm\xi=\bm\xi^{(\vec w)}$.
Combining what is said above, we obtain the reduction formula
\begin{align} \label{total}
\Psi_{S,X_{\vec w}^{[h]}p_{\vec w,h}}
(\bm\zeta;q)\=&
\mathcal J_h(z_n;\tau)\,
\Psi_{S_{n-1},p_{\vec w,h}}
(\bm\xi^{(\vec w)};q) \nonumber \\
&{}-
\sum_{\substack{\vec v\in T\\\vec v\neq\vec w}} \Psi_{S_{n-1},
X_{\vec w-\vec v}^{[h]}p_{\vec w,h}}
(\bm\xi^{(\vec v)};q).
\end{align}

By~\eqref{E:Before}, summing~\eqref{total} over~$\vec w\in T$ and~$h\geq1$ gives~$\Psi_{S,p}$. By~\eqref{eq:2forms} and the induction hypothesis, all series~$\Psi_{S_{n-1},\cdot}$ in~\eqref{total} are Fourier expansions of quasi-Jacobi forms of level one, also after the substitutions $\bm\zeta\mapsto\bm\xi^{(\vec w)}$ and $\bm\zeta\mapsto\bm\xi^{(\vec v)}$, which are the integral-linear substitutions $\bm z_0\mapsto\bm z_0-z_n\vec w_0$ and $\bm z_0\mapsto\bm z_0-z_n\vec v_0$ of the elliptic variables. Together with~\eqref{E:ReductionFactor}, every term of~\eqref{total} is a quasi-Jacobi form, except for the first term with~$h=1$, where $\mathcal J_1=H$. Summed over~$\vec w\in T$, these exceptional terms give $H(\zeta_n;q)\,Z(\bm\zeta;q)$ with
\[
Z(\bm\zeta;q)\df \sum_{\vec w\in T}\Psi_{S_{n-1},p_{\vec w,1}}(\bm\xi^{(\vec w)};q).
\]
The coefficients of~$Z$ as a $q$-series are Laurent polynomials in~$\bm\zeta$ by Lemma~\ref{lem:admissible}~\ref{lem:admissible:laurent}, and $Z(\bm\zeta_0,1;q)=\Psi_{S_{n-1},\sum_{\vec w}p_{\vec w,1}}(\bm\zeta_0;q)=0$ by~\eqref{eq:sumzero1}. Hence Lemma~\ref{lem:regularization} gives
\[
H(\zeta_n;q)\,Z(\bm\zeta;q)\=-2\,\iota_{\mathfrak c}\mathcal A(z_n;\tau)\,Z(\bm\zeta;q),
\]
and the right-hand side is quasi-Jacobi, as~$\mathcal A$ is a quasi-Jacobi form of weight one and~$Z$ is quasi-Jacobi by the induction hypothesis; the operator~$\iota_{\mathfrak c}$ only records the chamber~$\mathcal C_S$ in which the product is expanded.

It remains to check convergence. Each factor in~\eqref{total} is a $q$-series converging on~$\mathcal C_S$: the forms in~$z_n$ by Lemma~\ref{lem:XYjacobi}, as $0<\Im(z_n)<\Im(\tau)$, and the forms $\Psi_{S_{n-1},\cdot}$ at $\Lambda_{\vec w}(\bm z)$ by the induction hypothesis, as $(\Lambda_{\vec w}(\bm z),\tau)\in\mathcal C_{S_{n-1}}$. Hence the $q$-series $\Psi_{S,p}$ converges on~$\mathcal C_S$, and~\eqref{total} holds there as an identity of functions.

Thus, on~$\mathcal C_S$, $\Psi_{S,p}$ is the Fourier expansion of a quasi-Jacobi form, which we denote by~$\phi$. Its index is~$0$, as $\mathcal Y_r$, $\mathcal X_r$ and~$\mathcal A$ have index~$0$ and products and the substitutions $\bm z_0\mapsto\bm z_0-z_n\vec w_0$ preserve this. This completes the induction.

Next, we show by the same induction that~$\Psi_{S,p}$ has mixed weight at most the total degree~$d$ of~$p$ in the variables~$X_{\vec v}$. The rewriting by~\eqref{eq:induction1} does not increase the degree, $\mathcal J_h$ has mixed weight at most~$h$, the degree of~$X^{[h]}$ (see after~\eqref{E:ReductionFactor}), and~$\mathcal A$ has weight one. In each term of~\eqref{total} the degrees of the factors add up to at most~$d$, and the substitution $\bm z_0\mapsto\bm z_0-z_n\vec w_0$ does not change weights.

Finally, we pass from~$\mathcal C_S$ to~$\mathcal T_S$. By Lemma~\ref{lem:admissible}~\ref{lem:admissible:conv}, the lattice sum~$\Psi_{S,p}$ is a holomorphic function on~$\mathcal T_S$. It agrees with~$\phi$ on~$\mathcal C_S\cap\mathcal T_S$, since there both equal the $q$-series~$\Psi_{S,p}$, and this set is open and non-empty by Lemma~\ref{lem:chamber}. As~$\mathcal T_S$ is convex, hence connected, $\phi$ and the lattice sum agree on all of~$\mathcal T_S$. In particular, $\phi$ is holomorphic on~$\mathcal T_S$.
\end{proof}

The latter result enables us to prove one of the main findings of our paper.
\begin{cor}\label{cor:polyinsertion}
For a set~$S$ of level~one, $\vec{\ell}\in \N_0^n$ and~$k_{\vec v}\in\N$, the~$q$-series
\begin{equation}\label{eq:polyinsertion}
\Psi_{S,\vec{\ell},\vec{k}}(q)\df\!\!\!\!\sum_{\substack{{\bm m}\in \Z^n \\ \vec{v}\cdot \vec{m}\neq 0 \text{ for }\vec{v}\in S}}\!\!\!\!{\bm m}^{\bm \ell} \prod_{\vec{v} \in S} \left(\frac{q^{\vec{v}\cdot \vec{m}}}{\left(1-q^{\vec{v}\cdot \vec{m}}\right)^{\!2}}\right)^{\!k_{\vec{v}}}
\end{equation}
is a quasimodular form on~$\SL_2(\Z)$ of mixed weight at most $2\sum_{\vec{v}\in S}k_{\vec{v}}+|\bm\ell|$, where $|\bm\ell|\df \ell_1+\cdots+\ell_n$.
\end{cor}
\begin{proof}
By Theorem~\ref{thm:main1}, the lattice sum $\Psi_{S,\vec k}=\Psi_{S,p_{\vec k}}$ converges absolutely and locally uniformly on the region~$\mathcal T_S$ of Lemma~\ref{lem:admissible}~\ref{lem:admissible:conv}, which contains all $(\bm z,\tau)$ with $\bm z\in\R^n$, and equals there a polar quasi-Jacobi form~$\phi$ of level one and index~$0$, of mixed weight at most $2\sum_{\vec v\in S}k_{\vec v}$, which is holomorphic on this region. Let $D^{\bm\ell}\df\prod_{j}\big(\tfrac{1}{2\pi i}\tfrac{\partial}{\partial z_j}\big)^{\ell_j}$. By locally uniform convergence, $D^{\bm\ell}$ may be applied to the lattice sum termwise, which multiplies the summand by~$\bm m^{\bm\ell}$. Evaluating at $\bm z=\bm 0$ gives $D^{\bm\ell}\phi(\bm 0;\tau)=\Psi_{S,\vec\ell,\vec k}(q)$. On the other hand, $D^{\bm\ell}\phi$ is a polar quasi-Jacobi form of level one and index~$0$, of mixed weight at most $2\sum_{\vec v\in S}k_{\vec v}+|\bm\ell|$, and it is holomorphic at $\bm z=\bm 0$. Its value at $\bm z=\bm 0$ is the constant term of its Laurent expansion there, hence a quasimodular form for~$\SL_2(\Z)$ of the same mixed weight by the properties in Section~\ref{sec:prelim}.
\end{proof}

Although the proof of Corollary~\ref{cor:polyinsertion} is constructive, obtaining an explicit expression for~$\Psi_{S,\vec\ell,\vec k}$ in terms of familiar quasimodular forms is quite delicate. We illustrate this in Appendix~\ref{sec:A2quasimod} for $S=\Delta_+(A_2)$ and~$k_{\vec v}=1$ for~$\vec v\in S$.

\subsection{Quasimodularity for type \texorpdfstring{$A$}{A}}
Recall that the root lattice associated to $\mathfrak{g}=\mathfrak{sl}_n(\mathbb{C})$ is~$A_{n-1}$. We now record several corollaries
for this root system, some of which were conjectured in the literature. The next result is part of~\cite[Conjecture~3.10]{Milas}. 

\begin{cor} \label{k.unif}
For $\mathfrak{g}=\mathfrak{sl}_n(\mathbb{C})$ the corresponding function $F_{\mathfrak{g},\ell,k}$ is quasimodular if $k\equiv \ell \equiv 0 \pmod 2$.
\end{cor}
\begin{proof}
Let $S\df \Delta_+(A_{n-1})$, a level-one set spanning~$\R^{n-1}$ (Example~\ref{ex:An}). Recall the cone form
\begin{equation}\label{eq:coneA}
F_{\mathfrak g,\ell,k}(q)\=\sum_{\lambda\in P\cap C^0}\dim(V_{\lambda-\rho})^\ell\,\prod_{\alpha\in\Delta_+}\frac{q^{\frac{k}{2}(\lambda,\alpha)}}{\bigl(1-q^{(\lambda,\alpha)}\bigr)^{\!k}} .
\end{equation}
Identify~$\alpha\in\Delta_+$ with the vector~$\vec v\in S$ representing it and write~$\vec m$ for the coordinates of~$\lambda$ with respect to~$\omega_1,\ldots,\omega_{n-1}$, so that $(\lambda,\alpha)=\vec v\cdot\vec m$. Since~$k$ is even, the factor indexed by~$\alpha$ in the product in~\eqref{eq:coneA} equals $Y(\vec v\cdot\vec m)^{k/2}$ by~\eqref{def:XY}, so this product equals $\prod_{\vec v\in S}Y(\vec v\cdot\vec m)^{k_{\vec v}}$ with $k_{\vec v}\df k/2$. By the Weyl dimension formula~\eqref{weyl-dim},
\[
\dim(V_{\lambda-\rho})^\ell\=\prod_{\alpha\in\Delta_+}\frac{(\lambda,\alpha)^{\ell}}{(\rho,\alpha)^{\ell}}
\]
is a homogeneous polynomial in~$\vec m$. Write it as $\sum_{\bm\ell}s_{\bm\ell}\,\bm m^{\bm\ell}$. Since~$k\equiv\ell\pmod 2$, the symmetrization identity $\widehat F_{\mathfrak g,\ell,k}=|W|\,F_{\mathfrak g,\ell,k}$ of Corollary~\ref{cor:symm} expresses the dominant-chamber sum as~$\tfrac{1}{|W|}$ times the full-lattice sum over $P\setminus \mathcal H$. Concretely,
\[
F_{\mathfrak g,\ell,k}(q)\=\frac{1}{|W|}\sum_{\bm\ell}s_{\bm\ell}\,\Psi_{S,\bm\ell,\vec k}(q),
\]
which is quasimodular by Corollary~\ref{cor:polyinsertion}.
\end{proof}

The next result proves~\cite[Conjecture~1.5]{Milas} as a special case.

\begin{cor} \label{k.var}
For $\mathfrak{g}=\mathfrak{sl}_n(\mathbb{C})$, the function $F^W_{\mathfrak{g},\ell,\mathbf{k}}$ is quasimodular if $k_\alpha \equiv \ell \equiv 0 \pmod 2$. 
\end{cor}
\begin{proof}
By Theorem~\ref{thm.var} we have $F^W_{\mathfrak g,\ell,\mathbf k}(\bm\zeta;q)=\widehat F_{\mathfrak g,\ell,\mathbf k}(\bm\zeta;q)$. Specializing to~$\zeta_j=1$ gives
\begin{equation}\label{eq:FWfull}
F^W_{\mathfrak g,\ell,\mathbf k}(q)\=\sum_{\lambda\in P\setminus \mathcal H}\dim(V_{\lambda-\rho})^\ell\prod_{\alpha\in\Delta_+}\frac{q^{\frac{k_\alpha}{2}(\lambda,\alpha)}}{\bigl(1-q^{(\lambda,\alpha)}\bigr)^{\!k_\alpha}},
\end{equation}
a sum over the full lattice minus the hyperplanes. As each~$k_\alpha$ is even, the factor indexed by~$\alpha$ in the product in~\eqref{eq:FWfull} equals $Y(\vec v\cdot\vec m)^{k_\alpha/2}$ by~\eqref{def:XY}, where $\vec v\in S\df \Delta_+(A_{n-1})$ represents~$\alpha$ and~$\vec m$ collects the coordinates of~$\lambda$. Expanding $\dim(V_{\lambda-\rho})^\ell$ as a polynomial in~$\vec m$ via the Weyl dimension formula, exactly as in the proof of Corollary~\ref{k.unif}, exhibits $F^W_{\mathfrak g,\ell,\mathbf k}(q)$ as a~$\C$-linear combination of $\Psi_{S,\bm\ell,\vec k}(q)$ with $k_{\vec v}=\frac{k_\alpha}{2}$. By Corollary~\ref{cor:polyinsertion} these are quasimodular, hence so is $F^W_{\mathfrak g,\ell,\mathbf k}$. 
\end{proof}

Finally, we have a quasimodularity result for the~$W$-symmetrized series associated to general level-one sets.

\begin{cor}\label{cor:levelonequasimod} Let~$\mathfrak g=\mathfrak{sl}_n(\C)$, let~$S$ be a level-one set spanning the same space as~$\Delta_+$, and let~$k$ and~$\ell$ be even. Then $F^W_{S,\mathfrak{g},\ell,k}$ belongs to the ring of quasimodular forms.
\end{cor}
\begin{proof}
In the coordinates of Section~\ref{sec:prelim} we have $P=\Z^{n-1}$ and $(\lambda,\alpha)=\vec v\cdot\vec m$, where~$\vec v$ represents~$\alpha$, and $\Delta_+=\Delta_+(A_{n-1})$ is of level one by Example~\ref{ex:An}. Since~$k$ is even, \eqref{SW-cone} gives
\[
F^W_{S,\mathfrak g,\ell,k}(q)\=\!\!\!\!\sum_{\substack{\vec m\in\Z^{n-1}\\ \vec v\cdot\vec m\neq0\ (\vec v\in\Delta_+)}}\!\!\!\!\prod_{\vec v\in S}(\vec v\cdot\vec m)^{\ell}\;p\big((X(\vec v\cdot\vec m))_{\vec v\in\Delta_+}\big),\qquad p\df\prod_{\vec v\in S}Y_{\vec v}^{k/2},
\]
with $Y_{\vec v}=\frac14(X_{\vec v}^2-1)$. As~$S$ spans, $p$ is divisible by the admissible polynomial $\prod_{\vec v\in S}(X_{\vec v}^2-1)$ of Lemma~\ref{lem:admissible}~\ref{lem:admissible:prod}, so~$p$ is admissible. Hence, as in the proof of Corollary~\ref{cor:polyinsertion}, $F^W_{S,\mathfrak g,\ell,k}(q)$ is obtained from the form~$\Psi_{\Delta_+,p}$ of Theorem~\ref{thm:main1} by a differential operator in~$\bm z$ followed by setting $\bm z=\bm 0$. It is therefore quasimodular.
\end{proof}

\section{Level two}\label{sec:leveltwo}

In this section, we prove the main theorem in case the resulting form is of level~two. Throughout we use the coordinates~\eqref{eq:epscoord}, so that~$\iota(\bm{v})$ is the largest index~$i$ with~$v_i\neq0$ and $\kappa(\bm{v})=v_{\iota(\bm{v})}$.

The sets of vectors associated to the root systems~$B_n$, $C_n$, and~$D_n$ are not of level one. For~$B_n$ the first and the third condition of Definition~\ref{def:level1} do hold, but the second one fails, because the difference of the two positive roots $\varepsilon_j+\varepsilon_i$ and $\varepsilon_j-\varepsilon_i$ equals~$2\varepsilon_i$, which is not a root, but twice one. The same difference is the obstruction for~$C_n$ and~$D_n$. For~$C_n$ also the first condition fails, as the long roots~$2\varepsilon_i$ are not primitive, so that their last non-zero coordinate equals~$2$ rather than~$1$.

In this section we show that this obstruction can be accommodated, at the cost of passing from~$\SL_2(\Z)$ to~$\Gamma_0(2)$. The reason is that both~$X(2a)$ and~$Y(2a)$ are Laurent polynomials in~$X(a)$, so that a difference which is twice a vector~$\bm{u}$ can still be expressed in the variable~$X_{\bm{u}}$, provided negative powers of~$X_{\bm{u}}$ are allowed. Keeping track of the vectors for which they are allowed is what distinguishes the definition below from Definition~\ref{def:level1}. Besides this, for~$C_n$ and~$D_n$ the weight lattice is not~$\Z^n$ but the union of~$\Z^n$ and~$\Z^n+\frac12\bm{1}$ by~\eqref{eq:weightlattice}, so that we also let the sums range over cosets of~$\Z^n$.

\subsection{Pairs of level two}
Recall from Section~\ref{sec:levelone} that a set~$S$ with~$\kappa(\bm v)=1$ for all~$\bm v\in S$ satisfies condition~\ref{def:level1:diff} of Definition~\ref{def:level1} if and only if~\ref{eq:closure1} holds. For~$B_n$ condition~\ref{eq:closure1} fails only because of the differences $(\varepsilon_j+\varepsilon_i)-(\varepsilon_j-\varepsilon_i)=2\varepsilon_i$. In the proof of Theorem~\ref{thm:main1} such a difference enters through the function $X(2\bm{u}\cdot\bm{m})$, and since
\begin{equation}\label{eq:double}
X(2a)\=\frac{X(a)^2+1}{2X(a)}, \qquad Y(2a)\=\frac{(X(a)^2-1)^2}{16X(a)^2}\=\frac{Y(a)^2}{X(a)^2},
\end{equation}
this amounts to allowing negative powers of~$X(\bm{u}\cdot\bm{m})$. Note, however, that $X(2a)^{-1}=\frac{2X(a)}{X(a)^2+1}$ is not a Laurent polynomial in~$X(a)$, as it has poles at~$q^a=\pm i$. So negative powers may only be allowed on vectors that never occur in a \emph{doubled pair}, i.e., a pair of vectors in the same layer whose difference is twice a non-zero vector (such as~$\varepsilon_j+\varepsilon_i$ and~$\varepsilon_j-\varepsilon_i$). This motivates the following definition, in which~$U$ is the set of vectors on which negative powers are allowed.
\begin{defn}\label{def:level2}
Let $U\subseteq S\subseteq\Z^n$. We call the pair~$(S,U)$ of \emph{level two} if the following hold:
\begin{enumerate}[label = (\roman*),leftmargin=*]
\item\label{def:level2:kappa} $\kappa(\bm{v})=1$ for all~$\bm{v}\in S$;
\item\label{def:level2:diff} for all~$\bm{v},\bm{w}\in S$ with $\iota(\bm{v})=\iota(\bm{w})$ we have $\bm{v}-\bm{w}\in S^*\cup 2U^*$, and~$\bm{v}-\bm{w}\in U^*$ if moreover~$\bm{v}\in U$;
\item\label{def:level2:span} the dimension of the vector space generated by~$S$ is~$n$.
\end{enumerate}
\end{defn}
\begin{rems}\hspace{0cm}
\begin{enumerate}[label = \rm(\arabic*),leftmargin=*]
\item A set~$S$ is of level one if and only if~$(S,\emptyset)$ is of level two.
\item We have $S^*\cap2U^*=\{\bm{0}\}$, as the non-zero elements of~$S^*$ have $|\kappa|=1$ and those of~$2U^*$ have $|\kappa|=2$.
\item An element of~$U$ is never a member of a doubled pair: if~$\bm{v}\in U$ and $\bm{v}-\bm{w}=2\bm{u}$ with $\iota(\bm{v})=\iota(\bm{w})$, then~$2\bm{u}\in U^*$ by~\ref{def:level2:diff}, which is impossible since~$|\kappa(2\bm{u})|=2$. By symmetry the same holds for~$\bm{w}$.
\item Both conditions in~\ref{def:level2:diff} concern one layer at a time. Hence~$(S_r,U_r)$ with~$U_r\df U\cap S_r$ is again of level two, and Lemma~\ref{lem:parallel} together with its proof remains valid, part~\ref{lem:parallel:diff} now reading $\bm{v}-\bm{w}\in S_{r-1}^*\cup2U_{r-1}^*$.
\end{enumerate}
\end{rems}
We call pairs $(S,U)$ and $(S',U')$ \emph{equivalent} if there is $M\in\mathrm T_n(\Z)$ with $MS^*=S'^*$ and $MU^*=U'^*$. As for sets of level one, $(S,U)$ is of level two if and only if $(S',U')$ is.

\begin{example}[Root system~$B_n$]\label{ex:Bn}
Let $S_B\df \{\varepsilon_j\pm\varepsilon_i: 1\leq i<j\leq n\}\cup\{\varepsilon_j:1\leq j\leq n\}$ and $U_B\df \{\varepsilon_j:1\leq j\leq n\}$. Then~$(S_B,U_B)$ is of level two. Namely, all elements of~$S_B$ have~$\kappa=1$, the differences of two elements of layer~$j$ are $\pm(\varepsilon_i-\varepsilon_k)$, $\pm(\varepsilon_i+\varepsilon_k)$, $\pm\varepsilon_i$, and, for~$i=k$, the vectors~$\pm2\varepsilon_i$ with~$\varepsilon_i\in U_B$; the differences $\varepsilon_j-(\varepsilon_j\pm\varepsilon_i)=\mp\varepsilon_i$ lie in~$U_B^*$; and~$S_B$ contains the standard basis. The hyperplanes~$H_{\bm{v}}$, $\bm{v}\in S_B$, are the root hyperplanes of~$B_n$, and also those of~$C_n$.
\end{example}
\begin{rem}
By Proposition~\ref{prop:An}, every set of level one is, up to equivalence, a subset of~$\Delta_+(A_n)$, so that level one corresponds to the root system~$A_n$. 
It is not difficult to see that every pair~$(S,U)$ of level two of rank two such that~$S$ is not of level one is equivalent to a pair~$(S',U')$ with $S'\subseteq S_B$ and $U'\subseteq U_B$.  However, we do not know whether an analogous result holds in rank~$n$, that is, whether every pair of level two of rank~$n$ can be associated to a classical root system of rank~$n$.
\end{rem}

\subsection{Laurent polynomials and the associated quasi-Jacobi forms}\label{subsec:laurent}
For a pair~$(S,U)$ of level two we consider the commutative ring 
\begin{equation} \label{def:rsu}
R_{S,U}\df \C[X_{\bm{v}}:\bm{v}\in S][X_{\bm{u}}^{-1}:\bm{u}\in U]
\end{equation}
of polynomials in the variables~$X_{\bm{v}}$ in which negative exponents are allowed exactly for the variables indexed by~$U$. Definition~\ref{def:admissible} of admissibility applies unchanged to~$p\in R_{S,U}$, since only the values~$\pm1$ are substituted, and so does Lemma~\ref{lem:admissible}, since also $X(a)^{-1}\equiv\mathrm{sgn}(a)\pmod{q^{|a|}}$ and~$X(|a|)^{-1}$ has coefficients of the same absolute value as~$X(|a|)$. In the reduction we need, besides~\eqref{eq:induction1}, its variants with inverted variables.
The following identities are easily checked.
\begin{lem}\label{lem:laurentid}
For~$a\neq b$ we have
\begin{align}
X(a)^{-1}X(b)&\=X(a)^{-1}X(b-a)^{-1}-X(b)X(b-a)^{-1}+1, \label{eq:laurentid1}\\
X(a)^{-1}X(b)^{-1}&\=X(a)^{-1}X(b-a)-X(b)^{-1}X(b-a)+1. \label{eq:laurentid2}
\end{align}
\end{lem}

Note that~$X(a)X(b)^{-1}$ is obtained from~\eqref{eq:laurentid1} by interchanging~$a$ and~$b$, since~$X(a-b)=-X(b-a)$. Note also that both identities remain valid if~$X(a)$, $X(b)$, $X(b-a)$ are replaced by~$\sgn(a)$, $\sgn(b)$, $\sgn(b-a)$ whenever the corresponding argument is nonzero, so that admissibility is preserved when a polynomial is rewritten by them.

Next we describe the forms associated to Laurent polynomials with negative exponents. The substitution~$q^a\mapsto-q^a$ maps~$X(a)$ to~$X(a)^{-1}$ and~$Y(a)$ to
\[
Y^-(a)\df \tfrac14\left(X(a)^{-2}-1\right)\=-\frac{q^a}{(1+q^a)^2},
\]
and Lemma~\ref{lem:Lshape} carries over: with the same polynomials~$f_r,g_r$ we have
\[
\sum_{m\geq1}(-1)^mm^{2r-1}q^{am}\=f_r\big(Y^-(a)\big), \qquad \sum_{m\geq1}(-1)^mm^{2r}q^{am}\=X(a)^{-1}g_r\big(Y^-(a)\big).
\] In terms of Zagier's function~$F$, these series arise by expanding at the two-torsion point~$z_1=\frac12$ instead of at~$z_1=0$. Set
\begin{equation}\label{eq:Ahalf}
\mathcal{A}_{\frac12}(z;\tau)\df \frac{1}{2\pi i}F\left(\tfrac12,z;\tau\right)\=\frac12\coth(\pi iz)+\sum_{n\geq1}\left(\zeta^n-\zeta^{-n}\right)\frac{q^n}{1+q^n},
\end{equation}
where we used~\eqref{eq:FZagier} and $\coth(\frac{\pi i}{2})=0$. Since $\frac{c\tau+d}{2}\equiv\frac12\pmod{\Z\tau+\Z}$ for~$c$ even and~$d$ odd, the torsion point~$\frac12$ is fixed by~$\Gamma_0(2)$, and $\mathcal{A}_{\frac12}$ is a polar Jacobi form of weight one on~$\Gamma_0(2)$; here the elliptic transformation law holds up to a character, since $\mathcal{A}_{\frac12}(z+\tau;\tau)=-\mathcal{A}_{\frac12}(z;\tau)$. Likewise the Taylor coefficients of $\frac{1}{2\pi i}F(z_1,z;\tau)$ at~$z_1=\frac12$, which up to constants are the series $\sum_{m\geq1}(\zeta^m\mp\zeta^{-m})\sum_{r\geq1}(-1)^rr^sq^{mr}$, are polar quasi-Jacobi forms of index~$0$ on~$\Gamma_0(2)$, and exactly as in the proof of Lemma~\ref{lem:XYjacobi} we conclude that
\[
\sum_{m\in\Z\setminus\{0\}}Y^-(m)^r\zeta^m \qquad\text{and}\qquad \sum_{m\in\Z\setminus\{0\}}X(m)^{-1}Y^-(m)^r\zeta^m \qquad (r\in\N)
\]
are polar quasi-Jacobi forms of index~$0$ on~$\Gamma_0(2)$.

Recall from Definition~\ref{def:quasijacobi2} the polar quasi-Jacobi forms of level two, as well as the properties of these forms recorded there.
\begin{lem}\label{lem:rankone2}
Let~$p\in\C[x,x^{-1}]$ with~$p(1)=p(-1)=0$ and let~$c\in\{0,\frac12\}$. Then
\[
\mathcal{J}^{(c)}_p(z;\tau)\df \sum_{\substack{m\in\Z+c\\ m\neq0}}p\big(X(m)\big)\zeta^m
\]
is a polar quasi-Jacobi form of level two, index~$0$, and of mixed weight at most~$\|p\|$. For~$c=0$ it is a form on~$\Gamma_0(2)$, and on~$\SL_2(\Z)$ if~$p\in\C[x]$.
\end{lem}
\begin{proof}
Recall the polynomials~$x^{[h]}$ ($h\geq1$) from the proof of Theorem~\ref{thm:main1}, and set $x^{[-h]}\df x^{[h]}(x^{-1})$. The polynomials~$x^{[h]}$ ($|h|\geq2$) and $x^{[1]}-x^{[-1]}=x-x^{-1}$ span the space of~$p$ under consideration. For~$c=0$, the~$x^{[h]}$ with~$h\geq2$ yield the forms of Lemma~\ref{lem:XYjacobi}, those with~$h\leq-2$ the forms just described, and
\begin{equation}\label{eq:XminusXinv}
\sum_{m\in\Z\setminus\{0\}}\big(X(m)-X(m)^{-1}\big)\zeta^m\=4\sum_{m,r\geq1}\big(\zeta^m-\zeta^{-m}\big)q^{m(2r-1)}\=-2\big(\mathcal{A}-\mathcal{A}_{\frac12}\big)(z;\tau),
\end{equation}
expanded in the chamber~$0<\Im(z)<\Im(\tau)$, is a polar quasi-Jacobi form of weight one on~$\Gamma_0(2)$. (Here the two~$\coth$ terms in~\eqref{eq:Aproduct} and~\eqref{eq:Ahalf} cancel.) For~$c=\frac12$, note that~$X(m)$ for~$m\in\Z+\frac12$ equals~$X(2m)$ evaluated at~$q^{\frac12}$, with~$2m$ odd. Hence
\[
\mathcal{J}^{(\frac12)}_p(z;\tau)\=\frac12\left(\mathcal{J}^{(0)}_p\!\left(\tfrac z2;\tfrac\tau2\right)-\mathcal{J}^{(0)}_p\!\left(\tfrac{z+1}2;\tfrac\tau2\right)\right).
\]
For $\gamma=\abcd\in\Gamma(2)$ let $\gamma'\df\left(\begin{smallmatrix}a&b/2\\2c&d\end{smallmatrix}\right)\in\Gamma_0(2)$. Then $\gamma'\frac\tau2=\frac{\gamma\tau}2$ and
\[
\frac12\Big(\frac{z}{c\tau+d}+1\Big)\=\frac{\frac{z+1}2+c\,\frac\tau2+\frac{d-1}2}{2c\,\frac\tau2+d}.
\]
Hence the transformation of $\mathcal J^{(0)}_p\big(\frac{z+1}2;\frac\tau2\big)$ under~$\gamma$ follows from the modular transformation of~$\mathcal J^{(0)}_p$ under~$\gamma'$, followed by the elliptic transformation by $c\,\frac\tau2+\frac{d-1}2$. The latter contributes the factor $\psi\big(c,\frac{d-1}2\big)=\psi\big(\frac c2,0\big)^2\,\psi\big(0,\frac{d-1}2\big)=1$, since~$c$ is even and the Fourier exponents of~$\mathcal J^{(0)}_p$ in~$\zeta$ are integral. Hence~$\mathcal J^{(\frac12)}_p$ is of level two. The weights are as in Lemma~\ref{lem:XYjacobi} and the index is~$0$, as the rescaling $(z,\tau)\mapsto(\frac z2,\frac\tau2)$ changes neither.
\end{proof}
Finally we extend Lemma~\ref{lem:regularization}. For~$c\in\{0,\frac12\}$ let
\[
H^{(c),\pm}(\zeta;q)\df \!\!\!\!\sum_{\substack{m\in\Z+c\\ m\neq0}}\!\!\!\!X(m)^{\pm1}\zeta^m, \qquad \delta^{(c)}(\zeta)\df \sum_{m\in\Z+c}\zeta^m\=\zeta^c\delta(\zeta),
\]
and let $\mathcal{A}^{(0)}_+\df \mathcal{A}$, $\mathcal{A}^{(0)}_-\df \mathcal{A}_{\frac12}$, and $\mathcal{A}^{(\frac12)}_\pm(z;\tau)\df \frac12\big(\mathcal{A}^{(0)}_\pm(\frac z2;\frac\tau2)-\mathcal{A}^{(0)}_\pm(\frac{z+1}2;\frac\tau2)\big)$. By the argument in the proof of Lemma~\ref{lem:rankone2}, all four functions $\mathcal{A}^{(c)}_\pm$ are polar quasi-Jacobi forms of level two, weight one, and index~$0$.

\begin{lem}\label{lem:regularization2}
As formal series in~$\zeta$ with coefficients in $\C[\![q^{\frac12}]\!]$, we have
\begin{equation}\label{eq:Hidentity2}
H^{(c),\pm}(\zeta;q)\=-2\,\iota_{\mathfrak c}\mathcal{A}^{(c)}_\pm(z;\tau)-\delta^{(c)}(\zeta).
\end{equation}
Consequently, if~$Z=\zeta^{c'}Z_0$ with~$c'\in\frac12\Z$ and~$Z_0$ a series whose coefficients are Laurent polynomials in~$\zeta$ vanishing at~$\zeta=1$, then $H^{(c),\pm}Z=-2\,\iota_{\mathfrak c}\mathcal{A}^{(c)}_\pm Z$.
\end{lem}
\begin{proof}
For~$c=0$ and the plus sign this is Lemma~\ref{lem:regularization}. For~$c=0$ and the minus sign we write $X(m)^{-1}=\sgn(m)-2\sgn(m)\frac{q^{|m|}}{1+q^{|m|}}$ and proceed as in 
Lemma~\ref{lem:regularization}.

For~$c=\frac12$ we apply the case~$c=0$ in the variables $(\frac z2;\frac\tau2)$ and $(\frac{z+1}{2};\frac\tau2)$ and take the difference, using $\frac12(\delta(\zeta^{\frac12})-\delta(-\zeta^{\frac12}))=\sum_{m\in\Z+\frac12}\zeta^m$. 
The last claim follows from $\delta(\zeta)\zeta^{c+c'}Z_0(\zeta)=\zeta^{c+c'}\delta(\zeta)Z_0(1)=0$.
\end{proof}

\subsection{The main theorem for level two}
Recall from the introduction the lattice sum $\Psi^{(\bm{s})}_{S,p}(\bm{\zeta};q)$ for a pair~$(S,U)$ of level two, a coset~$\bm{s}\in\frac12\Z^n$, and admissible~$p\in R_{S,U}$. By Lemma~\ref{lem:admissible} its coefficients are Laurent polynomials in~$\bm{\zeta}$ with exponents in~$\Z^n+\bm{s}$. For~$\bm{s}=\bm{0}$ and~$U=\emptyset$ this is the series~$\Psi_{S,p}$ of Theorem~\ref{thm:main1}.
Next, we prove our second main result.

\begin{proof}[Proof of Theorem~\ref{thm:main2}]
We argue by induction on the rank~$n$, as in the proof of Theorem~\ref{thm:main1}. We show that the $q$-series~$\Psi^{(\bm s)}_{S,p}$ converges on~$\mathcal C_S$ and represents there the Fourier expansion of a polar quasi-Jacobi form of level two and index~$0$, for every pair~$(S,U)$ of level two of rank~$n$, every~$\bm s\in\frac12\Z^n$ and every admissible~$p\in R_{S,U}$.

For~$n=1$, condition~\ref{def:level2:kappa}
of Definition~\ref{def:level2}, together with the spanning condition, forces
$
S=\{1\},
$
while~$U$ is either~$\emptyset$ or~$\{1\}$. Hence
$
R_{S,U}=\C[X]
\ \text{or} \ 
R_{S,U}=\C[X,X^{-1}],
$
respectively. The assertion is therefore precisely
Lemma~\ref{lem:rankone2}. This establishes the base case.

Now let~$n\geq2$ and assume the assertion for all pairs of level two
of rank~$n-1$.
Suppose a monomial of~$p$ contains two distinct top variables, say $X_{\bm{v}}^{a}X_{\bm{w}}^{b}$ with $a,b\in\Z\setminus\{0\}$. Then $\bm{u}\df \bm{w}-\bm{v}\in S^*_{n-1}\cup2U^*_{n-1}$ by Lemma~\ref{lem:parallel}~\ref{lem:parallel:diff}, and we rewrite the factor $X_{\bm{v}}^{\sgn(a)}X_{\bm{w}}^{\sgn(b)}$ by~\eqref{eq:induction1} or by Lemma~\ref{lem:laurentid}. Each of the three resulting terms has strictly smaller total degree $\sum_{\bm{t}\in T}|a_{\bm{t}}|$ in the top variables, where~$a_{\bm{t}}$ denotes the exponent of~$X_{\bm{t}}$ in the monomial, so the process terminates. The new factors lie in~$R_{S_{n-1},U_{n-1}}$: if~$a,b>0$, the new factor is~$X_{\bm{u}}$, which is~$\pm X_{\bm{u}'}$ if $\bm{u}=\pm\bm{u}'\in S^*$ and a Laurent polynomial in~$X_{\bm{u}'}$ with~$\bm{u}'\in U$ by~\eqref{eq:double} if~$\bm{u}=\pm2\bm{u}'$; if, say, $a<0$, then~$\bm{v}\in U$, since negative exponents occur only on~$U$, hence~$\bm{u}\in U^*$ by condition~\ref{def:level2:diff} of Definition~\ref{def:level2}, and the new factors~$X_{\bm{u}}^{\pm1}$ are allowed. All identities used are identities of functions on $(\Z^n+\bm{s})\setminus\bigcup_{\bm{v}\in S}H_{\bm{v}}$, since $\bm{u}\cdot\bm{m}\neq0$ there, and they preserve admissibility by the remark following Lemma~\ref{lem:laurentid}. Hence $\Psi^{(\bm{s})}_{S,p}$ is unchanged and we may assume
\[
p\big((X_{\bm{v}})_{\bm{v}\in S}\big)\=\sum_{\bm{w}\in T}p_{\bm{w}}\big(X_{\bm{w}}^{\pm1},(X_{\bm{v}})_{\bm{v}\in S_{n-1}}\big),
\]
where the monomials not containing any top variable are assigned to one fixed element of~$T$.

Next, we decompose~$p$. Recall the polynomials~$x^{[h]}$ ($h\neq0$) from the proof of Lemma~\ref{lem:rankone2}. The polynomials~$1$ and~$x^{[h]}$ ($h\neq0$) form a basis of~$\C[x,x^{-1}]$, and all~$x^{[h]}$ with~$|h|\geq2$ vanish at~$x=\pm1$. We obtain
\begin{equation}\label{eq:decomp2}
p\big((X_{\bm{v}})_{\bm{v}\in S}\big)\=\sum_{\bm{w}\in T}\Big(p_{\bm{w},0}+\sum_{h\neq0}X_{\bm{w}}^{[h]}p_{\bm{w},h}\Big), \qquad X_{\bm{w}}^{[h]}\df x^{[h]}(X_{\bm{w}}),
\end{equation}
with $p_{\bm{w},h}\in R_{S_{n-1},U_{n-1}}$, and with~$p_{\bm{w},h}=0$ for~$h<0$ unless~$\bm{w}\in U$. Under the substitution $\bm{d}=\pm \varepsilon_n$ of the admissibility condition every top variable becomes~$\pm1$ (as~$\kappa(\bm{w})=w_n=1$) and every~$X_{\bm{w}}^{[h]}$ with~$|h|\geq2$ becomes~$0$. Hence
\begin{equation}\label{eq:sumzero}
\sum_{\bm{w}\in T}p_{\bm{w},0}\=0 \qquad\text{and}\qquad \sum_{\bm{w}\in T}\big(p_{\bm{w},1}+p_{\bm{w},-1}\big)\=0.
\end{equation}
Note that only the sum $p_{\bm{w},1}+p_{\bm{w},-1}$ is controlled; this is the one place where the argument differs from the one for level one. By the first relation we may drop the terms~$p_{\bm{w},0}$ from~\eqref{eq:decomp2}. Each~$p_{\bm{w},h}$ with~$h\neq0$ is admissible for~$S_{n-1}$: as in the proof of Theorem~\ref{thm:main1}, we apply the admissibility of~$p$ in a direction $\bm{d}\in\bm{w}^\perp$, under which the other top variables become constants or stay free, and in both cases only contribute to the coefficient of~$X_{\bm{w}}^{0}$; the coefficient of~$X_{\bm{w}}^{h}$ for~$|h|\geq2$ is~$p_{\bm{w},h}$ evaluated at the signs, which must vanish, and the coefficient of~$X_{\bm{w}}^{\pm1}$ is $p_{\bm{w},\pm1}-\sum_{j\geq2}p_{\bm{w},\pm(2j-1)}$ evaluated at the signs, so that~$p_{\bm{w},\pm1}$ is admissible as well.

Now fix~$\bm{w}\in T$ and~$h\neq0$, and consider the contribution of $X_{\bm{w}}^{[h]}p_{\bm{w},h}$ to $\Psi^{(\bm{s})}_{S,p}$; for~$h=\pm1$ this is understood as a formal series in~$\zeta_n$, as in the proof of Theorem~\ref{thm:main1}, and is regularized below. We make an integral change of variables carrying~$\Z^n+\bm{s}$ onto~$\Z^n+\nu(\bm{s})$, under which $\bm{w}\cdot\bm{m}\mapsto m_n$ and $\bm{\zeta}^{\bm{m}}\mapsto\bm{\xi}^{\bm{m}_0}\zeta_n^{m_n}$, where $\bm{m}=(\bm{m}_0,m_n)$, $\bm{z}=(\bm{z}_0,z_n)$, $\bm{w}=(\bm{w}_0,1)$, $\bm{\zeta}=(\bm{\zeta}_0,\zeta_n)$, $\xi^{(\bm{u})}_j\df \zeta_j\zeta_n^{-u_j}$ for~$\bm{u}\in T$, $\bm{\xi}\df\bm{\xi}^{(\bm{w})}$, and~$\bm{s}_0$ denotes the first~$n-1$ coordinates of~$\bm{s}$. Let $c_{\bm{w}}\in\{0,\tfrac12\}$ be the representative of~$s_n-\sum_{j<n}w_js_j$ modulo~$\Z$. This produces a range of summation~$\Omega$ consisting of those $\bm{m}=(\bm{m}_0,m_n)$ with
\[
m_n\in\Z+c_{\bm{w}},\quad m_n\neq0, \qquad \bm{m}_0\in\Z^{n-1}+\bm{s}_0,\quad \bm{u}\cdot\bm{m}_0\neq0\ (\bm{u}\in S_{n-1}),
\]
together with the condition $m_n\neq(\bm{w}-\bm{v})\cdot\bm{m}_0$ for $\bm{v}\in T\setminus\{\bm{w}\}$, which is vacuous unless $\bm{v}\cdot\bm{s}\in\Z$. Let~$T_{\bm{w}}$ be the set of $\bm{v}\in T\setminus\{\bm{w}\}$ with $\bm{v}\cdot\bm{s}\in\Z$ (for~$\bm{s}=\bm{0}$ this is all of~$T\setminus\{\bm{w}\}$), and, for~$\bm{v}\in T_{\bm{w}}$, let $\Omega(\bm{v})\subseteq \Omega$ be the subset on which $m_n=(\bm{w}-\bm{v})\cdot\bm{m}_0$. For~$\bm{v}\neq\bm{t}$ the sets~$\Omega(\bm{v})$ and~$\Omega(\bm{t})$ are disjoint, since $\bm{m}\in \Omega(\bm{v})\cap \Omega(\bm{t})$ would give $(\bm{v}-\bm{t})\cdot\bm{m}_0=0$ with $\bm{v}-\bm{t}\in S^*_{n-1}\cup2U^*_{n-1}$, contradicting $\bm{u}\cdot\bm{m}_0\neq0$ for~$\bm{u}\in S_{n-1}$. Hence~\eqref{E:inex2} expresses the sum over the range of summation as the sum over~$\Omega$ minus the sums over the~$\Omega(\bm{v})$, $\bm{v}\in T_{\bm{w}}$.

Since $X_{\bm{w}}^{[h]}=x^{[h]}(X(m_n))$ depends on~$m_n$ alone, the sum over~$\Omega$ factors as
\[
\sum_{\bm{m}\in \Omega}x^{[h]}\big(X(m_n)\big)\zeta_n^{m_n}\,\bm{\xi}^{\bm{m}_0}p_{\bm{w},h}\big((X_{\bm{u}}(\bm m_0))_{\bm{u}\in S_{n-1}}\big)\=\mathcal{J}^{(c_{\bm{w}})}_{h}(z_n;\tau)\,\Psi^{(\bm{s}_0)}_{S_{n-1},p_{\bm{w},h}}(\bm{\xi}^{(\bm{w})};q).
\]
Here $\mathcal{J}^{(c)}_{h}\df \mathcal{J}^{(c)}_{x^{[h]}}$, so that~$\mathcal J^{(0)}_h$ for~$h\geq1$ is the form~$\mathcal J_h$ of~\eqref{E:ReductionFactor}. For~$|h|\geq2$ the first factor is a form of level two by Lemma~\ref{lem:rankone2}, and the second factor is a form of level two by the induction hypothesis, since $\bm{z}_0\mapsto\bm{z}_0-z_n\bm{w}_0$ is an integral-linear substitution. For~$h=\pm1$ the first factor is the formal series $H^{(c_{\bm{w}}),\pm}(\zeta_n;q)$, and the product is well-defined because the coefficients of the second factor are Laurent polynomials in~$\zeta_n$. On~$\Omega(\bm{v})$, $\bm{v}\in T_{\bm{w}}$, we have $m_n=(\bm{w}-\bm{v})\cdot\bm{m}_0$, and every $\bm{m}_0\in\Z^{n-1}+\bm{s}_0$ with $\bm{u}\cdot\bm{m}_0\neq0$ for~$\bm{u}\in S_{n-1}$ occurs exactly once (the class of $(\bm{w}-\bm{v})\cdot\bm{m}_0$ modulo~$\Z$ is~$c_{\bm{w}}$ precisely because $\bm{v}\cdot\bm{s}\in\Z$). Hence
\[
\sum_{\bm{m}\in \Omega(\bm{v})}x^{[h]}\big(X(m_n)\big)\zeta_n^{m_n}\,\bm{\xi}^{\bm{m}_0}p_{\bm{w},h}\big((X_{\bm{u}}(\bm m_0))_{\bm{u}\in S_{n-1}}\big)\=\Psi^{(\bm{s}_0)}_{S_{n-1},\,x^{[h]}(X_{\bm{w}-\bm{v}})p_{\bm{w},h}}(\bm{\xi}^{(\bm{v})};q).
\]
Here $x^{[h]}(X_{\bm{w}-\bm{v}})\in R_{S_{n-1},U_{n-1}}$: for~$h>0$ this follows from~\eqref{eq:double} as above, and for~$h<0$ we have~$\bm{w}\in U$, hence~$\bm{w}-\bm{v}\in U^*$ by condition~\ref{def:level2:diff} of Definition~\ref{def:level2}, so that $X_{\bm{w}-\bm{v}}^{-1}$ is allowed. Moreover $x^{[h]}(X_{\bm{w}-\bm{v}})p_{\bm{w},h}$ is admissible, since admissible polynomials form an ideal. By the induction hypothesis these sums are forms of level two, for every~$h\neq0$.

It remains to show that
\[
\sum_{\bm{w}\in T}\sum_{\pm}H^{(c_{\bm{w}}),\pm}(\zeta_n;q)\,Z_{\bm{w},\pm}(\bm\zeta;q), \qquad Z_{\bm{w},\pm}(\bm\zeta;q)\df \Psi^{(\bm{s}_0)}_{S_{n-1},p_{\bm{w},\pm1}}(\bm{\xi}^{(\bm{w})};q),
\]
is a form of level two. By~\eqref{eq:Hidentity2} each summand equals
\[
-2\,\iota_{\mathfrak c}\mathcal{A}^{(c_{\bm{w}})}_\pm(z_n;\tau)\,Z_{\bm{w},\pm}(\bm\zeta;q)-\delta^{(c_{\bm{w}})}(\zeta_n)\,Z_{\bm{w},\pm}(\bm\zeta;q),
\]
and we claim that the terms involving~$\delta$ add up to zero. The coefficients of~$Z_{\bm{w},\pm}$ are Laurent polynomials in~$\zeta_n$ whose exponents~$-\sum_{j<n}w_jm_j$ are congruent modulo~$\Z$ to~$-\sum_{j<n}w_js_j$, whose representative in~$\{0,\frac12\}$ we denote by~$c'_{\bm{w}}$. Hence
\[
\delta^{(c_{\bm{w}})}(\zeta_n)\,Z_{\bm{w},\pm}(\bm\zeta;q)\=\zeta_n^{c_{\bm{w}}+c'_{\bm{w}}}\delta(\zeta_n)\,Z_{\bm{w},\pm}(\bm\zeta;q)\big|_{\zeta_n=1},
\]
where the evaluation refers to the integral part $\zeta_n^{-c'_{\bm{w}}}Z_{\bm{w},\pm}$. Now $c_{\bm{w}}+c'_{\bm{w}}\equiv s_n-2\sum_{j<n}w_js_j\equiv s_n\pmod{\Z}$, so that $\zeta_n^{c_{\bm{w}}+c'_{\bm{w}}}\delta(\zeta_n)=\zeta_n^{s_n}\delta(\zeta_n)$ does not depend on~$\bm{w}$, and the evaluation at~$\zeta_n=1$ equals $\Psi^{(\bm{s}_0)}_{S_{n-1},p_{\bm{w},\pm1}}(\bm{\zeta}_0;q)$. Therefore the terms involving~$\delta$ combine to
\[
\zeta_n^{s_n}\delta(\zeta_n)\sum_{\bm{w}\in T}\sum_\pm \Psi^{(\bm{s}_0)}_{S_{n-1},p_{\bm{w},\pm1}}(\bm{\zeta}_0;q)\=\zeta_n^{s_n}\delta(\zeta_n)\,\Psi^{(\bm{s}_0)}_{S_{n-1},\sum_{\bm{w}}(p_{\bm{w},1}+p_{\bm{w},-1})}(\bm{\zeta}_0;q)\=0
\]
by~\eqref{eq:sumzero}. The remaining expression $-2\iota_{\mathfrak c}\sum_{\bm{w},\pm}\mathcal{A}^{(c_{\bm{w}})}_\pm(z_n;\tau)\,Z_{\bm{w},\pm}(\bm\zeta;q)$ is a form of level two by the induction hypothesis, the operator~$\iota_{\mathfrak c}$ only recording the chamber in which the product is expanded.

Combining the above, we obtain the reduction formula
\begin{align*}
\Psi^{(\bm{s})}_{S,p}(\bm{\zeta};q)\=&\sum_{\bm{w}\in T}\sum_{|h|\geq2}\mathcal{J}^{(c_{\bm{w}})}_{h}(z_n;\tau)\,\Psi^{(\bm{s}_0)}_{S_{n-1},p_{\bm{w},h}}(\bm{\xi}^{(\bm{w})};q)\\
&-2\,\iota_{\mathfrak c}\sum_{\bm{w}\in T}\sum_{\pm}\mathcal{A}^{(c_{\bm{w}})}_{\pm}(z_n;\tau)\,\Psi^{(\bm{s}_0)}_{S_{n-1},p_{\bm{w},\pm1}}(\bm{\xi}^{(\bm{w})};q)\\
&-\sum_{\bm{w}\in T}\sum_{h\neq0}\sum_{\bm{v}\in T_{\bm{w}}}\Psi^{(\bm{s}_0)}_{S_{n-1},\,x^{[h]}(X_{\bm{w}-\bm{v}})p_{\bm{w},h}}(\bm{\xi}^{(\bm{v})};q),
\end{align*}
which for~$U=\emptyset$, $\bm{s}=\bm{0}$, and~$p$ a polynomial is~\eqref{total}, summed over~$\bm{w}\in T$ and~$h\geq1$, with the terms containing~$\mathcal J_1=H$ rewritten by Lemma~\ref{lem:regularization}. Every term on the right-hand side is a polar quasi-Jacobi form of level two and index~$0$. This completes the induction. Convergence on~$\mathcal C_S$ (the recursion~\eqref{eq:chamber} is not affected by the coset), the weight bound and the passage from~$\mathcal C_S$ to~$\mathcal T_S$ follow as in the proof of Theorem~\ref{thm:main1}. Here $\mathcal J^{(c)}_h$ has mixed weight at most~$|h|$ by Lemma~\ref{lem:rankone2}, $\mathcal A^{(c)}_\pm$ has weight one, the rewriting by~\eqref{eq:induction1}, Lemma~\ref{lem:laurentid} and~\eqref{eq:double} does not increase~$\|p\|$, and Lemma~\ref{lem:chamber} holds for pairs with the same proof. If~$\bm{s}=\bm{0}$, all Fourier exponents are integral, and the form is on~$\Gamma_0(2)$.
\end{proof}

\begin{cor}\label{cor:polyinsertion2}
Let~$(S,U)$ be of level two, $\bm{s}\in\frac12\Z^n$, $\bm{\ell}\in\N_0^n$, and let~$p\in R_{S,U}$ be admissible. Then
\[
\!\!\!\!\sum_{\substack{\bm{m}\in\Z^n+\bm{s}\\ \bm{v}\cdot\bm{m}\neq0\ (\bm{v}\in S)}}\!\!\!\!\bm{m}^{\bm{\ell}}\,p\big((X_{\bm{v}}(\bm m))_{\bm{v}\in S}\big)
\]
is a quasimodular form on~$\Gamma(2)$, and on~$\Gamma_0(2)$ if its~$q$-expansion is integral. In particular this applies to $p=\prod_{\bm{v}\in S}Y_{\bm{v}}^{k_{\bm{v}}}$ with~$k_{\bm{v}}\in\N$, where $Y_{\bm{v}}\df \frac14(X_{\bm{v}}^2-1)=Y(\bm{v}\cdot\bm{m})$.
\end{cor}
\begin{proof}
As in the proof of Corollary~\ref{cor:polyinsertion}, now using Theorem~\ref{thm:main2}, applying the operator~$D^{\bm\ell}$ of that proof and setting~$\bm z=\bm0$ gives a quasimodular form on~$\Gamma(2)$. If the~$q$-expansion is integral, it is invariant under~$\tau\mapsto\tau+1$, and~$\Gamma(2)$ together with this translation generates~$\Gamma_0(2)$. Finally, $p=\prod_{\bm{v}\in S}Y_{\bm{v}}^{k_{\bm{v}}}$ is admissible by Lemma~\ref{lem:admissible}~\ref{lem:admissible:prod}.
\end{proof}

For root systems of type~$A$ the level two results also give the case of odd exponents.
\begin{cor}\label{cor:Aodd}
Let $\mathfrak g=\mathfrak{sl}_n(\C)$, let~$\ell$ be odd and let~$k_\alpha$ be odd for all~$\alpha\in\Delta_+$. Then $F^W_{\mathfrak g,\ell,\mathbf k}$ is a quasimodular form on~$\Gamma(2)$, and on~$\Gamma_0(2)$ if its~$q$-expansion is integral. In particular, this holds for $F_{\mathfrak g,\ell,k}$ with~$k$ and~$\ell$ odd.
\end{cor}
\begin{proof}
For~$a\in\Z\setminus\{0\}$ we have
\[
\frac{q^{\frac a2}}{1-q^a}\=\frac14\left(X\big(\tfrac a2\big)-X\big(\tfrac a2\big)^{-1}\right),\qquad X\big(\tfrac a2\big)\=\frac{1+q^{\frac a2}}{1-q^{\frac a2}}.
\]
As in the proof of Corollary~\ref{k.var}, by Theorem~\ref{thm.var} and the Weyl dimension formula~\eqref{weyl-dim}, $F^W_{\mathfrak g,\ell,\mathbf k}(q)$ is a~$\C$-linear combination of the sums
\[
\!\!\!\!\sum_{\substack{\bm m\in\Z^{n-1}\\ \vec v\cdot\bm m\neq0\ (\vec v\in S)}}\!\!\!\!\bm m^{\bm\ell}\prod_{\vec v\in S}\left(\frac{q^{\frac12\vec v\cdot\bm m}}{1-q^{\vec v\cdot\bm m}}\right)^{\!k_{\vec v}},\qquad S\df \Delta_+(A_{n-1}),
\]
with~$k_{\vec v}=k_\alpha$ for the root~$\alpha$ represented by~$\vec v$. Substituting~$\bm m=2\bm m'$ and splitting $\bm m'\in\frac12\Z^{n-1}$ into the cosets~$\Z^{n-1}+\bm s$, $\bm s\in\{0,\frac12\}^{n-1}$, each such sum becomes $2^{|\bm\ell|}4^{-\sum_{\vec v}k_{\vec v}}$ times the sum over~$\bm s$ of
\[
\!\!\!\!\sum_{\substack{\bm m'\in\Z^{n-1}+\bm s\\ \vec v\cdot\bm m'\neq0\ (\vec v\in S)}}\!\!\!\!\bm m'^{\bm\ell}\,p\big((X_{\vec v}(\bm m'))_{\vec v\in S}\big),\qquad p\df \prod_{\vec v\in S}\left(X_{\vec v}-X_{\vec v}^{-1}\right)^{\!k_{\vec v}}.
\]
Since~$S$ is of level one, the pair~$(S,S)$ is of level two, and~$p\in R_{S,S}$ is admissible, as~$x-x^{-1}$ vanishes at~$x=\pm1$ and~$S$ spans~$\R^{n-1}$. Hence each of these sums is quasimodular on~$\Gamma(2)$ by Corollary~\ref{cor:polyinsertion2}, and so is $F^W_{\mathfrak g,\ell,\mathbf k}(q)$. If its~$q$-expansion is integral, it is invariant under~$\tau\mapsto\tau+1$, hence quasimodular on~$\Gamma_0(2)$.
\end{proof}

\subsection{Quasimodularity for types \texorpdfstring{$B$}{B}, \texorpdfstring{$C$}{C}, and \texorpdfstring{$D$}{D}}
We use the coordinates $\bm m=(m_1,\ldots,m_n)$ of~$\lambda$ from~\eqref{eq:epscoord} and the variables $X_{\bm v}=X(\bm v\cdot\bm m)$, recalling that $Y(\bm v\cdot\bm m)=\frac14(X_{\bm v}^2-1)$. The following two examples explain how the series for~$C_n$ and~$D_n$ are described by the pair~$(S_B,U_B)$ of Example~\ref{ex:Bn}.

\begin{example}[Root system~$C_n$]\label{ex:Cn}
One has $\Delta_+(C_n)=\{\varepsilon_j\pm\varepsilon_i:1\leq i<j\leq n\}\cup\{2\varepsilon_j:1\leq j\leq n\}$. The short roots $\varepsilon_j\pm\varepsilon_i$ are elements of~$S_B$, and each long root~$2\varepsilon_j$ is twice the element~$\varepsilon_j$ of~$U_B$. Accordingly:
\begin{itemize}[leftmargin=*]
\item the root hyperplanes are the hyperplanes~$H_{\bm v}$, $\bm v\in S_B$, since $(\lambda,2\varepsilon_j)=2m_j$;
\item the factor of the long root~$2\varepsilon_j$ in the summand of $F_{\mathfrak g,\ell,k}$ is, by~\eqref{eq:double},
\[
Y(2m_j)^{\frac k2}\=Y(m_j)^{k}X(m_j)^{-k},
\]
so that the summand is a polynomial in~$\bm m$ times
$$\prod_{i<j}\bigl(Y(m_j+m_i)Y(m_j-m_i)\bigr)^{\!\frac k2}\prod_{j}Y(m_j)^{k}X(m_j)^{-k}\in R_{S_B,U_B}, $$ the negative powers occurring exactly for the vectors in~$U_B$.
\end{itemize}
\end{example}

\begin{example}[Root system~$D_n$]\label{ex:Dn}
One has  $\Delta_+(D_n)=\{\varepsilon_j\pm\varepsilon_i:1\leq i<j\leq n\}$, the set~$S_B$ without the vectors of~$U_B$. Accordingly:
\begin{itemize}[leftmargin=*]
\item the summand of $F_{\mathfrak g,\ell,k}$ is a polynomial in~$\bm m$ times $\prod_{i<j}\bigl(Y(m_j+m_i)Y(m_j-m_i)\bigr)^{\!\frac k2}$, an element of~$R_{S_B,U_B}$ in which the variables~$X_{\varepsilon_j}$ do not occur;
\item the root hyperplanes are only the~$H_{\bm v}$ with $\bm v\in S_B\setminus U_B$, so that the sum over~$P\setminus\mathcal H$ contains, besides the points of the series for~$(S_B,U_B)$, the points with~$m_j=0$ for some~$j$. On the coset~$\Z^n+\frac12\bm1$ there are no such points, and on~$\Z^n$ at most one~$m_j$ vanishes, since~$m_i=m_j=0$ would give~$m_j-m_i=0$; these points are added separately in the proof of Corollary~\ref{cor:BCD}.
\end{itemize}
\end{example}

Recall from~\eqref{eq:FWfull} that for $k_\alpha\equiv\ell\pmod2$ the symmetrized series $F^W_{\mathfrak{g},\ell,\mathbf{k}}(q)$ is a sum over $\lambda\in P\setminus \mathcal H$, which for~$k_\alpha=k$ equals $|W|F_{\mathfrak{g},\ell,k}(q)$ by Corollary~\ref{cor:symm}. In the coordinates~\eqref{eq:epscoord} the weight lattice is~$\Z^n$ for~$B_n$ and $\Z^n\cup(\Z^n+\frac12\bm{1})$ for~$C_n$ and~$D_n$ by~\eqref{eq:weightlattice}, and by the Weyl dimension formula~\eqref{weyl-dim} the factor $\dim(V_{\lambda-\rho})^\ell$ is a polynomial in~$\bm{m}$, exactly as in the proof of Corollary~\ref{k.unif}. Hence it suffices to consider monomial insertions~$\bm{m}^{\bm{\ell}}$.
\begin{cor}\label{cor:BCD}
Let~$\mathfrak{g}$ be of type~$B_n$, $C_n$, or~$D_n$, and let $k_\alpha\equiv\ell\equiv0\pmod2$. Then $F_{\mathfrak{g},\ell,k}$ and $F^W_{\mathfrak{g},\ell,\mathbf{k}}$ are quasimodular forms on~$\Gamma_0(2)$.
\end{cor}
\begin{proof}
We write out the case~$k_\alpha=k$; the case of varying exponents is identical.

\emph{Type~$B_n$.} Here~$\Delta_+=S_B$, the root hyperplanes are~$H_{\bm{v}}$ for~$\bm{v}\in S_B$, the lattice is~$\Z^n$, and the summand is $\bm{m}^{\bm{\ell}}\prod_{\bm{v}\in S_B}Y_{\bm{v}}^{k/2}$ (for the roots $\varepsilon_i-\varepsilon_j$ with~$i<j$ we use~$Y(-a)=Y(a)$). This is Corollary~\ref{cor:polyinsertion2} with~$\bm{s}=\bm{0}$.

\emph{Type~$C_n$.} The root hyperplanes of~$C_n$ are again the~$H_{\bm{v}}$, $\bm{v}\in S_B$, and the summand is $\bm{m}^{\bm{\ell}}\prod_{i<j}Y(m_j\pm m_i)^{k/2}\prod_iY(2m_i)^{k/2}$, with $Y(2m_i)=Y(m_i)^2X(m_i)^{-2}$ by~\eqref{eq:double}. This element of~$R_{S_B,U_B}$ is admissible by Lemma~\ref{lem:admissible}~\ref{lem:admissible:prod}, as it is divisible by $\prod_{\bm{v}\in S_B}(X_{\bm{v}}^2-1)$ (recall that~$k\geq2$ is even). The sum over the weight lattice is the sum of the two coset sums with~$\bm{s}=\bm{0}$ and~$\bm{s}=\frac12\bm{1}$, and each of them has an integral~$q$-expansion, since $(\lambda,\alpha)\in\Z$ for all roots. By Corollary~\ref{cor:polyinsertion2} both are quasimodular on~$\Gamma_0(2)$.

\emph{Type~$D_n$.} The summand $\bm{m}^{\bm{\ell}}\prod_{i<j}Y(m_j\pm m_i)^{k/2}$ is a polynomial in the~$X_{\bm{v}}$, $\bm{v}\in S_B$, not involving the~$X_{\varepsilon_i}$, and it is admissible for~$S_B$ since~$\Delta_+(D_n)$ spans~$\R^n$. On the coset~$\Z^n+\frac12\bm{1}$ no coordinate~$m_i$ vanishes, so the sum over this coset minus the root hyperplanes of~$D_n$ is the case~$\bm{s}=\frac12\bm{1}$ of Corollary~\ref{cor:polyinsertion2}. On the coset~$\Z^n$ the hyperplanes~$H_{\varepsilon_i}$ are not excluded, and we decompose
\[
\!\!\!\!\sum_{\substack{\bm{m}\in\Z^n\\ m_j\pm m_i\neq0\ (i<j)}}\=\!\!\!\!\sum_{\substack{\bm{m}\in\Z^n\\ \bm{v}\cdot\bm{m}\neq0\ (\bm{v}\in S_B)}}+\sum_{i=1}^n\ \!\!\!\!\sum_{\substack{\bm{m}\in\Z^n,\ m_i=0\\ m_j\pm m_l\neq0\ (j<l)}},
\]
where no terms with two vanishing coordinates occur because~$m_i=m_l=0$ would give~$m_l-m_i=0$. The first sum is the case~$\bm{s}=\bm{0}$ of Corollary~\ref{cor:polyinsertion2}. In the~$i$-th sum, both factors~$Y(m_j\pm m_i)$ become~$Y(m_j)$, the remaining conditions~$m_j\neq0$ and~$m_j\pm m_l\neq0$ for~$j,l\neq i$ are exactly the conditions $\bm{v}\cdot\bm{m}\neq0$ for the set~$S_B$ of the pair~$(S_B,U_B)$ of rank~$n-1$ in the coordinates~$(m_j)_{j\neq i}$, and the summand is $\prod_{j\neq i}Y(m_j)^{k}\prod_{j<l;\,j,l\neq i}Y(m_j\pm m_l)^{k/2}$ times the restriction of~$\bm{m}^{\bm{\ell}}$, which is admissible. So each of these sums is again covered by Corollary~\ref{cor:polyinsertion2}. All~$q$-expansions involved are integral, so the total is quasimodular on~$\Gamma_0(2)$.
\end{proof}

\begin{rem}
One should be able to extend Corollary~\ref{cor:BCD} to $k_\alpha\equiv\ell\equiv1\pmod2$ as well, possibly for a smaller group than~$\Gamma_0(2)$. Numerically, $F_{B_2,\ell,k}(q)\in\widetilde M(\Gamma(2))$ for $(k,\ell)\in\{(1,1),(3,1),(1,3),(3,3)\}$, and $F_{C_3,1,1}(q),F_{D_4,1,1}(q)\in\widetilde M(\Gamma_0(2))$. However, $F_{B_3,1,1}(q)$ does not lie in~$\widetilde M(\Gamma(2))$, but it lies in~$\widetilde M(\Gamma(4))$. We do not pursue the general case here, as it requires a slightly different setup.
\end{rem}

\subsection{Higher level sets}
\begin{rem}
The same method should apply to sets in which differences of the form~$c\bm{u}$ with~$c>2$ occur, at the cost of replacing~$\Gamma(2)$ by~$\Gamma(N)$. The natural coefficient ring is generated by the functions
\[
X_\eta(a)\df \frac{1+\eta q^a}{1-\eta q^a}\qquad(\eta^N\=1),
\]
i.e., by the rational functions of~$q^a$ with poles at~$N$-th roots of unity; note that~$X_{-1}(a)=X(a)^{-1}$, so that for~$N=2$ this is the ring~$\C[X,X^{-1}]$ used above. 

The exceptional root system~$G_2$, whose long roots have coefficient~$3$ in the simple-root basis, is expected to require~$N=6$. For~$F_4$, $E_6$, $E_7$ and~$E_8$, however, new features of the root systems arise, and the notion of higher level as we have defined it does not seem to admit a description in terms of higher level sets; a different approach is needed. We hope to return to these exceptional cases in future work.

\end{rem}

\appendix
\section{Level one sets of low rank}\label{sec:appendix}

In this appendix we list all sets of level one and rank~$\leq 4$ that are subsets of~$\Delta_+(A_n)$. We write~${\sim}(k)$ after a set if it is equivalent to the set~$(k)$, where~$(k)$ is the first set in the list with this property.

\subsection{Rank one}
For completeness: $\Delta_+(A_1)=\{(1)\}$ is the only set of rank one.

\subsection{Rank two} 
There are 3 level-one sets in~$\mathbb{R}^2$ that are subsets of~$\Delta_+(A_2)$:

{\footnotesize
\begin{multicols}{2}
\begin{enumerate}[leftmargin=18pt]
\item $\{(1,0),(0,1)\}$
\item $\{(1,0),(1,1) \}$\nobreak\,${\sim}(1)$
\item $\Delta_+(A_2)$
\end{enumerate}
\end{multicols}
}

\subsection{Rank three} 

There are 17 level-one sets in~$\mathbb{R}^3$ that are subsets of~$\Delta_+(A_3)$:

{\footnotesize
\begin{multicols}{2}
\begin{enumerate}[leftmargin=18pt]
\item $\{(1,0,0), (1,1,0), (1,1,1)\}$
\item $\{(1,0,0), (1,1,0), (0,1,1)\}$\nobreak\,${\sim}(1)$
\item $\{(1,0,0), (1,1,0), (0,0,1)\}$\nobreak\,${\sim}(1)$
\item $\{(1,0,0), (1,1,1), (0,1,0)\}$\nobreak\,${\sim}(1)$
\item $\{(1,0,0), (0,1,0), (0,1,1)\}$\nobreak\,${\sim}(1)$
\item $\{(1,0,0), (0,1,0), (0,0,1)\}$\nobreak\,${\sim}(1)$
\item $\{(1,0,0), (1,1,0), (1,1,1), (0,1,0)\}$
\item $\{(1,0,0), (1,1,0), (1,1,1), (0,1,1)\}$
\item $\{(1,0,0), (1,1,0), (1,1,1), (0,0,1)\}$
\item $\{(1,0,0), (1,1,0), (0,1,0), (0,1,1)\}$\nobreak\,${\sim}(7)$
\item $\{(1,0,0), (1,1,0), (0,1,0), (0,0,1)\}$\nobreak\,${\sim}(7)$
\item $\{(1,0,0), (1,1,1), (0,1,0), (0,1,1)\}$\nobreak\,${\sim}(8)$
\item $\{(1,0,0), (0,1,0), (0,1,1), (0,0,1)\}$\nobreak\,${\sim}(9)$
\item $\{(1,0,0), (1,1,0), (1,1,1), (0,1,0), (0,1,1)\}$
\item $\{(1,0,0), (1,1,0), (1,1,1), (0,1,0), (0,0,1)\}$
\item $\{(1,0,0), (1,1,0), (0,1,0), (0,1,1), (0,0,1)\}$\nobreak\,${\sim}(15)$
\item $\Delta_+(A_3)$
\end{enumerate}
\end{multicols}
}

\subsection{Rank four}

There are 153 level-one sets in~$\mathbb{R}^4$ that are subsets of~$\Delta_+(A_4)$:

{\footnotesize
\begin{multicols}{2}
\begin{enumerate}[leftmargin=18pt]
\item $\{(1,0,0,0), (1,1,0,0), (1,1,1,0), (1,1,1,1)\}$
\item $\{(1,0,0,0), (1,1,0,0), (1,1,1,0), (0,1,1,1)\}$\nobreak\,${\sim}(1)$
\item $\{(1,0,0,0), (1,1,0,0), (1,1,1,0), (0,0,1,1)\}$\nobreak\,${\sim}(1)$
\item $\{(1,0,0,0), (1,1,0,0), (1,1,1,0), (0,0,0,1)\}$\nobreak\,${\sim}(1)$
\item $\{(1,0,0,0), (1,1,0,0), (1,1,1,1), (0,1,1,0)\}$\nobreak\,${\sim}(1)$
\item $\{(1,0,0,0), (1,1,0,0), (1,1,1,1), (0,0,1,0)\}$\nobreak\,${\sim}(1)$
\item $\{(1,0,0,0), (1,1,0,0), (0,1,1,0), (0,1,1,1)\}$\nobreak\,${\sim}(1)$
\item $\{(1,0,0,0), (1,1,0,0), (0,1,1,0), (0,0,1,1)\}$\nobreak\,${\sim}(1)$
\item $\{(1,0,0,0), (1,1,0,0), (0,1,1,0), (0,0,0,1)\}$\nobreak\,${\sim}(1)$
\item $\{(1,0,0,0), (1,1,0,0), (0,1,1,1), (0,0,1,0)\}$\nobreak\,${\sim}(1)$
\item $\{(1,0,0,0), (1,1,0,0), (0,0,1,0), (0,0,1,1)\}$\nobreak\,${\sim}(1)$
\item $\{(1,0,0,0), (1,1,0,0), (0,0,1,0), (0,0,0,1)\}$\nobreak\,${\sim}(1)$
\item $\{(1,0,0,0), (1,1,1,0), (1,1,1,1), (0,1,0,0)\}$\nobreak\,${\sim}(1)$
\item $\{(1,0,0,0), (1,1,1,0), (0,1,0,0), (0,1,1,1)\}$\nobreak\,${\sim}(1)$
\item $\{(1,0,0,0), (1,1,1,0), (0,1,0,0), (0,0,1,1)\}$\nobreak\,${\sim}(1)$
\item $\{(1,0,0,0), (1,1,1,0), (0,1,0,0), (0,0,0,1)\}$\nobreak\,${\sim}(1)$
\item $\{(1,0,0,0), (1,1,1,1), (0,1,0,0), (0,1,1,0)\}$\nobreak\,${\sim}(1)$
\item $\{(1,0,0,0), (1,1,1,1), (0,1,0,0), (0,0,1,0)\}$\nobreak\,${\sim}(1)$
\item $\{(1,0,0,0), (0,1,0,0), (0,1,1,0), (0,1,1,1)\}$\nobreak\,${\sim}(1)$
\item $\{(1,0,0,0), (0,1,0,0), (0,1,1,0), (0,0,1,1)\}$\nobreak\,${\sim}(1)$
\item $\{(1,0,0,0), (0,1,0,0), (0,1,1,0), (0,0,0,1)\}$\nobreak\,${\sim}(1)$
\item $\{(1,0,0,0), (0,1,0,0), (0,1,1,1), (0,0,1,0)\}$\nobreak\,${\sim}(1)$
\item $\{(1,0,0,0), (0,1,0,0), (0,0,1,0), (0,0,1,1)\}$\nobreak\,${\sim}(1)$
\item $\{(1,0,0,0), (0,1,0,0), (0,0,1,0), (0,0,0,1)\}$\nobreak\,${\sim}(1)$
\end{enumerate}
\end{multicols}
\vspace{-\multicolsep}
\begin{enumerate}[leftmargin=18pt,start=25,topsep=3pt]
\item $\{(1,0,0,0), (1,1,0,0), (1,1,1,0), (1,1,1,1), (0,1,0,0)\}$
\item $\{(1,0,0,0), (1,1,0,0), (1,1,1,0), (1,1,1,1), (0,1,1,0)\}$
\item $\{(1,0,0,0), (1,1,0,0), (1,1,1,0), (1,1,1,1), (0,1,1,1)\}$
\item $\{(1,0,0,0), (1,1,0,0), (1,1,1,0), (1,1,1,1), (0,0,1,0)\}$
\item $\{(1,0,0,0), (1,1,0,0), (1,1,1,0), (1,1,1,1), (0,0,1,1)\}$
\item $\{(1,0,0,0), (1,1,0,0), (1,1,1,0), (1,1,1,1), (0,0,0,1)\}$
\item $\{(1,0,0,0), (1,1,0,0), (1,1,1,0), (0,1,0,0), (0,1,1,1)\}$\nobreak\,${\sim}(25)$
\item $\{(1,0,0,0), (1,1,0,0), (1,1,1,0), (0,1,0,0), (0,0,1,1)\}$\nobreak\,${\sim}(25)$
\item $\{(1,0,0,0), (1,1,0,0), (1,1,1,0), (0,1,0,0), (0,0,0,1)\}$\nobreak\,${\sim}(25)$
\item $\{(1,0,0,0), (1,1,0,0), (1,1,1,0), (0,1,1,0), (0,1,1,1)\}$\nobreak\,${\sim}(26)$
\item $\{(1,0,0,0), (1,1,0,0), (1,1,1,0), (0,1,1,0), (0,0,1,1)\}$\nobreak\,${\sim}(26)$
\item $\{(1,0,0,0), (1,1,0,0), (1,1,1,0), (0,1,1,0), (0,0,0,1)\}$\nobreak\,${\sim}(26)$
\item $\{(1,0,0,0), (1,1,0,0), (1,1,1,0), (0,1,1,1), (0,0,1,0)\}$\nobreak\,${\sim}(28)$
\item $\{(1,0,0,0), (1,1,0,0), (1,1,1,0), (0,0,1,0), (0,0,1,1)\}$\nobreak\,${\sim}(28)$
\item $\{(1,0,0,0), (1,1,0,0), (1,1,1,0), (0,0,1,0), (0,0,0,1)\}$\nobreak\,${\sim}(28)$
\item $\{(1,0,0,0), (1,1,0,0), (1,1,1,1), (0,1,0,0), (0,1,1,0)\}$\nobreak\,${\sim}(25)$
\item $\{(1,0,0,0), (1,1,0,0), (1,1,1,1), (0,1,0,0), (0,0,1,0)\}$\nobreak\,${\sim}(25)$
\item $\{(1,0,0,0), (1,1,0,0), (1,1,1,1), (0,1,1,0), (0,1,1,1)\}$\nobreak\,${\sim}(27)$
\item $\{(1,0,0,0), (1,1,0,0), (1,1,1,1), (0,1,1,0), (0,0,1,1)\}$\nobreak\,${\sim}(29)$
\item $\{(1,0,0,0), (1,1,0,0), (1,1,1,1), (0,1,1,1), (0,0,1,0)\}$\nobreak\,${\sim}(27)$
\item $\{(1,0,0,0), (1,1,0,0), (1,1,1,1), (0,0,1,0), (0,0,1,1)\}$\nobreak\,${\sim}(29)$
\item $\{(1,0,0,0), (1,1,0,0), (0,1,0,0), (0,1,1,0), (0,1,1,1)\}$\nobreak\,${\sim}(25)$
\item $\{(1,0,0,0), (1,1,0,0), (0,1,0,0), (0,1,1,0), (0,0,1,1)\}$\nobreak\,${\sim}(25)$
\item $\{(1,0,0,0), (1,1,0,0), (0,1,0,0), (0,1,1,0), (0,0,0,1)\}$\nobreak\,${\sim}(25)$
\item $\{(1,0,0,0), (1,1,0,0), (0,1,0,0), (0,1,1,1), (0,0,1,0)\}$\nobreak\,${\sim}(25)$
\item $\{(1,0,0,0), (1,1,0,0), (0,1,0,0), (0,0,1,0), (0,0,1,1)\}$\nobreak\,${\sim}(25)$
\item $\{(1,0,0,0), (1,1,0,0), (0,1,0,0), (0,0,1,0), (0,0,0,1)\}$\nobreak\,${\sim}(25)$
\item $\{(1,0,0,0), (1,1,0,0), (0,1,1,0), (0,1,1,1), (0,0,0,1)\}$\nobreak\,${\sim}(30)$
\item $\{(1,0,0,0), (1,1,0,0), (0,0,1,0), (0,0,1,1), (0,0,0,1)\}$\nobreak\,${\sim}(30)$
\item $\{(1,0,0,0), (1,1,1,0), (1,1,1,1), (0,1,0,0), (0,1,1,0)\}$\nobreak\,${\sim}(26)$
\item $\{(1,0,0,0), (1,1,1,0), (1,1,1,1), (0,1,0,0), (0,1,1,1)\}$\nobreak\,${\sim}(27)$
\item $\{(1,0,0,0), (1,1,1,0), (1,1,1,1), (0,1,0,0), (0,0,0,1)\}$\nobreak\,${\sim}(30)$
\item $\{(1,0,0,0), (1,1,1,0), (0,1,0,0), (0,1,1,0), (0,1,1,1)\}$\nobreak\,${\sim}(26)$
\item $\{(1,0,0,0), (1,1,1,0), (0,1,0,0), (0,1,1,0), (0,0,1,1)\}$\nobreak\,${\sim}(26)$
\item $\{(1,0,0,0), (1,1,1,0), (0,1,0,0), (0,1,1,0), (0,0,0,1)\}$\nobreak\,${\sim}(26)$
\item $\{(1,0,0,0), (1,1,1,0), (0,1,0,0), (0,1,1,1), (0,0,1,1)\}$\nobreak\,${\sim}(29)$
\item $\{(1,0,0,0), (1,1,1,1), (0,1,0,0), (0,1,1,0), (0,1,1,1)\}$\nobreak\,${\sim}(27)$
\item $\{(1,0,0,0), (1,1,1,1), (0,1,0,0), (0,1,1,0), (0,0,1,0)\}$\nobreak\,${\sim}(28)$
\item $\{(1,0,0,0), (1,1,1,1), (0,1,0,0), (0,1,1,1), (0,0,1,0)\}$\nobreak\,${\sim}(27)$
\item $\{(1,0,0,0), (0,1,0,0), (0,1,1,0), (0,1,1,1), (0,0,1,0)\}$\nobreak\,${\sim}(28)$
\item $\{(1,0,0,0), (0,1,0,0), (0,1,1,0), (0,1,1,1), (0,0,1,1)\}$\nobreak\,${\sim}(29)$
\item $\{(1,0,0,0), (0,1,0,0), (0,1,1,0), (0,1,1,1), (0,0,0,1)\}$\nobreak\,${\sim}(30)$
\item $\{(1,0,0,0), (0,1,0,0), (0,1,1,0), (0,0,1,0), (0,0,1,1)\}$\nobreak\,${\sim}(28)$
\item $\{(1,0,0,0), (0,1,0,0), (0,1,1,0), (0,0,1,0), (0,0,0,1)\}$\nobreak\,${\sim}(28)$
\item $\{(1,0,0,0), (0,1,0,0), (0,1,1,1), (0,0,1,0), (0,0,1,1)\}$\nobreak\,${\sim}(29)$
\item $\{(1,0,0,0), (0,1,0,0), (0,0,1,0), (0,0,1,1), (0,0,0,1)\}$\nobreak\,${\sim}(30)$
\item $\{(1,0,0,0), (1,1,0,0), (1,1,1,0), (1,1,1,1), (0,1,0,0), (0,1,1,0)\}$
\item $\{(1,0,0,0), (1,1,0,0), (1,1,1,0), (1,1,1,1), (0,1,0,0), (0,1,1,1)\}$
\item $\{(1,0,0,0), (1,1,0,0), (1,1,1,0), (1,1,1,1), (0,1,0,0), (0,0,1,0)\}$
\item $\{(1,0,0,0), (1,1,0,0), (1,1,1,0), (1,1,1,1), (0,1,0,0), (0,0,1,1)\}$
\item $\{(1,0,0,0), (1,1,0,0), (1,1,1,0), (1,1,1,1), (0,1,0,0), (0,0,0,1)\}$
\item $\{(1,0,0,0), (1,1,0,0), (1,1,1,0), (1,1,1,1), (0,1,1,0), (0,1,1,1)\}$
\item $\{(1,0,0,0), (1,1,0,0), (1,1,1,0), (1,1,1,1), (0,1,1,0), (0,0,1,1)\}$
\item $\{(1,0,0,0), (1,1,0,0), (1,1,1,0), (1,1,1,1), (0,1,1,0), (0,0,0,1)\}$
\item $\{(1,0,0,0), (1,1,0,0), (1,1,1,0), (1,1,1,1), (0,1,1,1), (0,0,1,0)\}$
\item $\{(1,0,0,0), (1,1,0,0), (1,1,1,0), (1,1,1,1), (0,0,1,0), (0,0,1,1)\}$
\item $\{(1,0,0,0), (1,1,0,0), (1,1,1,0), (1,1,1,1), (0,0,1,0), (0,0,0,1)\}$
\item $\{(1,0,0,0), (1,1,0,0), (1,1,1,0), (0,1,0,0), (0,1,1,0), (0,1,1,1)\}$\nobreak\,${\sim}(71)$
\item $\{(1,0,0,0), (1,1,0,0), (1,1,1,0), (0,1,0,0), (0,1,1,0), (0,0,1,1)\}$\nobreak\,${\sim}(71)$
\item $\{(1,0,0,0), (1,1,0,0), (1,1,1,0), (0,1,0,0), (0,1,1,0), (0,0,0,1)\}$\nobreak\,${\sim}(71)$
\item $\{(1,0,0,0), (1,1,0,0), (1,1,1,0), (0,1,0,0), (0,1,1,1), (0,0,1,0)\}$\nobreak\,${\sim}(73)$
\item $\{(1,0,0,0), (1,1,0,0), (1,1,1,0), (0,1,0,0), (0,1,1,1), (0,0,1,1)\}$\nobreak\,${\sim}(74)$
\item $\{(1,0,0,0), (1,1,0,0), (1,1,1,0), (0,1,0,0), (0,0,1,0), (0,0,1,1)\}$\nobreak\,${\sim}(73)$
\item $\{(1,0,0,0), (1,1,0,0), (1,1,1,0), (0,1,0,0), (0,0,1,0), (0,0,0,1)\}$\nobreak\,${\sim}(73)$
\item $\{(1,0,0,0), (1,1,0,0), (1,1,1,0), (0,1,1,0), (0,1,1,1), (0,0,0,1)\}$\nobreak\,${\sim}(78)$
\item $\{(1,0,0,0), (1,1,0,0), (1,1,1,0), (0,0,1,0), (0,0,1,1), (0,0,0,1)\}$\nobreak\,${\sim}(81)$
\item $\{(1,0,0,0), (1,1,0,0), (1,1,1,1), (0,1,0,0), (0,1,1,0), (0,1,1,1)\}$\nobreak\,${\sim}(72)$
\item $\{(1,0,0,0), (1,1,0,0), (1,1,1,1), (0,1,0,0), (0,1,1,0), (0,0,1,0)\}$\nobreak\,${\sim}(73)$
\item $\{(1,0,0,0), (1,1,0,0), (1,1,1,1), (0,1,0,0), (0,1,1,0), (0,0,1,1)\}$\nobreak\,${\sim}(74)$
\item $\{(1,0,0,0), (1,1,0,0), (1,1,1,1), (0,1,0,0), (0,1,1,1), (0,0,1,0)\}$\nobreak\,${\sim}(72)$
\item $\{(1,0,0,0), (1,1,0,0), (1,1,1,1), (0,1,0,0), (0,0,1,0), (0,0,1,1)\}$\nobreak\,${\sim}(74)$
\item $\{(1,0,0,0), (1,1,0,0), (0,1,0,0), (0,1,1,0), (0,1,1,1), (0,0,1,0)\}$\nobreak\,${\sim}(73)$
\item $\{(1,0,0,0), (1,1,0,0), (0,1,0,0), (0,1,1,0), (0,1,1,1), (0,0,1,1)\}$\nobreak\,${\sim}(74)$
\item $\{(1,0,0,0), (1,1,0,0), (0,1,0,0), (0,1,1,0), (0,1,1,1), (0,0,0,1)\}$\nobreak\,${\sim}(75)$
\item $\{(1,0,0,0), (1,1,0,0), (0,1,0,0), (0,1,1,0), (0,0,1,0), (0,0,1,1)\}$\nobreak\,${\sim}(73)$
\item $\{(1,0,0,0), (1,1,0,0), (0,1,0,0), (0,1,1,0), (0,0,1,0), (0,0,0,1)\}$\nobreak\,${\sim}(73)$
\item $\{(1,0,0,0), (1,1,0,0), (0,1,0,0), (0,1,1,1), (0,0,1,0), (0,0,1,1)\}$\nobreak\,${\sim}(74)$
\item $\{(1,0,0,0), (1,1,0,0), (0,1,0,0), (0,0,1,0), (0,0,1,1), (0,0,0,1)\}$\nobreak\,${\sim}(75)$
\item $\{(1,0,0,0), (1,1,1,0), (1,1,1,1), (0,1,0,0), (0,1,1,0), (0,1,1,1)\}$\nobreak\,${\sim}(76)$
\item $\{(1,0,0,0), (1,1,1,0), (1,1,1,1), (0,1,0,0), (0,1,1,0), (0,0,0,1)\}$\nobreak\,${\sim}(78)$
\item $\{(1,0,0,0), (1,1,1,0), (0,1,0,0), (0,1,1,0), (0,1,1,1), (0,0,1,1)\}$\nobreak\,${\sim}(77)$
\item $\{(1,0,0,0), (1,1,1,0), (0,1,0,0), (0,1,1,0), (0,1,1,1), (0,0,0,1)\}$\nobreak\,${\sim}(78)$
\item $\{(1,0,0,0), (1,1,1,1), (0,1,0,0), (0,1,1,0), (0,1,1,1), (0,0,1,0)\}$\nobreak\,${\sim}(79)$
\item $\{(1,0,0,0), (0,1,0,0), (0,1,1,0), (0,1,1,1), (0,0,1,0), (0,0,1,1)\}$\nobreak\,${\sim}(80)$
\item $\{(1,0,0,0), (0,1,0,0), (0,1,1,0), (0,1,1,1), (0,0,1,0), (0,0,0,1)\}$\nobreak\,${\sim}(81)$
\item $\{(1,0,0,0), (0,1,0,0), (0,1,1,0), (0,0,1,0), (0,0,1,1), (0,0,0,1)\}$\nobreak\,${\sim}(81)$
\item $\{(1,0,0,0), (1,1,0,0), (1,1,1,0), (1,1,1,1), (0,1,0,0), (0,1,1,0), (0,1,1,1)\}$
\item $\{(1,0,0,0), (1,1,0,0), (1,1,1,0), (1,1,1,1), (0,1,0,0), (0,1,1,0), (0,0,1,0)\}$
\item $\{(1,0,0,0), (1,1,0,0), (1,1,1,0), (1,1,1,1), (0,1,0,0), (0,1,1,0), (0,0,1,1)\}$
\item $\{(1,0,0,0), (1,1,0,0), (1,1,1,0), (1,1,1,1), (0,1,0,0), (0,1,1,0), (0,0,0,1)\}$
\item $\{(1,0,0,0), (1,1,0,0), (1,1,1,0), (1,1,1,1), (0,1,0,0), (0,1,1,1), (0,0,1,0)\}$
\item $\{(1,0,0,0), (1,1,0,0), (1,1,1,0), (1,1,1,1), (0,1,0,0), (0,1,1,1), (0,0,1,1)\}$
\item $\{(1,0,0,0), (1,1,0,0), (1,1,1,0), (1,1,1,1), (0,1,0,0), (0,0,1,0), (0,0,1,1)\}$
\item $\{(1,0,0,0), (1,1,0,0), (1,1,1,0), (1,1,1,1), (0,1,0,0), (0,0,1,0), (0,0,0,1)\}$
\item $\{(1,0,0,0), (1,1,0,0), (1,1,1,0), (1,1,1,1), (0,1,1,0), (0,1,1,1), (0,0,0,1)\}$
\item $\{(1,0,0,0), (1,1,0,0), (1,1,1,0), (1,1,1,1), (0,0,1,0), (0,0,1,1), (0,0,0,1)\}$
\item $\{(1,0,0,0), (1,1,0,0), (1,1,1,0), (0,1,0,0), (0,1,1,0), (0,1,1,1), (0,0,1,0)\}$\nobreak\,${\sim}(112)$
\item $\{(1,0,0,0), (1,1,0,0), (1,1,1,0), (0,1,0,0), (0,1,1,0), (0,1,1,1), (0,0,1,1)\}$\nobreak\,${\sim}(113)$
\item $\{(1,0,0,0), (1,1,0,0), (1,1,1,0), (0,1,0,0), (0,1,1,0), (0,1,1,1), (0,0,0,1)\}$\nobreak\,${\sim}(114)$
\item $\{(1,0,0,0), (1,1,0,0), (1,1,1,0), (0,1,0,0), (0,1,1,0), (0,0,1,0), (0,0,1,1)\}$\nobreak\,${\sim}(112)$
\item $\{(1,0,0,0), (1,1,0,0), (1,1,1,0), (0,1,0,0), (0,1,1,0), (0,0,1,0), (0,0,0,1)\}$\nobreak\,${\sim}(112)$
\item $\{(1,0,0,0), (1,1,0,0), (1,1,1,0), (0,1,0,0), (0,1,1,1), (0,0,1,0), (0,0,1,1)\}$
\item $\{(1,0,0,0), (1,1,0,0), (1,1,1,0), (0,1,0,0), (0,0,1,0), (0,0,1,1), (0,0,0,1)\}$\nobreak\,${\sim}(118)$
\item $\{(1,0,0,0), (1,1,0,0), (1,1,1,1), (0,1,0,0), (0,1,1,0), (0,1,1,1), (0,0,1,0)\}$\nobreak\,${\sim}(115)$
\item $\{(1,0,0,0), (1,1,0,0), (1,1,1,1), (0,1,0,0), (0,1,1,0), (0,1,1,1), (0,0,1,1)\}$\nobreak\,${\sim}(116)$
\item $\{(1,0,0,0), (1,1,0,0), (1,1,1,1), (0,1,0,0), (0,1,1,0), (0,0,1,0), (0,0,1,1)\}$\nobreak\,${\sim}(126)$
\item $\{(1,0,0,0), (1,1,0,0), (1,1,1,1), (0,1,0,0), (0,1,1,1), (0,0,1,0), (0,0,1,1)\}$\nobreak\,${\sim}(116)$
\item $\{(1,0,0,0), (1,1,0,0), (0,1,0,0), (0,1,1,0), (0,1,1,1), (0,0,1,0), (0,0,1,1)\}$\nobreak\,${\sim}(117)$
\item $\{(1,0,0,0), (1,1,0,0), (0,1,0,0), (0,1,1,0), (0,1,1,1), (0,0,1,0), (0,0,0,1)\}$\nobreak\,${\sim}(118)$
\item $\{(1,0,0,0), (1,1,0,0), (0,1,0,0), (0,1,1,0), (0,0,1,0), (0,0,1,1), (0,0,0,1)\}$\nobreak\,${\sim}(118)$
\item $\{(1,0,0,0), (1,1,1,0), (1,1,1,1), (0,1,0,0), (0,1,1,0), (0,1,1,1), (0,0,0,1)\}$\nobreak\,${\sim}(119)$
\item $\{(1,0,0,0), (0,1,0,0), (0,1,1,0), (0,1,1,1), (0,0,1,0), (0,0,1,1), (0,0,0,1)\}$\nobreak\,${\sim}(120)$
\item $\{(1,0,0,0), (1,1,0,0), (1,1,1,0), (1,1,1,1), (0,1,0,0), (0,1,1,0), (0,1,1,1), (0,0,1,0)\}$
\item $\{(1,0,0,0), (1,1,0,0), (1,1,1,0), (1,1,1,1), (0,1,0,0), (0,1,1,0), (0,1,1,1), (0,0,1,1)\}$
\item $\{(1,0,0,0), (1,1,0,0), (1,1,1,0), (1,1,1,1), (0,1,0,0), (0,1,1,0), (0,1,1,1), (0,0,0,1)\}$
\item $\{(1,0,0,0), (1,1,0,0), (1,1,1,0), (1,1,1,1), (0,1,0,0), (0,1,1,0), (0,0,1,0), (0,0,1,1)\}$
\item $\{(1,0,0,0), (1,1,0,0), (1,1,1,0), (1,1,1,1), (0,1,0,0), (0,1,1,0), (0,0,1,0), (0,0,0,1)\}$
\item $\{(1,0,0,0), (1,1,0,0), (1,1,1,0), (1,1,1,1), (0,1,0,0), (0,1,1,1), (0,0,1,0), (0,0,1,1)\}$
\item $\{(1,0,0,0), (1,1,0,0), (1,1,1,0), (1,1,1,1), (0,1,0,0), (0,0,1,0), (0,0,1,1), (0,0,0,1)\}$
\item $\{(1,0,0,0), (1,1,0,0), (1,1,1,0), (0,1,0,0), (0,1,1,0), (0,1,1,1), (0,0,1,0), (0,0,1,1)\}$\nobreak\,${\sim}(140)$
\item $\{(1,0,0,0), (1,1,0,0), (1,1,1,0), (0,1,0,0), (0,1,1,0), (0,1,1,1), (0,0,1,0), (0,0,0,1)\}$\nobreak\,${\sim}(141)$
\item $\{(1,0,0,0), (1,1,0,0), (1,1,1,0), (0,1,0,0), (0,1,1,0), (0,0,1,0), (0,0,1,1), (0,0,0,1)\}$
\item $\{(1,0,0,0), (1,1,0,0), (1,1,1,1), (0,1,0,0), (0,1,1,0), (0,1,1,1), (0,0,1,0), (0,0,1,1)\}$\nobreak\,${\sim}(142)$
\item $\{(1,0,0,0), (1,1,0,0), (0,1,0,0), (0,1,1,0), (0,1,1,1), (0,0,1,0), (0,0,1,1), (0,0,0,1)\}$\nobreak\,${\sim}(143)$
\item $\{(1,0,0,0), (1,1,0,0), (1,1,1,0), (1,1,1,1), (0,1,0,0), (0,1,1,0), (0,1,1,1), (0,0,1,0), (0,0,1,1)\}$
\item $\{(1,0,0,0), (1,1,0,0), (1,1,1,0), (1,1,1,1), (0,1,0,0), (0,1,1,0), (0,1,1,1), (0,0,1,0), (0,0,0,1)\}$
\item $\{(1,0,0,0), (1,1,0,0), (1,1,1,0), (1,1,1,1), (0,1,0,0), (0,1,1,0), (0,0,1,0), (0,0,1,1), (0,0,0,1)\}$
\item $\{(1,0,0,0), (1,1,0,0), (1,1,1,0), (0,1,0,0), (0,1,1,0), (0,1,1,1), (0,0,1,0), (0,0,1,1), (0,0,0,1)\}$\nobreak\,${\sim}(151)$
\item $\Delta_+(A_4)$
\end{enumerate}
}
\medskip

For $S\subset\Delta_+(A_n)$, positive roots are intervals $[r,j]=\alpha_r+\cdots+\alpha_j$, and condition~\ref{eq:closure1} becomes the interval-closure rule
\[
 [r,j],[k,j]\in S,\quad r<k
 \qquad\Rightarrow\qquad
 [r,k-1]\in S.
\]
For~$n=1,\ldots,6$ there are~$1$, $3$, $17$, $153$, $2005$, $36435$ such subsets of full rank (computed with Python). Up to equivalence, i.e., under~$\mathrm{T}_n(\Z)$, they form~$1$, $2$, $7$, $41$, $390$, $5841$ classes, and under the subgroup~$\mathrm{UT}_n(\Z)$ there are~$1$, $2$, $8$, $58$, $682$, $11984$ classes. We did not find these sequences in the literature, nor in the OEIS.

\section{Explicit computation for~\texorpdfstring{$A_2$}{A2}}\label{sec:appendixA2}

In this part we present a detailed analysis of the proof of Theorem~\ref{thm:main1} and Corollary~\ref{cor:polyinsertion}.

\subsection{Setup}
In the case of~$A_2$, we have
\[
S\=\{(1,0),(0,1),(1,1)\},
\qquad
S_1\=\{(1,0)\},
\]
so that $S\setminus S_1=\{(0,1),(1,1)\}$. We choose
\[
p(x,y,z)\df (x^2-1)(y^2-1)(z^2-1),
\]
so that $p=4^3 \,p_{\vec k}= 64\,p_{\vec k}$ with~$k_{\vec v}=1$ for~$\vec v\in S$. Writing
\[
A\df X(m),\qquad B\df X(n),\qquad C\df X(m+n),
\]
the reduction in the first part of the proof of Theorem~\ref{thm:main1} becomes simply
\[
p(A,B,C)
\=
B\,p_{(0,1),1}(A)
+C\,p_{(1,1),1}(A)
+(B^2-1)\,p_{(0,1),2}(A)
+(C^2-1)\,p_{(1,1),2}(A),
\]
where
\[
\begin{aligned}
p_{(0,1),1}(A)&\=-2A(A^2-1)^2, & p_{(1,1),1}(A)&\=2A(A^2-1)^2,\\
p_{(0,1),2}(A)&\=(A^2-1)^2, & p_{(1,1),2}(A)&\=(A^2-1)^2,
\end{aligned}
\]
and therefore
\begin{equation}\label{eq:A2-decomposition}
p(A,B,C)
\=
(A^2-1)^2
\bigl((B^2-1)+(C^2-1)-2A(B-C)\bigr).
\end{equation}
Introduce the unary Fourier series
\begin{align*}
J(z)&\df \!\sum_{m\in\Z\setminus\{0\}}\!\zeta^m\big(X(m)^2-1\big), &
F(z)&\df \!\sum_{m\in\Z\setminus\{0\}}\!\zeta^m\big(X(m)^2-1\big)^2,\\
K(z)&\df \!\sum_{m\in\Z\setminus\{0\}}\!\zeta^mX(m)\big(X(m)^2-1\big)^2, &
L(z)&\df \!\sum_{m\in\Z\setminus\{0\}}\!\zeta^mX(m)^2\big(X(m)^2-1\big)^2,\\
M(z)&\df \!\sum_{m\in\Z\setminus\{0\}}\!\zeta^m\big(X(m)^2-1\big)^3, &
H(\zeta)&\df \!\sum_{m\in\Z\setminus\{0\}}\!\zeta^mX(m).
\end{align*}
Since $X^2-1=4Y$, in the notation of Lemma~\ref{lem:XYjacobi} we have $J=4\mathcal Y_1$, $F=16\mathcal Y_2$, $K=16\mathcal X_2$, $M=64\mathcal Y_3$ and $L=F+M$, and~$H$ is the series $H(\zeta;q)$ of Lemma~\ref{lem:regularization}. The unary series $J,F,K,L,M$ converge absolutely and locally uniformly in the annulus $|q|<|\zeta|<|q|^{-1}$, equivalently, $-\Im\tau<\Im z<\Im\tau$. The series~$H$, however, is only a formal series.

\subsection{Reduction to quasi-Jacobi form}\label{sec:A2reduction}
Consider
\[
\Psi_{S,p}(\zeta_1,\zeta_2)
\=
\!\!\!\!\sum_{\substack{m,n\in\Z\\ mn(m+n)\neq0}}\!\!\!\!
\zeta_1^m\zeta_2^n\, p(A,B,C).
\]
It is easy to see that this series converges absolutely and locally uniformly on
\[
\mathcal U
\df
\bigl\{
(z_1,z_2,\tau)\in\C^2\times\H:
|\Im z_1|,|\Im z_2|<\tfrac12\Im\tau
\bigr\}.
\]

We evaluate the four contributions arising from~\eqref{eq:A2-decomposition}, using the change of variables and the notation introduced in the proof of Theorem~\ref{thm:main1}.

\emph{Case 1: $\vec w=(0,1)$, $h=2$.}
Here $X_{\vec w}^{[2]}=B^2-1$ and $p_{\vec w,2}(A)=(A^2-1)^2$, and we obtain
\begin{equation}\label{eq:A2-reduction-B2}
\Psi_{S,(B^2-1)(A^2-1)^2}(\zeta_1,\zeta_2)
\=
J(z_2)F(z_1)-M(z_1-z_2).
\end{equation}

\emph{Case 2: $\vec w=(1,1)$, $h=2$.}
Here $X_{\vec w}^{[2]}=C^2-1$ and $p_{\vec w,2}(A)=(A^2-1)^2$, and we obtain
\begin{equation}\label{eq:A2-reduction-C2}
\Psi_{S,(C^2-1)(A^2-1)^2}(\zeta_1,\zeta_2)
\=
J(z_2)F(z_1-z_2)-M(z_1).
\end{equation}

\emph{Case 3: $\vec w=(0,1)$, $h=1$.}
Here $X_{\vec w}^{[1]}=B$ and $p_{\vec w,1}(A)=-2A(A^2-1)^2$, and this contribution gives
\begin{equation}\label{eq:A2-reduction-B1}
\Psi_{S,-2BA(A^2-1)^2}(\zeta_1,\zeta_2)
\=
-2H(\zeta_2)K(z_1)-2L(z_1-z_2).
\end{equation}

\emph{Case 4: $\vec w=(1,1)$, $h=1$.}
Here $X_{\vec w}^{[1]}=C$ and $p_{\vec w,1}(A)=2A(A^2-1)^2$, and we obtain
\begin{equation}\label{eq:A2-reduction-C1}
\Psi_{S,2CA(A^2-1)^2}(\zeta_1,\zeta_2)
\=
2H(\zeta_2)K(z_1-z_2)-2L(z_1).
\end{equation}

Adding these four contributions and using $L=F+M$ gives the compact formula
\begin{equation}\label{eq:A2-explicit-final}
\begin{aligned}
\Psi_{S,p}(\zeta_1,\zeta_2)
\={}&
\bigl(J(z_2)-2\bigr)
\bigl(F(z_1)+F(z_1-z_2)\bigr)
+
2H(\zeta_2)\bigl(K(z_1-z_2)-K(z_1)\bigr)
\\
&-
3\bigl(M(z_1)+M(z_1-z_2)\bigr)
\end{aligned}
\end{equation}
of formal series in~$\zeta_1$ and~$\zeta_2$. As in the proof of Theorem~\ref{thm:main1}, since $K(z_1-z_2)-K(z_1)$ vanishes at~$\zeta_2=1$, Lemma~\ref{lem:regularization} allows us to replace~$H(\zeta_2)$ by $-2\,\iota_{\mathfrak c}\mathcal A(z_2)$. We work in the nonempty chamber
\[
\mathcal C
\df
\left\{
(z_1,z_2,\tau)\in\C^2\times\H:
0<\Im z_2<\Im z_1<
\tfrac12\Im\tau
\right\}\:\subset\:\mathcal U.
\]
All convergent unary series appearing above converge on compact subsets of~$\mathcal C$. Consequently, on~$\mathcal C$ the original lattice sum agrees with the Fourier expansion of a quasi-Jacobi form:
\begin{equation}\label{eq:A2-meromorphic-expression}
\begin{aligned}
\Psi_{S,p}(\zeta_1,\zeta_2)
\={}&
\Phi_{S,p}(z_1,z_2)
\\
\df{}&
\bigl(J(z_2)-2\bigr)
\bigl(F(z_1)+F(z_1-z_2)\bigr)
-
4\mathcal A(z_2)
\bigl(K(z_1-z_2)-K(z_1)\bigr)
\\
&-
3\bigl(M(z_1)+M(z_1-z_2)\bigr).
\end{aligned}
\end{equation}
Here we used the quasi-Jacobi forms~$\mathcal Y_r$ and~$\mathcal X_r$ of Lemma~\ref{lem:XYjacobi} for~$J,F,K,M$. As at the end of the proof of Theorem~\ref{thm:main1}, the identity extends from~$\mathcal C$ to~$\mathcal U$.

\subsection{Quasimodularity}\label{sec:A2quasimod}
We use the notation of Corollary~\ref{cor:polyinsertion} with~$k_{\vec v}=1$ for~$\vec v\in S$, and for~$a,b\in\N_0$ we write $\Psi_{a,b}\df\Psi_{S,(a,b),\vec k}$, that is,
\begin{equation}\label{eq:A2-general-moments}
\Psi_{a,b}
\=
\!\!\!\!\sum_{\substack{m,n\in\Z\\mn(m+n)\neq0}}\!\!\!\!
m^a n^b\,
\frac{q^m}{(1-q^m)^2}\,
\frac{q^n}{(1-q^n)^2}\,
\frac{q^{m+n}}{(1-q^{m+n})^2}.
\end{equation}
Since $\Psi_{S,p}$ converges absolutely and locally uniformly on~$\mathcal U$, we may differentiate termwise, and with the operator $D^{(a,b)}\df D^{\bm\ell}$ with $\bm\ell=(a,b)$ from the proof of Corollary~\ref{cor:polyinsertion} we obtain
\begin{equation}\label{eq:A2-moments-derivatives}
64\,\Psi_{a,b}
\=
D^{(a,b)}\,\Phi_{S,p}(z_1,z_2)\big|_{z_1=z_2=0}.
\end{equation}
We now evaluate these derivatives explicitly. For~$r\geq0$ let
\[
J_r\df D^rJ(0),\qquad
F_r\df D^rF(0),\qquad
K_r\df D^rK(0),\qquad
M_r\df D^rM(0),
\]
where $D\df\frac{1}{2\pi i}\frac{\partial}{\partial z}$. Since~$J$, $F$ and~$M$ are even and~$K$ is odd,
\[
J_{2r+1}\=F_{2r+1}\=M_{2r+1}\=0,
\qquad
K_{2r}\=0.
\]
For~$r\geq0$ we have
\begin{equation}\label{eq:A2-unary-moments}
\begin{aligned}
J_{2r}
&\=
8\sum_{m\geq1}
m^{2r}\frac{q^m}{(1-q^m)^2}, &
F_{2r}
&\=
32\sum_{m\geq1}
m^{2r}\frac{q^{2m}}{(1-q^m)^4},
\\
K_{2r+1}
&\=
32\sum_{m\geq1}
m^{2r+1}
\frac{q^{2m}(1+q^m)}{(1-q^m)^5}, &
M_{2r}
&\=
128\sum_{m\geq1}
m^{2r}\frac{q^{3m}}{(1-q^m)^6};
\end{aligned}
\end{equation}
these series all belong to the ring of quasimodular forms on~$\SL_2(\Z)$.

We first consider~$b=0$, where the computation is simplest. Since~$K$ is holomorphic,
\[
K(z_1-z_2)-K(z_1)
\=
-z_2\,\partial_{z_1}K(z_1)+O(z_2^2),
\qquad
\mathcal A(z_2)
\=
\frac{1}{2\pi iz_2}+O(z_2),
\]
and therefore
\[
\lim_{z_2\to0}
\mathcal A(z_2)
\bigl(K(z_1-z_2)-K(z_1)\bigr)
\=
-DK(z_1).
\]
Thus
\[
\Phi_{S,p}(z_1,0)
\=
2(J_0-2)F(z_1)
+
4DK(z_1)
-
6M(z_1).
\]
Applying~$D^a$ and setting~$z_1=0$, we find
\begin{equation}\label{eq:A2-single-insertion}
64\,\Psi_{a,0}
\=
2(J_0-2)F_a
+
4K_{a+1}
-
6M_a.
\end{equation}
Using $2K_{2r+1}=q\frac{d}{dq}F_{2r}$, we obtain for~$a=2r$
\begin{equation}\label{eq:A2-even-insertion}
\Psi_{2r,0}
\=
\frac1{64}
\left(
2(J_0-2)F_{2r}
+
2q\frac{d}{dq}F_{2r}
-
6M_{2r}
\right).
\end{equation}
If~$a$ is odd, all three terms in~\eqref{eq:A2-single-insertion} vanish, so that $\Psi_{2r+1,0}=0$. Interchanging~$m$ and~$n$ gives $\Psi_{0,a}=\Psi_{a,0}$.

We now treat arbitrary $a,b\geq0$. The function~$\mathcal A$ has the Laurent expansion around~$z=0$
\begin{equation}\label{eq:A2-A-Laurent}
\mathcal A(z)
\=
\frac{1}{2\pi iz}+
\sum_{j\geq1}
\frac{B_{2j}}{(2j)!}\,
E_{2j}\,(2\pi iz)^{2j-1},
\end{equation}
where~$B_{2j}$ denotes the Bernoulli numbers. On the other hand, Taylor expansion gives
\[
K(z_1-z_2)-K(z_1)
\=
\sum_{r\geq1}
\frac{(-2\pi iz_2)^r}{r!}\,
D^rK(z_1).
\]
Multiplying by~\eqref{eq:A2-A-Laurent} with~$z=z_2$, we obtain a power series in~$z_2$. In particular,
\begin{equation}\label{eq:A2-general-polar-derivative}
\begin{aligned}
&D^{(a,b)}
\Bigl[
\mathcal A(z_2)
\bigl(K(z_1-z_2)-K(z_1)\bigr)
\Bigr]
\Big|_{z_1=z_2=0}
\\
&\quad\=
(-1)^{b+1}
\left(
\frac{K_{a+b+1}}{b+1}
+
b!\sum_{j=1}^{\lfloor b/2\rfloor}
\frac{B_{2j}E_{2j}}
{(2j)!\,(b-2j+1)!}\,
K_{a+b-2j+1}
\right).
\end{aligned}
\end{equation}
For the remaining terms of~$\Phi_{S,p}$, Leibniz's rule gives
\[
D^{(a,b)}
\Bigl[
(J(z_2)-2)
(F(z_1)+F(z_1-z_2))
\Bigr]
\Big|_{\bm z=\bm 0}
\=
(J_b-2\delta_{b,0})F_a+
\sum_{h=0}^b
\binom bh(-1)^{b-h}
(J_h-2\delta_{h,0})F_{a+b-h}
\]
and
\[
D^{(a,b)}
\bigl[
M(z_1)+M(z_1-z_2)
\bigr]
\big|_{\bm z=\bm 0}
\=
\delta_{b,0}M_a+(-1)^bM_{a+b}.
\]
Combining these identities yields the following explicit reduction formula:
\begin{equation}\label{eq:A2-all-polynomial-insertions}
\begin{aligned}
64\,\Psi_{a,b}
\={}&
(J_b-2\delta_{b,0})F_a
+
\sum_{h=0}^b
\binom bh(-1)^{b-h}
(J_h-2\delta_{h,0})F_{a+b-h}
-
3\bigl(
\delta_{b,0}M_a+(-1)^bM_{a+b}
\bigr)
\\
&+
4(-1)^b
\left(
\frac{K_{a+b+1}}{b+1}
+
b!\sum_{j=1}^{\lfloor b/2\rfloor}
\frac{B_{2j}E_{2j}}
{(2j)!\,(b-2j+1)!}\,
K_{a+b-2j+1}
\right).
\end{aligned}
\end{equation}
We conclude that~$\Psi_{a,b}$ is a quasimodular form on~$\SL_2(\Z)$ of mixed weight at most~$6+a+b$, as claimed in Corollary~\ref{cor:polyinsertion}.

\end{document}